\documentclass[12pt]{amsart}
\usepackage{graphicx, amssymb, amsthm, amsfonts, latexsym, epsfig, amsmath, amscd, amsbsy}
\usepackage[usenames]{color}

\theoremstyle{definition}
\newtheorem{theorem}{Theorem}[section]
\newtheorem{lemma}[theorem]{Lemma}
\newtheorem{corollary}[theorem]{Corollary}
\newtheorem{proposition}[theorem]{Proposition}

\newtheorem{defi}[theorem]{Definition}
\newtheorem{remark}[theorem]{Remark}
\newtheorem{ex}[theorem]{Example}
\newtheorem{maintheorem}{Main Theorem}

\numberwithin{equation}{section}

\def\ie{{\em i.e.,} }
\def\eg{{\em e.g.} }

\newfont\bbf{msbm10 at 12pt}
\def\eps{\varepsilon}
\def\phi{\varphi}
\def\R{{\mathbb R}}
\def\C{{\mathbb C}}
\def\None{{\mathbb N}}
\def\N0{{\mathbb N}_0}
\def\Z{{\mathbb Z}}

\def\C0{{\mathfrak C}}
\def\Rr{{\mathfrak R}}
\def\A{{\mathfrak A}}
\def\Dlev{\mathfrak D}

\def\Cl{\mathop\mathrm{Cl}}
\def\mesh{\mathop\mathrm{mesh}}
\def\nin{\not\in}
\def\wasv{n}

\def\diam{\mbox{\rm diam}\,}
\def\orb{\mbox{\rm orb}}
\def\theta{\vartheta}
\def\le{\leqslant}
\def\ge{\geqslant}

\def\IL{\underleftarrow\lim([0,s/2],T_s)}
\def\ILp{\underleftarrow\lim([0,s'/2],T_{s'})}
\def\ILcores{\underleftarrow\lim([c_2, c_1],T_{s})}
\def\ILcoresp{\underleftarrow\lim([c_2, c_1],T_{s'})}

\def\chain{{\mathcal C}}

\newcommand{\bd}{\partial}

\begin{document}

\title[Fibonacci-like unimodal inverse limits
and the core Ingram conjecture]
{Fibonacci-like unimodal inverse limit spaces
and the core Ingram conjecture}
\author{H.~Bruin \quad and \quad S.~\v{S}timac}
\thanks{HB was supported by EPSRC grant EP/F037112/1.
S\v{S} was supported in part by the MZOS Grant
037-0372791-2802 of the Republic of Croatia.}


\begin{abstract}
We study the structure of inverse limit space of so-called Fibonacci-like tent maps.
The combinatorial constraints implied by the Fibonacci-like assumption allow us to
introduce certain chains that enable a more detailed analysis of symmetric arcs within
this space than is possible in the general case. We show that link-symmetric arcs are
always symmetric or a well-understood concatenation of quasi-symmetric arcs. This leads to
the proof of the Ingram Conjecture for cores of Fibonacci-like unimodal inverse limits.
\end{abstract}

\maketitle

\baselineskip=18pt

\section{Introduction}\label{sec:intro}

A unimodal map is called Fibonacci-like if it satisfies certain combinatorial conditions
implying an extreme recurrence behavior of the critical point.
The Fibonacci unimodal map itself was first described by Hofbauer and Keller \cite{hofkel0} as a candidate
to have a so-called wild attractor.
(The combinatorial property defining the Fibonacci unimodal map is
that its so-called {\em cutting times} are exactly the Fibonacci numbers $1, 2, 3 ,5, 8, \dots$)
In \cite{BKNS} it was indeed shown that Fibonacci unimodal maps with sufficiently large critical
order possess a wild attractor, whereas Lyubich \cite{Lyu} showed that such is not the
case if the critical order is $2$ (or $\leq 2+\eps$ as was shown in \cite{KN}).
This answered a question in Milnor's well-known paper on the structure of metric attracts \cite{Mil}.
In \cite{BTams} the strict Fibonacci combinatorics were relaxed to Fibonacci-like.
Intricate number-theoretic properties of Fibonacci-like critical omega-limit sets were
revealed in \cite{LM} and \cite{BKS}, and \cite[Theorem 2]{BNonlin} shows that
Fibonacci-like combinatorics are incompatible with the Collet-Eckmann condition of exponential derivative
growth along the critical orbit.
This underlines that Fibonacci-like maps are an extremely interesting class of maps in
between the regular and the stochastic unimodal maps in the classification of \cite{ALM}.

One of the reasons for studying the inverse limit spaces of  Fibonacci-like unimodal maps
is that they present a toy model of invertible strange attractors (such as H\'enon attractors)
for which as of today very little is known beyond the Benedicks-Carleson parameters \cite{BC}
resulting in strange attractors with positive unstable Lyapunov exponent.
It is for example unknown if invertible wild attractors exist in the smooth planar context, or
to what extent H\'enon-like attractors satisfy Collet-Eckmann-like growth conditions.
The precise recurrence and folding structure of  H\'enon-like attractors
may be of crucial importance to answer such questions, and we therefore focus on these aspects
of the structure of Fibonacci-like inverse limit spaces.

A second reason for this paper is to provide a better understanding and the solution of the Ingram Conjecture
for cores of Fibonacci-like inverse limit spaces.
The original conjecture was posed by Tom Ingram in 1991 for tent maps $T_s : [0, 1] \to [0, 1]$ with slope
$\pm s$, $s \in [1, 2]$, defined as $T_s(x) = \min\{sx, s(1-x)\}$:
\begin{quote}
If $1 \leq s < s' \leq 2$, then the corresponding inverse limit spaces $\IL$ and $\ILp$ are non-homeomorphic.
\end{quote}
The first results towards solving this conjecture have been obtained for tent maps with a finite critical
orbit \cite{Kail2,Stim,Betal}. Raines and \v{S}timac \cite{RS} extended these results to tent maps with
an infinite, but non-recurrent critical orbit. Recently Ingram's Conjecture was solved for all slopes $s \in [1,2]$ (in
the affirmative) by Barge, Bruin and \v Stimac in \cite{BBS}, but we still know very little of the structure of inverse limit spaces (and
their subcontinua) for the case that $\orb(c)$ is infinite and recurrent, see \cite{BBD, BB, Bsubcontinua}.
Also, the arc-component $\C0$ of $\IL$ containing the endpoint $\bar 0 := (\dots,0,0,0)$ is important in
the proof of the Ingram Conjecture in \cite{BBS},
leaving open the ``core'' version of the
Ingram Conjecture. It is this version that we solve here for Fibonacci-like tent maps:

\begin{maintheorem}\label{mainthm}
If $1 \leq s < s' \leq 2$ are the parameters of Fibonacci-like tent-maps,
then the corresponding cores of inverse limit spaces $\ILcores$ and $\ILcoresp$ are
non-homeomorphic.
\end{maintheorem}

The core version of the Ingram Conjecture was proved already
in the postcritically finite case, since neither Kailhofer \cite{Kail2}, nor \v{S}timac
\cite{Stim} use the arc-component $\C0$, but work on some other arc-components of the core
(although not on the arc-component of the fixed point $(\dots , r, r, r)$
which we use in this paper).

The key observation in our proof is Proposition~\ref{prop:symmetric} which
implies that every homeomorphism $h$ maps symmetric arcs to symmetric arcs,
not just to quasi-symmetric arcs.
(The difficulty that quasi-symmetric arcs pose was first observed and overcome in \cite{RS} in the setting of tent maps with non-recurrent critical point.)
To prove Proposition~\ref{prop:symmetric}, the special structure of the
Fibonacci-like maps, and especially the special chains it allows, is used.
But assuming the result of Proposition~\ref{prop:symmetric},
the proof of the main theorem works for general tent maps.

The paper is organized as follows. In Section~\ref{sec:def} we review the basic definitions of inverse limit
spaces and tent maps and their symbolic dynamics.
In Section~\ref{sec:homeomorphisms} we
introduce salient points, show that any homeomorphism on the core of the Fibonacci-like inverse limit space
maps salient points ``close'' to salient points, and using this
we prove our main theorem in Section~\ref{sec:maintheorems}.
Appendix~\ref{sec:chains} is devoted to the construction of
the chains $\chain$ having special properties that allow us to prove desired properties of folding structure
in Appendix~\ref{sec:furtherlemmas}. In Appendix~\ref{sec:link}, we show that link-symmetric arcs are always
symmetric or a well-understood concatenation of quasi-symmetric arcs.

\section{Preliminaries}\label{sec:def}

\subsection{Combinatorics of tent maps}
The tent map $T_s : [0,1] \to [0,1]$ with slope $\pm s$ is defined as $T_s(x) = \min\{sx, s(1-x)\}$. The
critical or turning point is $c = 1/2$ and we write $c_k = T_s^k(c)$, so in particular $c_1 = s/2$ and
$c_2 = s(1-s/2)$. We will restrict $T_s$ to the interval $I = [0, s/2]$; this is larger than the {\em core}
$[c_2, c_1] = [s-s^2/2, s/2]$, but it contains both fixed points $0$ and $r = \frac{s}{s+1}$.

Recall now some background on the combinatorics of unimodal maps,
see \eg \cite{Bknea}. The {\em cutting times}
$\{S_k\}_{k \geq 0}$ are those iterates $n$ (written in increasing order) for which the central branch
of $T_s^n$ covers $c$. More precisely, let $Z_n \subset [0,c]$ be the maximal interval with boundary
point $c$ on which $T_s^n$ is monotone, and let $\Dlev_n = T_s^n(Z_n)$. Then $n$ is a {\em cutting
time} if $\Dlev_n \ni c$.
Let $\None = \{ 1,2,3,\dots\}$ be the set of natural numbers
and $\N0 = \None \cup \{ 0 \}$.
There is a function $Q : \None \to \N0$ called the {\em kneading map}
such that
\begin{equation}\label{eq:Q}
S_k - S_{k-1} = S_{Q(k)}
\end{equation}
for all $k$. The kneading map $Q(k) = \max \{ k-2, 0 \}$ (with cutting times
$\{ S_k \}_{k \geq 0} = \{ 1,2,3,5,8,\dots\}$) belongs to the {\em Fibonacci map}. We call $T_s$
{\em Fibonacci-like} if its kneading map is eventually non-decreasing,
and satisfies Condition \eqref{eq:cond4} as well:
\begin{equation}\label{eq:cond4}
Q(k+1) > Q(Q(k)+1) \qquad \text{ for all $k$ sufficiently large.}
\end{equation}
\begin{remark}\label{rem:cond4}
Condition \eqref{eq:cond4} follows if the $Q$ is eventually non-decreasing
and $Q(k) \leq k-2$ for $k$ sufficiently large.
(In fact, since tent maps are not renormalizable of arbitrarily high period,
$Q(k) \leq k-2$ for $k$ sufficiently large follows from $Q$ being
eventually non-decreasing, see \cite[Proposition 1]{Bknea}.)
Geometrically, it means that $|c-c_{S_k}| < |c-c_{S_{Q(k)}}|$, see Lemma~\ref{lem:order}
and also \cite{Bknea}.
\end{remark}
\begin{lemma}\label{lem:order}
If the kneading map of $T_s$ satisfies \eqref{eq:cond4}, then
\begin{equation}\label{eq:order}
|c_{S_k} - c| < |c_{S_{Q(k)}} - c| \quad \text{ and } \quad |c_{S_k} - c| < \frac12 |c_{S_{Q^2(k)}} - c|.
\end{equation}
for all $k$ sufficiently large.
\end{lemma}

\begin{proof}
For each cutting time $S_k$, let $\zeta_k \in Z_{S_k}$ be the point such that
$T_s^{S_k}(\zeta_k) = c$. Then $\zeta_k$ together with its symmetric image
$\hat \zeta_k := 1-\zeta_k$
are closest precritical points in the sense that $T_s^j((\zeta_k,c)) \not\ni c$
for $0 \le j \le S_k$.
Consider the points $\zeta_{k-1}$, $\zeta_k$ and $c$, and their images under $T_s^{S_k}$, see
Figure~\ref{fig:order}.
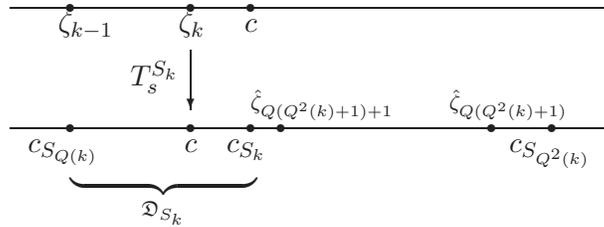
\begin{figure}[ht]
\unitlength=8mm
\begin{picture}(10,3.6)(0,0)
\put(0,1){\line(1,0){10}}
\put(1, 1){\circle*{0.15}}\put(0.3, 0.6){\small $c_{S_{Q(k)}}$}
\put(3, 1){\circle*{0.15}}\put(2.9, 0.6){\small $c$}
\put(4, 1){\circle*{0.15}}\put(3.6, 0.6){\small $c_{S_k}$}
\put(9, 1){\circle*{0.15}}\put(8.3, 0.6){\small $c_{S_{Q^2(k)}}$}
\put(4.5, 1){\circle*{0.15}}\put(4, 1.3){\tiny $\hat \zeta_{Q(Q^2(k)+1)+1}$}
\put(8, 1){\circle*{0.15}}\put(7.3, 1.3){\tiny $\hat \zeta_{Q(Q^2(k)+1)}$}
\put(1, 0.2){$\underbrace{\qquad\qquad\qquad}_{\Dlev_{S_k}}$}
\put(3,2.3){\vector(0,-1){1}} \put(2, 1.7){\small $T_s^{S_k}$}
\put(0,3){\line(1,0){10}}
\put(1,3){\circle*{0.15}}\put(0.8, 2.6){\small $\zeta_{k-1}$}
\put(3,3){\circle*{0.15}}\put(2.8, 2.6){\small $\zeta_k$}
\put(4,3){\circle*{0.15}}\put(3.9, 2.6){\small $c$}
\end{picture}
\caption{The points $\zeta_{k-1}$, $\zeta_k$ and $c$, and their images under $T_s^{S_k}$.}
\label{fig:order}
\end{figure}
Note that $Z_{S_k} = [\zeta_{k-1},c]$ and $T_s^{S_k}([\zeta_{k-1},c]) = \Dlev_{S_k} =
[c_{S_{Q(k)}}, c_{S_k}]$.
Since $S_{k+1} = S_k + S_{Q(k+1)}$ is the first cutting time after $S_k$,
the precritical point of lowest order on $[c,c_{S_k}]$ is
$\zeta_{Q(k+1)}$ or its symmetric image $\hat \zeta_{Q(k+1)}$.
Applying this to $c_{S_k}$ and $c_{Q(k)}$, and using \eqref{eq:cond4},
we find
$$
c_{S_k} \subset (\zeta_{Q(k+1)-1}, \hat \zeta_{Q(k+1)-1}) \subset
(\zeta_{Q(Q(k)+1)}, \hat \zeta_{Q(Q(k)+1)}) \subset (c_{S_{Q(k)}}, \hat c_{S_{Q(k)}}).
$$
Therefore $|c_{S_k} - c| < |c_{S_{Q(k)}} - c|$.
Since $T_s^{S_k}|_{[\zeta_{k-1},c]}$ is affine,
also the preimages $\zeta_{k-1}$ and $\zeta_k$ of
$c_{S_{Q(k)}}$ and $c$ satisfy $|\zeta_k - c| < |\zeta_{k-1} - \zeta_k|$.
Applying \eqref{eq:cond4} twice we obtain
\begin{equation}\label{eq:cond3}
Q(k+1) > Q(Q^2(k)+1)+1,
\end{equation}
for all $k$ sufficiently large.
Therefore there are at least two closest precritical points
($\hat \zeta_{Q(Q^2(k)+1)}$ and $\hat \zeta_{Q(Q^2(k)+1)+1}$ in Figure~\ref{fig:order})
between $c_{S_k}$ and $c_{S_{Q^2(k)}}$.
Therefore
\begin{equation}\label{eq:cond4twice}
|c_{S_k} - c| < |\hat \zeta_{Q(Q^2(k)+1)+1} - c|< \frac12 |\hat \zeta_{Q(Q^2(k)+1)} - c| < \frac12 |c_{S_{Q^2(k)}} - c|,
\end{equation}
proving the lemma.
\end{proof}
Not all maps $Q: \None \to \N0$ nor all sequences of cutting times
(as defined in \eqref{eq:Q}) correspond to a unimodal map.
As was shown by Hofbauer \cite{Hof1}, a kneading map $Q$ belongs to a unimodal map
(with infinitely many cutting times)
if and only if
\begin{equation}\label{eq:Hofbauer}
\{ Q(k+j) \}_{j \geq 1} \geq_{lex} \{ Q(Q^2(k)+j) \}_{j \geq 1}
\end{equation}
for all $k \geq 1$, where $\geq_{lex}$ indicates lexicographical order.
Clearly, Condition \eqref{eq:cond4} is compatible with
(and for large $k$ implies) Condition~\eqref{eq:Hofbauer}.

\begin{remark}\label{rem:Hofbauer}
The condition $\{ Q(k+j) \}_{j \geq 1} \geq_{lex} \{ Q(l+j) \}_{j \geq 1}$
is equivalent to $|c-c_{S_k}| < |c-c_{S_l}|$.
Therefore, because $c_{S_{k-1}} \in (\zeta_{Q(k)-1}, \zeta_{Q(k)})$,
we find by taking the $T_s^{S_{Q(k)}}$-images, that
$c_{S_k} \in [c_{S_{Q^2(k)}}, c]$ and \eqref{eq:Hofbauer} follows.
The other direction, namely that  \eqref{eq:Hofbauer} is sufficient for admissibility is much more involved, see \cite{Hof1,Bknea}.
\end{remark}

Let $\beta(n) = n - \sup\{ S_k < n\}$ for $n \geq 2$ and find recursively
the images of the central branch of $T_s^n$
(the levels in the Hofbauer tower, see \eg \cite{Bknea, BBbook})
as
\[
\Dlev_1 = [0,c_1] \text{ and } \Dlev_n = [c_n , c_{\beta(n)}].
\]
It is not hard to see that $\Dlev_n \subset \Dlev_{\beta(n)}$ for each $n$,
see \cite{Bknea},
and that if $J \subset [0,s/2]$ is a maximal interval
on which $T_s^n$ is monotone, then $T_s^n(J) = \Dlev_m$ for some $m \le n$.

The condition that $Q(k) \to \infty$ has consequence on the structure
of the critical orbit:

\begin{lemma}\label{lem:Qinfty}
If $Q(k) \to \infty$, then $|\Dlev_n| \to 0$ as $n \to \infty$,
$c$ is recurrent and $\omega(c)$ is a minimal Cantor set.
\end{lemma}

\begin{proof}
See \cite[Lemma 2.1.]{Btree}.  
\end{proof}

\subsection{Definitions for inverse limit spaces}
The inverse limit space $K_s = \IL$ is the collection of all backward orbits
\[
\{ x = (\dots, x_{-2}, x_{-1}, x_0) :  T_s(x_{i-1}) = x_i \in [0,s/2] \text{ for all } i \leq 0\},
\]
equipped with metric
$d(x,y) = \sum_{n \le 0} 2^n |x_n - y_n|$ and {\em induced} $($or {\em shift$)$ homeo\-morphism}
\[
\sigma(\dots, x_{-2}, x_{-1}, x_0) = (\dots, x_{-2}, x_{-1}, x_0, T_s(x_0)).
\]
Let $\pi_k : \IL \to [0,s/2]$, $\pi_k(x) = x_{-k}$ be the $k$-th projection map.
The \emph{arc-component} of
$x \in X$ is defined as the union of all arcs of $X$ containing $x$. Since $0 \in [0,s/2]$, the endpoint
$\bar 0 = (\dots, 0,0,0)$ is contained in  $\IL$, and the arc-component of $\IL$ of $\bar 0$ will be denoted as
$\C0$; it is a ray converging to, but disjoint from the core $\underleftarrow\lim([c_2, c_1],T_s)$ of the
inverse limit space. Since $r \in [c_2, c_1]$, the point $\rho = (\dots, r, r, r)$ is contained in
$\underleftarrow\lim([c_2, c_1],T_s)$. The arc-component of $\rho$ will be denoted as $\Rr$; it is a
continuous image of $\R$ and is dense in $\underleftarrow\lim([c_2, c_1],T_s)$ in both directions.

We fix $s \in (\sqrt{2}, 2]$; for these parameters $T_s$ is not renormalizable and
$\underleftarrow\lim([c_2, c_1],T_s)$ is indecomposable.

A point $x = (\dots, x_{-2}, x_{-1}, x_0) \in K_s$ is called a {\em $p$-point} if $x_{-p-l} = c$ for some
$l \in \N0$. The number $L_p(x) := l$ is called the {\em $p$-level} of $x$. In particular,
$x_0 = T_s^{p + l}(c)$. By convention, the endpoint $\bar 0 = (\dots, 0,0,0)$ of $\C0$ and the point
$\rho = (\dots, r, r, r)$ of $\Rr$ are also $p$-points and $L_p(\bar 0) = L_p(\rho) := \infty$, for every $p$.

The {\em folding pattern of the arc-component} $\C0$, denoted by $FP(\C0)$, is the sequence
$$
L_p(z^0), L_p(z^1), L_p(z^2), \dots , L_p(z^n), \dots ,
$$
where
$E_p^{\C0} = \{ z^0, z^1, z^2, \dots , z^n, \dots \}$ is the ordered set of all $p$-points of $\C0$ with
$z^0 = \bar 0$, and $p$ is any nonnegative integer. Let $q \in \None$, $q > p$, and
$E_q^{\C0} = \{ y^0, y^1, y^2, \dots , y^n, \dots \}$. Since $\sigma^{q-p}$ is an
order-preserving homeomorphism of $\C0$, it is easy to see that $\sigma^{q-p}(z^i) = y^i$ and
$L_p(z^i) = L_q(y^i)$ for every $i \in \None$. Therefore the folding pattern of $\C0$ does not depend on $p$.

The {\em folding pattern of the arc-component} $\Rr$, denoted by $FP(\Rr)$, is the sequence
\begin{equation}\label{eq:Rr}
\dots , L_p(z^{-n}), \dots , L_p(z^{-1}), L_p(z^0), L_p(z^1), \dots , L_p(z^n), \dots ,
\end{equation}
where
$E_p^{\Rr} = \{ \dots , z^{-n}, \dots , z^{-1}, z^0, z^1, \dots , z^n, \dots \}$ is the ordered set (indexed by $\Z$) of all
$p$-points of $\Rr$ with $z^0 = \rho$, and $p$ is any nonnegative integer. Since $r > 1/2$, we have
$\pi_i(\rho) > 1/2$ for every $i \in \N0$. It is easy to see that for every $i \in \N0$, there exists an arc
$A = A(i) \subset \Rr$ containing $\rho$ such that $\pi_i(A) = [c, c_1]$. Therefore two neighboring $p$-points
of $\rho$ have $p$-levels 0 and 1. From now on we assume, without loss of generality, that the ordering on
$\Rr$, \ie the parametrization of $\Rr$, is such that $L_p(z^{-1}) = 0$ and $L_p(z^1) = 1$. Let
$q \in \None$, $q > p$, and $E_q^{\Rr} = \{ \dots , y^{-n}, \dots , y^{-1}, y^0, y^1, \dots , y^n, \dots \}$
with $y^0 = \rho$.
Since $\sigma^{q-p}$ is an order-preserving (respectively, order-reversing) homeomorphism of $\Rr$ if $q-p$ is
even (respectively, odd), $\sigma^{q-p}(z^i) = y^i$ and $L_p(z^i) = L_q(y^i)$
for every $i \in \Z$. Therefore the folding pattern of $\Rr$ does not depend on $p$.

Note that every arc of $\C0$ and of $\Rr$ has only finitely many $p$-points, but an arc $A$ of the core of
$K_s$ can have infinitely many $p$-points.

We will mostly be interested in the arc-component $\Rr$, but also in some other arc-components 'topologically
similar' to $\Rr$. Therefore, unless stated otherwise, let
$\A \subset \underleftarrow\lim([c_2, c_1],T_s)$ denote an arc-component which does not contain any end-point,
such that every arc $A \subset \A$ contains finitely many $p$-points, and let $\A$ be dense in the core of
$K_s$ in both directions. Let $E_p^{\A} = (a^i)_{i \in \Z}$ denote the set of all $p$-points of $\A$,
where $a^0 = (\dots, a^0_{-2}, a^0_{-1}, a^0_0) \in \A$ is the only $p$-point of $\A$ with $a^0_{-j} \ne c$
for every $j \in \N0$, and let by convention $L_p(a^0) = \infty$ for every $p$. Also, we abbreviate
$E_p := E_p^{\A}$. The {\em $p$-folding pattern of the arc-component} $\A$, denoted by $FP_p(\A)$, is the
sequence $$\dots , L_p(a^{-n}), \dots , L_p(a^{-1}), L_p(a^0), L_p(a^1), \dots , L_p(a^n), \dots .$$ Given an
arc $A \subset \A$ with successive $p$-points $x^0, \dots , x^n$, the {\em $p$-folding pattern} of $A$ is the
sequence
$$
FP_p(A) := L_p(x^0), \dots , L_p(x^n).
$$

An arc $A$ in $\IL$ is said to {\em $p$-turn at $c_n$} if there is a $p$-point $a \in A$ such that
$a_{-(p+n)} = c$, so $L_p(a) = n$. This implies that  $\pi_p:A \to [0,s/2]$ achieves $c_n$ as a local
extremum at $a$. If $x$ and $y$ are two adjacent $p$-points on the same arc-component, then
$\pi_p([x,y]) = \Dlev_n$ for some $n$, so $\pi_p(x) = c_n$ and $\pi_p(y) = c_{\beta(n)}$ or vice versa.
Let us call $x$ and $y$ (or $\pi_p(x)$ and $\pi_p(y)$) {\em $\beta$-neighbors} in this case. Notice,
however, that there may be many post-critical points between $\pi_p(x)$ and $\pi_p(y)$. Obviously, every
$p$-point of $\C0$ and $\Rr$ has exactly two $\beta$-neighbors, except the endpoint $\bar 0$ of
$\C0$ whose $\beta$-neighbor (w.r.t.\ $p$) is by convention the first proper $p$-point in $\C0$, necessarily
with $p$-level $1$.

\subsection{Chainability and (quasi-)symmetry}

A space is {\em chainable} if there are finite open covers $\chain = \{ \ell_i\}_{i = 1}^N$,
called {\em chains}, of arbitrarily small {\em mesh} $(\mesh \chain = \max_i \diam \ell_i)$ with the
property that the {\em links} $\ell_i$ satisfy $\ell_i \cap \ell_j \neq \emptyset$ if and only if
$|i-j| \le 1$. The combinatorial properties of Fibonacci-like maps allow us to construct chains $\chain_p$
such that whenever an arc $A$ $p$-turns in $\ell \in \chain_p$, \ie enters and exits $\ell$ through the same
neighboring link, then the projections $\pi_p(x) = \pi_p(y)$ of the first and last $p$-point $x$ and $y$
of $A \cap \ell$ depend only on $\ell$ and not on $A$, see Proposition~\ref{prop:chains}.
We will work with the chains $\chain_p$ which are
the $\pi_p^{-1}$ images of chains of the interval $[0,s/2]$.

\begin{defi}\label{def:sym}
An arc $A \subset \A$ such that $\partial A = \{ u, v \}$ and $A \cap E_p = \{ x^0, \dots , x^{n} \}$ is
called \emph{$p$-symmetric}\, if $\pi_p(u) = \pi_p(v)$ and $L_p(x^i) = L_p(x^{n-i})$, for every $0 \le i \le n$.
\end{defi}
\noindent It is easy to see that if $A$ is $p$-symmetric, then $n$ is even and
$L_p(x^{n/2}) = \max \{ L_p(x^i) : x^i \in A \cap E_p \}$. The point $x^{n/2}$ is called the
{\em midpoint} of $A$.

It frequently happens that $\pi_p(u) \ne \pi_p(v)$, but  $u$ and $v$ belong to the same link
$\ell \in \chain_p$. Let us call the arc-components $A_u$, $A_v$ of $\A \cap \ell$ that contain $u$ and $v$
respectively the {\em link-tips} of $A$, see Figure~\ref{fig:link-tips}. Sometimes we can make $A$ $p$-symmetric
by removing the link-tips. Let us denote this as $A \setminus \ell\mbox{-tips}$. Adding the closure of the
link-tips can sometimes also produce a $p$-symmetric arc.
\begin{figure}[ht]
\unitlength=8mm
\begin{picture}(8,4)(9,2)
\thicklines
\put(12,3.75){\line(1,0){1}}
\put(12,3.625){\oval(0.25,0.25)[l]}
\put(12,3.5){\line(1,0){4}}
\put(16,3.375){\oval(0.25,0.25)[r]}
\put(10,3.25){\line(1,0){6}}
\put(10,3){\oval(0.5,0.5)[l]}
\put(10,2.75){\line(1,0){6}}
\put(16,2.625){\oval(0.25,0.25)[r]}
\put(12.35,2.5){\line(1,0){3.65}}
\thinlines
\put(14.4, 3.9){$A = [u,v]$}\put(11, 3.8){$\ell$}
\put(14.4,5){\vector(-1,-1){1.2}}\put(13.5, 5.2){link-tips $A_u$ and $A_v$}
\put(14.5,5){\vector(-1,-2){1.35}}
\put(12.5,3){\oval(3,2)}
\put(12.5,3.75){\line(1,0){1.2}}
\put(12,2.5){\line(1,0){0.5}}
\put(12,2.375){\oval(0.25,0.25)[l]}
\put(12,2.25){\line(1,0){1.7}}
\put(12.3, 2.5){\circle*{0.11}}\put(13, 3.75){\circle*{0.11}}
\put(12.03, 2.57){\tiny $v$}\put(12.67, 3.82){\tiny $u$}
\end{picture}
\caption{The arc $A$ is neither $p$-symmetric, nor quasi-$p$-symmetric, but both arcs
$A \setminus \ell\mbox{-tips}$ and $A \cup \Cl (\ell\mbox{-tips})$ are $p$-symmetric.}
\label{fig:link-tips}
\end{figure}
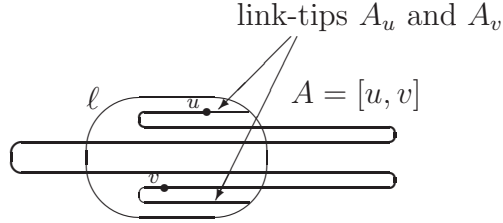

\begin{remark}\label{rem_basic}

{\bf (a)} Let $A$ be an arc and $m \in A$ be a $p$-point of maximal $p$-level, say $L_p(m) = L$. Then $\pi_p$
is one-to-one on both components of $\sigma^{1-L}(A \setminus \{ m \})$, so $m$ is the only $p$-point of
$p$-level $L$. It follows that between every two $p$-points of the same $p$-level, there is a $p$-point $m$ of
higher $p$-level.

{\bf (b)} If $A \owns m$ is the maximal open arc such that $m$ has the highest $p$-level on $A$, then we can
write $\Cl A = [x,y]$ or $[y,x]$ with $L_p(x) > L_p(y)> L_p(m) =: L$, and $\pi_p$ is one-to-one on
$\sigma^{-L}(\Cl A)$. Here $L_p(x) = \infty$ is possible, but if $L_p(x) < \infty$, then
$A' := \pi_p \circ \sigma^{-L}(A)$ is a neighborhood of $c$ with boundary points
$c_{S_k} = \pi_p \circ  \sigma^{-L}(x)$ and $c_{S_l} = \pi_p \circ  \sigma^{-L}(y)$ for some $k,l \in \None$
such that $l = Q(k)$. By Lemma~\ref{lem:order} this means that the arc $[x,m]$ is shorter than $[m,y]$.
\end{remark}

\begin{defi}\label{quasi-p-symmetric}
Let $A$ be an arc of $\A$. We say that the arc $A$ is \emph{quasi-$p$-symmetric with respect to} $\chain_p$ if
\begin{itemize}
\item[(i)] $A$ is not $p$-symmetric;
\item[(ii)] $\bd A$ belongs to a single link $\ell$;
\item[(iii)] $A \setminus \ell\mbox{-tips}$ is $p$-symmetric;
\item[(iv)]  $A \cup \ell\mbox{-tips}$ is not $p$-symmetric.
(So $A$ cannot be extended to a symmetric arc within its boundary link $\ell$.)
\end{itemize}
\end{defi}

\begin{defi}\label{def:linksym}
Let $\ell_0, \ell_1, \dots, \ell_k$ be the links in $\chain_p$ that are successively visited by an arc
$A \subset \A$, and let $A_i \subset \Cl({\ell_i})$ be the corresponding maximal subarcs of $A$.
(Hence $\ell_i \neq \ell_{i+1}$, $\ell_i \cap \ell_{i+1} \neq \emptyset$ but $\ell_i = \ell_{i+2}$ is
possible if $A$ turns in $\ell_{i+1}$.) We call $A$ {\em $p$-link-symmetric} if $\ell_i = \ell_{k-i}$ for
$i = 0, \dots, k$. In this case, we say that $A_i$ is $p$-link-symmetric to $A_{k-i}$.
\end{defi}
\begin{remark}
Every $p$-symmetric and quasi-$p$-symmetric arc is $p$-link-symmetric by definition, but there are
$p$-link-symmetric arcs which are not $p$-symmetric or quasi-$p$-symmetric. This occurs if $A$ turns
both at $A_i$ and $A_{k-i}$, but the midpoint of $A_i$ has a higher $p$-level than the midpoint of
$A_{k-i}$ and $i \notin \{ 0, k \}$. Note that for a $p$-link-symmetric arc $A$, if $U$ and $V$ are
$p$-link-symmetric arc-components which do not contain any boundary point of $A$, then $U$ contains at
least one $p$-point if and only if $V$ contains at least one $p$-point.
\end{remark}

Appendix~\ref{sec:furtherlemmas} is devoted to give a precise description
of quasi-symmetric arcs and their concatenated components.
In Appendix~\ref{sec:link} we use this structure to show that
link-symmetric arcs are always symmetric or a well-understood concatenation of quasi-symmetric arcs.

\section{Salient Points and Homeomorphisms}\label{sec:homeomorphisms}

Note that in this section all proofs except the proof of Proposition \ref{prop:symmetric} work in general,
only the proof of Proposition~\ref{prop:symmetric} uses the special structure of the Fibonacci-like inverse limit spaces revealed in this paper.

\begin{defi}\label{df:salient}[see \cite[Definition 2.7]{BBS}]
Let $(s_i)_{i \in \None}$ be the sequence of all $p$-points of the arc-component $\C0$ such that
$0 \leq L_p(x) < L_p(s_i)$ for every $p$-point $x \in (\bar 0 , s_i)$. We call $p$-points satisfying this
property \emph{salient}.
\end{defi}

For every slope $s > 1$ and $p \in \N0$, the folding pattern of $\C0$ starts as
$\infty \ 0 \ 1 \ 0 \ 2 \ 0 \ 1 \ \dots$, and since by definition $L_p(s_1) > 0$, we have $L_p(s_1) = 1$.
Also, since $s_i = \sigma^{i-1}(s_1)$, $L_p(s_i) = i$, for every $i \in \None$. Note that the salient
$p$-points depend on $p$: if $p \geq q$, then the salient $p$-point $s_i$ equals the salient $q$-point
$s_{i+p-q}$.

\begin{defi}\label{df:Rsalient}
Recall that $\Rr$ is the arc-component containing the point $\rho = (\dots,r,r,r)$ where
$r = \frac{s}{s+1}$ is fixed by $T_s$.
Let $(t^i)_{i \in \Z} \subset E^{\Rr}_p$ be the bi-infinite sequence of all $p$-points of the arc-component $\Rr$ such that for every $i \in \None$
$$
\left\{ \begin{array}{ll}
t^0 = \rho, & \\
L_p(t^i) > L_p(x) & \text{ for every $p$-point } x \in (\rho, t^i), \\
L_p(t^{-i}) > L_p(x) & \text{ for every $p$-point } x \in (t^{-i}, \rho).
\end{array} \right.
$$
\end{defi}

Note that $p$-points $(t^i)_{i \in \Z} \subset \Rr$ are defined similarly as salient $p$-points
$(s_i)_{i \in \None}$; we call them \emph{$\Rr$-salient} $p$-points, or simply
\emph{salient} $p$-points when it is clear which arc-component they belong to. There is an
important difference between the sets $(s_i)_{i \in \None} \subset \C0$ and $(t^i)_{i \in \Z} \subset \Rr$, namely $L_p(s_i) = i$ for every $i \in \None$, whereas $L_p(t^i) \ne |i|$ for all
$i \in \Z \setminus \{ 1 \}$.

\begin{lemma} \label{lem:PR1}
For $(t^i)_{i \in \Z} \subset \Rr$ we have
$$
L_p(t^i) = \begin{cases}
2i-1 & \text { if } i > 0, \\
-2i & \text { if } i < 0.
\end{cases}
$$
\end{lemma}

\begin{proof}
Since $r$ is the positive fixed point of $T_s$, the $p$-points closest to $\rho = (\dots, r, r, r)$ have
$p$-levels 0 and 1. Also $\sigma(\rho) = \rho$ implies $\sigma(\Rr) = \Rr$. The parametrization of $\Rr$,
chosen in Section~\ref{sec:def} below \eqref{eq:Rr}, is such that for $\rho \in [x^{-1}, x^1]$ we have
$L_p(x^{-1}) = 0$ and $L_p(x^{1}) = 1$, thus $x^1 = t^1$. Since
$\sigma(\rho) = \rho \in \sigma([x^{-1}, x^1]) \subset \Rr$ and
$\sigma|_{\Rr}$ is order reversing, we have $\sigma(x^{-1}) = x^1$, $\sigma(x^1) \prec x^1$, \ie
$\sigma([x^{-1}, x^1]) = [x^{-2}, x^1]$ with $L_p(x^{-2}) = 2$. Note that $x^{-2} = t^{-1}$. For the same
reason, $\sigma([x^{-2}, x^1]) = [x^{-2}, x^j]$, where $x^j$ is the first $p$-point to the right of
$x^1$ such that $L_p(x^j) = 3$, \ie $x^j = t^2$. The claim of the lemma follows by induction.
\end{proof}

Analogously, we define $\A$-salient $p$-points of an arc-component $\A$ of the core of $K_s$.
\begin{defi}\label{df:Asalient}
Let $(u^i)_{i \in \Z} \subset E^{\A}_p = (a^i)_{i \in \Z}$ be the bi-infinite sequence of all $p$-points of the arc-component $\A$ such that for every $i \in \None$
$$
\left\{ \begin{array}{ll}
u^0 = a^0, & \\
L_p(u^i) > L_p(x) & \text{ for every $p$-point } x \in (u^0, u^i), \\
L_p(u^{-i}) > L_p(x) & \text{ for every $p$-point } x \in (u^{-i}, u^0).
\end{array} \right.
$$
This fixes an orientation on $\A$; the choice of orientation is immaterial, as long as we make one.
\end{defi}

\begin{lemma} \label{lem:PR2}
If there exist $J, J', K \in \N0$ such that for every $j \in \None$, $L_p(u^{J+j}) = 2(K+j)-1$  and
$L_p(u^{-(J'+j)}) = 2(K+j)$, then $\A = \Rr$.
\end{lemma}

In other words, the asymptotic shape of this folding pattern is unique to $\Rr$.

\begin{proof}
Let $J, J', K \in \N0$ be as in the statement of the lemma. Then for every $j \in \None$ we have:
\begin{itemize}
\item[(i)] $L_p(u^{-(J'+j)}) - L_p(u^{J+j}) = L_p(u^{J+j+1}) - L_p(u^{-(J'+j)}) = 1$,
\item[(ii)]  $L_p(x) < L_p(u^{J+j})$ for every $p$-point $x \in (u^{-(J'+j)}, u^{J+j})$, and
\item[(iii)] $L_p(x) < L_p(u^{-(J'+j)})$ for every $p$-point $x \in (u^{-(J'+j)}, u^{J+j+1})$.
\end{itemize}
Therefore,
$\pi_{p+2(K+j)-1}: [u^{-(J'+j)}, u^{J+j}] \to [c, c_1]$ and
$\pi_{p+2(K+j)}: [u^{-(J'+j)}, u^{J+j+1}] \to [c, c_1]$ are
bijections, implying that for every $j \in \None$, $FP_p([u^{-(J'+j)}, u^{J+j}])$ and
$FP_p([u^{-(J'+j)}, u^{J+j+1}])$ are uniquely determined by $T^{2(K+j)-1}_s$ and $T^{2(K+j)}_s$
respectively. Thus we have the following:
$$
\left\{ \begin{array}{l}
FP_p([u^{-(J'+j)}, u^{J+j}]) = FP_p([t^{-(K+j)}, t^{K+j}]),  \\[2mm]
FP_p([u^{-(J'+j)}, u^{J+j+1}]) = FP_p([t^{-(K+j)}, t^{K+j+1}]),
\end{array} \right.
$$
whence $FP_p([u^{-(J'+j)}, u^{J+j+1}]) = FP_p(\sigma([u^{-(J'+j)}, u^{J+j}]))$ for every $j \in \None$.
It follows that $FP_p(\sigma(\A)) = FP_p(\A) = FP_p(\Rr)$ implying $\A = \Rr$.
\end{proof}

Note that in general $J, J', K$ in the above lemma are not related since $u^0 = a^0$ can be any point, but
there exists a point $a \in \A$ such that for $u^0 = a$, we have $J = J' = K$.

Let $h : \ILcores \to \ILcores$ be a homeomorphism on the core of a (Fibonacci-like) inverse limit space.
Let $q, p, g \in \N0$ be such that $\chain_q$, $\chain_p$ and $\chain_g$ are chains as in
Proposition~\ref{prop:chains}, and such that
$$
h(\chain_{q}) \preceq \chain_{p} \preceq h(\chain_{g}).
$$
It is straightforward that any $q$-link-symmetric arc $A \subset \ILcores$ maps to a $p$-link-symmetric arc
$h(A) \subset \ILcores$.

In Appendix~\ref{sec:chains}, we construct special chains by which we are able to describe the structure of
link-symmetric arcs (see Definition~\ref{def:linksym})
precisely.
The Fibonacci-like structure, and the extra structure of these chains,
allow us to conclude the stronger statement that $q$-symmetric arcs map to $p$-symmetric arcs.
This is a rather technical undertaking, but let us paraphrase Remark~\ref{rem:quasi-sym} so as to make
this section understandable (although for the fine points we will still refer forward to the appendix).
Link-symmetric arcs tend to be composed of smaller {\em (basic) quasi-symmetric
arcs} $A_k$ (see Definition~\ref{basic-quasi-p-symmetric})
that are ordered linearly such that $A_k$ and $A_{k+1}$ overlap, and the midpoint
of $A_{k+1}$ is the endpoint of $A_k$. An entire concatenation of such arcs
is called {\em decreasing quasi-symmetric} (respectively {\em increasing quasi-symmetric},
see Definition~\ref{def:decreasing})
if the levels of the successive midpoints (also called {\em nodes})
- all contained in, alternately, one of two given links - are decreasing (respectively increasing).
The concatenation is called {\em maximal decreasing quasi-symmetric} (respectively {\em maximal
increasing quasi-symmetric}, see Definition~\ref{def:maximaldecreasing}) if it cannot be extended to
a concatenation with more components. The last endpoint (respectively the first endpoint), namely,
of the arc with midpoint of the lowest level, is then no longer a $p$-point.

For a point $x$, we denote a link of $\chain_p$ which contains $x$ by $\ell_p^x$, and the
arc-component of $\ell_p^x$ which contains $x$ by $A_x$.

\begin{defi}\label{def:extended_arc}
Let $x \in E_q^{\A} \subset \A$ be a $q$-point, and let $A_{h(x)} \subset \ell_p^{h(x)}$ be the
arc-component of $\ell_p^{h(x)}$ which contains $h(A_x)$ (and therefore $h(x)$). Let $a, b \in \None$,
$a \le b$, be such that $h(\cup_{i = a}^b \ell_q^i) \subseteq \ell_p^{h(x)}$,
$h(\ell_q^{a-1}) \nsubseteq \ell_p^{h(x)}$ and $h(\ell_q^{b+1}) \nsubseteq \ell_p^{h(x)}$. Let $\hat A_x$ be an
arc-component of $\cup_{i = a}^b \ell_q^i$ such that $h(\hat A_x) \subseteq A_{h(x)} \subset \ell_p^{h(x)}$.
We call $\hat A_x$ the \emph{extended arc-component of the $q$-point} $x$. If a $p$-point $u$ is the midpoint
of $A_{h(x)}$, then we write $u \vdash h(x)$.
\end{defi}

The extended arc-component $\hat A_x$ is obtained by extending $A_x$ so much on both sides that
$h(\hat A_x)$ fits almost exactly in the $p$-link containing $h(A_x)$. Note that the arc-component $A_x$
of a $q$-point $x$ depends on the chain $\chain_q$, while the extended arc-component $\hat A_x$
of the $q$-point $x$ also depends on the chain $\chain_p$. But we still can define its midpoint as the
$q$-point $z \in \hat A_x$ such that $L_q(z) \ge L_q(y)$ for every $q$-point $y \in \hat A_x = \hat A_z$.
If a $q$-point $x$ is the midpoint of its extended arc-component $\hat A_x$ we call it a \emph{$q_p$-point}.

\begin{proposition}\label{prop:symmetric}
Let $x, y \in E_q^{\A} \subset \A$ be $q_p$-points and let $u \vdash h(x)$ and $v \vdash h(y)$.
Then $L_q(x) = L_q(y)$ implies $L_p(u) = L_p(v)$.
\end{proposition}

Since the endpoints of a symmetric arc have the same level, and
$q$-link symmetric arcs are mapped to $p$-link-symmetric arcs
by a homeomorphism $h$, Proposition~\ref{prop:symmetric} implies that $h$
maps symmetric arcs to symmetric arcs.

\begin{proof}
Without loss of generality we suppose that between $x$ and $y$, there are no $q$-points with $q$-level
$L_q(x)$. Then the arc $A = [x, y]$ is $q$-symmetric. The midpoint $m$
of $A$ is a $q_p$-point. Let $w \vdash h(m)$.

Let us assume by contradiction that $L_p(u) \ne L_p(v)$. Then $D = [u, v]$ is not $p$-symmetric with
midpoint $w$. Since $A$ is $q$-symmetric, $D$ is $p$-link symmetric. By Proposition~\ref{thm:linksym} and
Remark~\ref{rem:quasi-sym}, $D$ is contained either in an extended maximal decreasing/increasing (basic)
quasi-$p$-symmetric arc, or in a $p$-symmetric arc which is concatenation of two arcs, one of which is
a maximal increasing (basic) quasi-$p$-symmetric arc, and the other one is a maximal decreasing (basic)
quasi-$p$-symmetric arc.

{\bf (1)} Let us assume that $D$ is contained in an extended maximal increasing (basic) quasi-$p$-symmetric arc $G$.
Let $B'$ and $B$ be the link-tips of $G$, so $G = [B', B]$. Then, by Remark \ref{rem:quasi-sym}, $B'$ does
not contain any $p$-point and hence $B' \ne A_u$.

(a) Suppose first that the $p$-point $z \in G$, such that $L_p(z) \ge L_p(d)$ for all $p$-points $d \in G$,
does not belong to the open arc $(u, v)$. Then $B \ne A_v$.

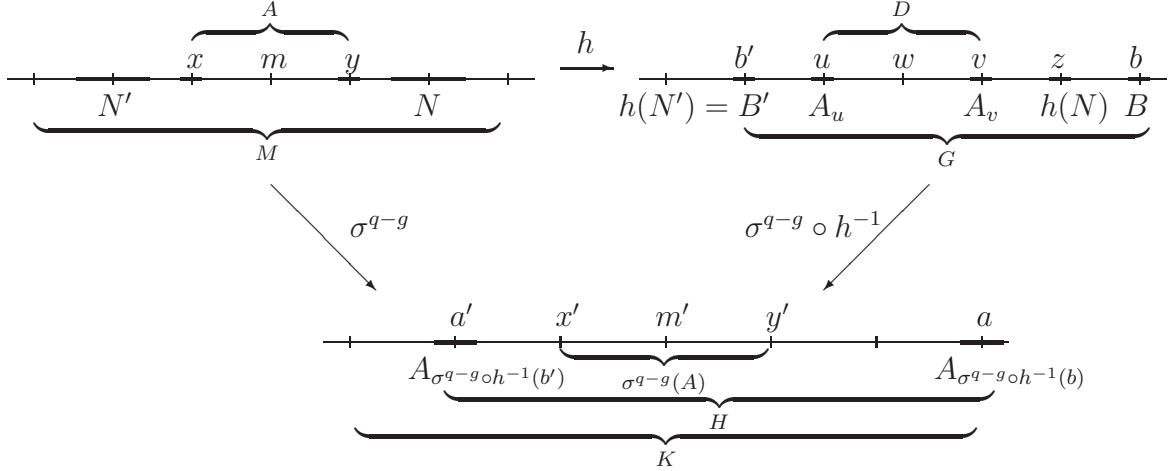
\begin{figure}[ht]
\unitlength=7mm
\begin{picture}(15,8.5)(3.5,-1)
\put(0,6){\line(1,0){10}}
\put(0.5,5.9){\line(0,1){0.2}}
\put(2,5.9){\line(0,1){0.2}}
\put(1.7,5.3){$N'$}
\put(3.5,5.9){\line(0,1){0.2}}
\put(3.4,6.2){$x$}
\put(5,5.9){\line(0,1){0.2}}
\put(4.8,6.2){$m$}
\put(6.5,5.9){\line(0,1){0.2}}
\put(6.4,6.2){$y$}
\put(8,5.9){\line(0,1){0.2}}
\put(7.7,5.3){$N$}
\put(9.5,5.9){\line(0,1){0.2}}
\put(0.5,5.2){$\underbrace{\qquad \qquad \qquad \qquad \qquad \qquad \qquad \quad}_{M}$}
\put(3.5,6.7){$\overbrace{\qquad \qquad \quad}^{A}$}
\put(10.5,6.2){\vector(1,0){1}}\put(10.8,6.5){$h$}
\put(12,6){\line(1,0){10}}
\put(12.5,5.9){\line(0,1){0.2}}
\put(14,5.9){\line(0,1){0.2}}
\put(11.6,5.3){$h(N') = B'$}
\put(13.8,6.2){$b'$}
\put(15.5,5.9){\line(0,1){0.2}}
\put(15.2,5.3){$A_{u}$}
\put(15.3,6.2){$u$}
\put(17,5.9){\line(0,1){0.2}}
\put(16.8,6.2){$w$}
\put(18.5,5.9){\line(0,1){0.2}}
\put(18.15,5.3){$A_{v}$}
\put(18.3,6.2){$v$}
\put(20,5.9){\line(0,1){0.2}}
\put(19.6,5.3){$h(N)$}
\put(19.8,6.2){$z$}
\put(21.5,5.9){\line(0,1){0.2}}
\put(21.2,5.3){$B$}
\put(21.3,6.2){$b$}
\put(14,5.1){$\underbrace{\qquad \qquad \qquad \qquad \qquad \qquad \quad}_{G}$}
\put(15.5,6.7){$\overbrace{\qquad \qquad \quad}^{D}$}
\put(5,4){\vector(1,-1){2}}\put(6.5,3){$\sigma^{q-g}$}
\put(17.5,4){\vector(-1,-1){2}}\put(14,3){$\sigma^{q-g} \circ h^{-1}$}
\put(6,1){\line(1,0){13}}
\put(6.5,0.9){\line(0,1){0.2}}
\put(8.5,0.9){\line(0,1){0.2}}
\put(7.6,0.3){$A_{\sigma^{q-g} \circ h^{-1}(b')}$}
\put(10.5,0.9){\line(0,1){0.2}}
\put(8.4,1.3){$a'$}
\put(10.4,1.3){$x'$}
\put(12.5,0.9){\line(0,1){0.2}}
\put(12.3,1.3){$m'$}
\put(14.5,0.9){\line(0,1){0.2}}
\put(14.4,1.3){$y'$}
\put(16.5,0.9){\line(0,1){0.2}}
\put(17.6,0.3){$A_{\sigma^{q-g} \circ h^{-1}(b)}$}
\put(18.4,1.3){$a$}
\put(18.5,0.9){\line(0,1){0.2}}
\put(8.3,0.1){$\underbrace{ \qquad \qquad \qquad \qquad \qquad \qquad \qquad \quad \quad \quad \ \ }_{H}$}
\put(6.6,-0.6){$\underbrace{\qquad \qquad \qquad \qquad \qquad \qquad \qquad \qquad \qquad \quad \ \ \ }_{K}$}
\put(10.5,0.9){$\underbrace{\qquad \qquad \qquad \ \ }_{\sigma^{q-g}(A)}$}
\thicklines
\put(1.3,6.01){\line(1,0){1.4}}
\put(1.3,5.99){\line(1,0){1.4}}
\put(7.3,6.01){\line(1,0){1.4}}
\put(7.3,5.99){\line(1,0){1.4}}
\put(3.28,6.01){\line(1,0){0.4}}
\put(3.28,5.99){\line(1,0){0.4}}
\put(6.28,6.01){\line(1,0){0.4}}
\put(6.28,5.99){\line(1,0){0.4}}
\put(13.78,6.01){\line(1,0){0.4}}
\put(13.78,5.99){\line(1,0){0.4}}
\put(15.28,6.01){\line(1,0){0.4}}
\put(15.28,5.99){\line(1,0){0.4}}
\put(18.28,6.01){\line(1,0){0.4}}
\put(18.28,5.99){\line(1,0){0.4}}
\put(19.78,6.01){\line(1,0){0.4}}
\put(19.78,5.99){\line(1,0){0.4}}
\put(21.28,6.01){\line(1,0){0.4}}
\put(21.28,5.99){\line(1,0){0.4}}
\put(8.1,1.01){\line(1,0){0.8}}
\put(8.1,0.98){\line(1,0){0.8}}
\put(18.1,1.01){\line(1,0){0.8}}
\put(18.1,0.98){\line(1,0){0.8}}
\end{picture}
\caption{The relations between points and arcs in $\chain_q$ (left), $\chain_p$ (right), and $\chain_g$ (bottom).
\label{fig:lem:l}}
\end{figure}

Let $b'$ be any point of $B'$ and let $b$ be the midpoint of $B$. 
Then $b'$ and $b$ are nodes of $G$ (see Remark~\ref{rem:quasi-sym} for the definition of a node).
Since $L_p(u) \ne L_p(v)$, $u$ and $v$ are also nodes of $G$, as well as $w$ and $z$.

Let $a \vdash \sigma^{q-g} \circ h^{-1}(b)$ (note that $b$ is a $p_g$-point, \ie $b$ is the midpoint of the
extended arc-component $\hat A_b$ such that $\sigma^{q-g} \circ h^{-1}(\hat A_b) \subseteq
A_{\sigma^{q-g} \circ h^{-1}(b)} = A_a \subset \ell_g^a \in \chain_g$). If the arc-component
$A_{\sigma^{q-g} \circ h^{-1}(b')}$ contains a $g$-point, let $a'$ be its midpoint; otherwise let $a'$ be
any point of $A_{\sigma^{q-g} \circ h^{-1}(b')}$. Let us consider the arc $H = [a', a]$, see
Figure~\ref{fig:lem:l}. Let $x' \vdash \sigma^{q-g} \circ h^{-1}(u)$, $y' \vdash \sigma^{q-g} \circ h^{-1}(v)$,
$z' \vdash \sigma^{q-g} \circ h^{-1}(z)$ and $m' \vdash \sigma^{q-g} \circ h^{-1}(w)$. Since
$\chain_p \prec h(\chain_g)$, the arc $H$ is $g$-link-symmetric and $g$-points $a', x', m', y', a$ are some
of its nodes. Note that $x' = \sigma^{q-g}(x)$ and $y' = \sigma^{q-g}(y)$, thus the arc $[x', y']$ is
$g$-symmetric. Since there is at least one node in $H$ on either side of $[x', y']$,
Remark~\ref{rem:quasi-sym} says that $H$ is contained in the maximal $g$-symmetric arc $K$ with midpoint
$m'$. Therefore the arc $M = \sigma^{-q+g}(K) \supset A$ is $q$-symmetric with midpoint $m$.

Let $j, k \in \None$, $j \le k$, be such that $h(\cup_{i = j}^k \ell_q^i) \subseteq \ell_p^{b'}$,
$h(\ell_q^{j-1}) \nsubseteq \ell_p^{b'}$ and $h(\ell_q^{k+1}) \nsubseteq \ell_p^{b'}$. Let $N'$ be an
arc-component of $\cup_{i = j}^k \ell_q^i$ such that $h(N') = B' \subset \ell_p^{b'}$. Obviously,
$N' \subset M$. Since $M$ is $q$-link symmetric, there exists an arc-component $N$ of
$\cup_{i = j}^k \ell_q^i$ such that the arc $[N', N] \subset M$ is $q$-symmetric with midpoint $m$.
Then $h(N) \subset h(M)$ is an arc-component of $\ell_p^{b'}$. Since $[N', N]$ is $q$-symmetric, the
arc-component $h(N')$ contains a $p$-point if and only if the arc-component $h(N)$ contains a $p$-point.
Since $h(N') = B'$, the arc-components $h(N')$ and $h(N)$ do not contain any $p$-point, see
Figure~\ref{fig:lem:l}.

On the other hand, the arc $[h(N'), h(N)]$ is $p$-link-symmetric with midpoint $w$. Recall that $w$
is also the midpoint of the arc $D \subset [h(N'), h(N)]$, $D$ is not $p$-symmetric by assumption, and
$D \subset G$, where $G$ is an extended maximal increasing (basic) quasi-$p$-symmetric arc. The arc-component
$h(N)$ can be contained in the arc $[A_v, B]$, as in Figure~\ref{fig:lem:l}. In this case $h(N)$ does
contain at least one $p$-point, a contradiction.

The other possibility is that $h(N)$ is not contained in $[A_v, B]$, \ie $h(N)$ is on the right hand
side of $B$. Since $[h(N'), h(N)]$ is $p$-link symmetric and $h(N') = B'$ contains a node $b'$ of $G$, we
have that $h(N)$ also contains a node of $G$, say $n$. Hence, on the right hand side of $z$ (which is the
$p$-point with the highest $p$-level in $G$), there are at least two nodes, $b$ and $n$. Therefore, by
Remark~\ref{rem:quasi-sym}, $G$ is contained in a $p$-symmetric arc with midpoint $z$ and this arc
conatins $h(N)$, implying that $h(N)$ does contain at least one $p$-point, a contradiction.

(b) Let us assume now that $B = A_v$. Then $z \in (u, v)$. Let $a', x', m', z', y'$ and $H$ be defined
as in case (a). Since $b', u, w, z, v$ are nodes of $G$, we have that $a', x', m', z', y'$ are also nodes of
$H$. Moreover, since $[x', y']$ is $g$-symmetric with midpoint $m'$, there is $z'' \in [x', m']$ such that
$[z'', z']$ is $g$-symmetric with midpoint $m'$, and $z''$ is a node of $H$. Thus, the arc
between nodes $z''$ and $z'$ is $g$-symmetric, and on either side of $[z'', z']$ there is at least one
additional node. By Remark~\ref{rem:quasi-sym}, $H$ is contained in the maximal $g$-symmetric arc $K$ with
midpoint $m'$, and the arc $M = \sigma^{-q+g}(K) \supset A$ is $q$-symmetric with midpoint $m$. Now the
proof follows in the same way as in case (a).

If $D$ is contained in an extended maximal decreasing (basic) quasi-$p$-symmetric arc $G$, the proof is
analogous.

{\bf (2)} Let us assume that $D$ is contained in a $p$-symmetric arc $G$ which is concatenation of two arcs,
one of which is a maximal increasing (basic) quasi-$p$-symmetric arc, and the other one is a maximal
decreasing (basic) quasi-$p$-symmetric arc. Let $B'$ and $B$ be the link-tips of $G$, thus $G = [B', B]$.
Then, by Remark \ref{rem:quasi-sym}, $B'$ and $B$ do not contain any $p$-point and hence $B' \ne A_u$
and $B \ne A_v$. If for the midpoint $z$ of $G$ we have $z \nin (u, v)$, we are in case (1). If
$z \in (u, v)$ (note $z \ne m$ since the arc $D$ is not $p$-symmetric),
then the proof is analogous to the
proof of case (1a) (since $B \ne A_v$).
\end{proof}

\begin{defi}\label{defi:bridges}
Let $\kappa \in \None$, $\kappa > 2$, be the smallest integer with $c_{\kappa} < c$. It is easy to see that
$\kappa$ is odd. Set
$$
\Lambda_\kappa :=
\None \setminus \{ 1,3, 5, \dots, \kappa-4\}.
$$
\end{defi}

\begin{lemma}\label{lem:bridges}
Let $x,y$ be $q$-points of $\A$. Then there exist $q_p$-points $x'$, $z'$ and $y'$ such that the arc
$A = [x',z']$ is $q$-symmetric with midpoint $y'$, $L_q(x') = L_q(z') = L_q(x)$ and $L_q(y') = L_q(y)$
if and only if $L_q(y)-L_q(x) \in \Lambda_\kappa$.
\end{lemma}

This is proven in Lemma 46 of \cite{Kail2} and in Lemmas 3.13 and 3.14 of \cite{Stim2}. Although
\cite{Kail2} deals with the periodic case and \cite{Stim2} with the finite orbit case, the proofs of the
mentioned lemmas work in the general case, as stated above.

\begin{proposition}\label{prop:nonsymmetric}
Let $x, y \in E_q^{\A} \subset \A$ be $q_p$-points and let $u \vdash h(x)$ and $v \vdash h(y)$.
Then $L_q(x) < L_q(y)$ implies $L_p(u) < L_p(v)$.
\end{proposition}

\begin{proof}
{\bf (1)}
Let us first assume that $L_q(y) - L_q(x) \in \Lambda_\kappa$. Then, by Lemma~\ref{lem:bridges}, there exist
$q_p$-points $x'$, $z'$ and $y'$ such that the arc $A = [x', z']$ is $q$-symmetric with midpoint $y'$,
$L_q(x') = L_q(z') = L_q(x)$, $L_q(y') = L_q(y)$ and between $x'$ and $z'$ there are no $q_p$-points with
$q$-level $L_q(x')$. Let $u \vdash h(x)$, $v \vdash h(y)$, $u' \vdash h(x')$, $v' \vdash h(y')$,
$w' \vdash h(z')$. By Proposition~\ref{prop:symmetric} we have $L_p(u) = L_p(u') = L_p(w')$,
$L_p(v) = L_p(v')$ and between the points $u'$ and $w'$ there are no $p$-points with the $p$-level $L_p(u')$.
Therefore, the arc $[u', w']$ is $p$-symmetric with midpoint $v'$, implying
$L_p(v) = L_p(v') > L_p(u') = L_p(u)$, which proves the proposition in this case. Note that also we have
$L_p(v) - L_p(u) \in \Lambda_\kappa$.

\begin{figure}[ht]
\unitlength=8mm
\begin{picture}(16,5)(0,1)
\put(1,5){\oval(0.25,0.25)[l]}
\put(0.875, 5){\circle*{0.11}}\put(0.3,4.75){$x$}
\put(1.5,4.8){\dots \dots \dots}
\put(4,4.5){\oval(0.25,0.25)[l]}
\put(3.875, 4.5){\circle*{0.11}}\put(3.4,4.4){$y$}
\put(1,3){\oval(0.25,0.25)[l]}
\put(0.875, 3){\circle*{0.11}}\put(0.3,2.75){$x'$}
\put(1,2.875){\line(1,0){4}}
\put(5,2.75){\oval(0.25,0.25)[r]}
\put(5,2.625){\line(-1,0){1}}
\put(4,2.5){\oval(0.25,0.25)[l]}
\put(3.875, 2.5){\circle*{0.11}}\put(3.4,2.4){$y'$}
\put(5,2.375){\line(-1,0){1}}
\put(5,2.25){\oval(0.25,0.25)[r]}
\put(1,2.125){\line(1,0){4}}
\put(1,2){\oval(0.25,0.25)[l]}
\put(0.875, 2){\circle*{0.11}}\put(0.3,1.85){$z'$}
\put(6.5,3.5){\vector(1,0){3}} \put(8,3.7){$h$}
\put(11,5){\oval(0.25,0.25)[l]}
\put(10.875, 5){\circle*{0.11}}\put(10.3,4.75){$u$}
\put(11.5,4.8){\dots \dots \dots}
\put(14,4.5){\oval(0.25,0.25)[l]}
\put(13.875, 4.5){\circle*{0.11}}\put(13.4,4.4){$v$}
\put(11,3){\oval(0.25,0.25)[l]}
\put(10.875, 3){\circle*{0.11}}\put(10.3,2.75){$u'$}
\put(11,2.875){\line(1,0){4}}
\put(15,2.75){\oval(0.25,0.25)[r]}
\put(15,2.625){\line(-1,0){1}}
\put(14,2.5){\oval(0.25,0.25)[l]}
\put(13.875, 2.5){\circle*{0.11}}\put(13.4,2.4){$v'$}
\put(15,2.375){\line(-1,0){1}}
\put(15,2.25){\oval(0.25,0.25)[r]}
\put(11,2.125){\line(1,0){4}}
\put(11,2){\oval(0.25,0.25)[l]}
\put(10.875, 2){\circle*{0.11}}\put(10.3,1.85){$w'$}
\end{picture}
\caption{The points $x$ and $y$, their companion arc $A = [x', z']$ and their images under $h$. Dots indicate some shape of the arc
$[x,y]$ and $[u,v]$; the shape of $[x,y]$ can be very different from the shape of $[x',y']$ and similar for the shapes of $[u,v]$ and $[u',v']$.}
\label{fig:AhA}
\end{figure}
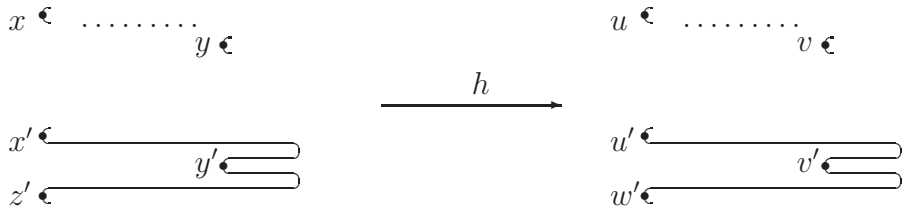

{\bf (2)}
Let us now assume that $\Lambda_\kappa \ne \None$, $L_q(y) - L_q(x) \in \{ 1,3, \dots, \kappa-4\}$, and that
for $u \vdash h(x)$ and $v \vdash h(y)$ we have, by contradiction, $L_p(u) > L_p(v)$.

Without loss of generality we suppose that $x$ has the smallest $q$-level among all $q_p$-points which
satisfy the above assumption and that, for this choice of $x$, the $q_p$-point $y$ (which also satisfies
the above assumption) is such that $L_q(y) - L_q(x) > 0$ is the smallest difference of $q$-levels.

{\bf Claim 1:} $L_q(y) - L_q(x) = 1$.

Let us assume, by contradiction, that $L_q(y) - L_q(x) > 1$, and
let $z$ be a $q_p$-point such that
$L_q(y) - L_q(z) = 2$. Note first that $L_q(z) \ne L_q(x)$ since $L_q(y) - L_q(x) \ne 2$ by assumption.
Therefore, $L_q(z) > L_q(x)$.

Let $w \vdash h(z)$ and recall $u \vdash h(x)$ and $v \vdash h(y)$. By the choice of $q_p$-points $x$ and
$y$ and since $L_q(z) - L_q(x) < L_q(y) - L_q(x)$, we have $L_p(w) > L_p(u)$ and $L_p(u) > L_p(v)$, implying
$L_p(w) > L_p(v)$.

On the other hand, $L_q(y) - L_q(z) \in \Lambda_\kappa$ and by {\bf (1)} we have $L_p(v) > L_p(w)$, a
contradiction. This proves Claim 1.

{\bf Claim 2:} $L_p(u) - L_p(v) = 1$.

Let us assume, by contradiction, that $L_p(u) - L_p(v) > 1$. For a $q_p$-point $z$ let $w$ denote the
$p$-point with $w \vdash h(z)$. We will show that the above assumption implies that there is no $q_p$-point
$z$ such that $L_p(w) = L_p(v) + 1$. This contradicts assumption that both arc-components $\A$ and
$h(\A)$ are dense in $\ILcores$ in both directions.

By the choice of $q_p$-points $x$ and $y$, for every $q_p$-point $z$ such that
$L_q(z) < L_q(x) < L_q(y) = L_q(x) + 1$ we have $L_p(w) < L_p(v)$ and hence $L_p(w) \ne L_p(v) + 1$.

Let $L_q(z) = L_q(x) + 2$. Since $L_q(z) - L_q(x) \in \Lambda_\kappa$, by {\bf (1)} we have
$L_p(w) > L_p(u) > L_p(v) + 1$.

Let $L_q(z) = L_q(x) + 3$. Then $L_q(z) - L_q(y) \in \Lambda_\kappa$ (recall $L_q(y) = L_q(x) + 1$ by Claim 1)
and again by {\bf (1)} we have $L_p(w) >  L_p(v)$ and $L_p(w) -  L_p(v) \in \Lambda_\kappa$. Note that
$L_p(w) -  L_p(v) \ne 1$ since $1 \nin \Lambda_\kappa$ (recall $\Lambda_\kappa \ne \None$ by assumption).
Hence $L_p(w) >  L_p(v) + 1$.

It follows now by induction that for every $i \in \None$, $L_q(z) = L_q(x) + 3 + i$ implies
$L_p(w) > L_p(v) + 1$. To see this, for a $q_p$-point $z'$ let $w'$ denote the $p$-point with
$w' \vdash h(z')$. Take $j \in \None$ such that $L_q(z) = L_q(x) + 3 + i$ implies
$L_p(w) > L_p(v) + 1$ for every $i < j$. Let $L_q(z') = L_q(x) + 1 + j$ and $L_q(z) = L_q(x) + 3 + j$.
Then $L_q(z) - L_q(z') \in \Lambda_\kappa$ and by {\bf (1)} we have $L_p(w) >  L_p(w')$. Since
$L_p(w') > L_p(v) + 1$, we have $L_p(w) > L_p(v) + 1$. This proves Claim 2.

{\bf Claim 3:} For a $q_p$-point $z$ let $w$ denote the $p$-point with $w \vdash h(z)$. For every
$i \in \None$, $L_q(z) = L_q(x) + 2i$ implies $L_p(w) = L_p(u) + 2i$,
and $L_q(z) = L_q(y) + 2i$ implies $L_p(w) = L_p(v) + 2i$.

Let $L_q(z) = L_q(x) + 2 = L_q(y) + 1$. Note first that $L_p(w) \ne L_p(u) + 1$, since by {\bf (1)}
$L_q(z) - L_q(x) \in \Lambda_\kappa$ implies $L_p(w) - L_p(u) \in \Lambda_\kappa$. Note also that
$L_q(z) - L_q(y) \nin \Lambda_\kappa$. Therefore, $L_p(w) = L_p(v) + L = L_p(u) - 1 + L$, where
$1 < L < \kappa - 2$ is odd.

For $q_p$-points $z'$ and $z''$, let $w'$ and $w''$ denote the $p$-points with $w' \vdash h(z')$ and
$w'' \vdash h(z'')$ respectively.

Let us assume that $L_q(z') = L_q(y) + 2$ and $L_p(w') \ne L_p(v) + 2 = L_p(u) + 1$. Then
$L_p(w') > L_p(v) + 2$ and for every $q_p$-point $z''$ with $L_q(z'') > L_q(z')$ we have
$L_p(w'') > L_p(v) + 2$. This implies that there is no $q_p$-point $z''$ such that
$L_p(w'') = L_p(v) + 2 = L_p(u) + 1$, a contradiction. Therefore, $L_p(w') = L_p(v) + 2$, and by Claims
1 and 2, $L_p(w) = L_p(v) + 3 = L_p(u) + 2$. The proof of Claim 3 follows by induction in the same way.

Finally, to complete the proof of the proposition, let us consider $q_p$-point $z$ such that
$L_q(z) - L_q(x) = \kappa - 2 \in \Lambda_\kappa$. Then,
by Claim 3 (see Figure \ref{fig:level1}), $L_p(w) - L_p(u) = \kappa - 4 \nin \Lambda_\kappa$, a contradiction.

\begin{figure}[ht]
\unitlength=10mm
\begin{picture}(15,4.2)(1,1)
\put(0,4){${\color{Red} L_q(x)} \ \ L_q(y) = L_q(x) + 1 \ \ L_q(x) + 2 \ \ L_q(x) + 3 \ \ \dots \ \ L_q(x) + \kappa - 3 \ \ {\color{Red} L_q(z) = L_q(x) + \kappa - 2}$}
\put(0.5,4.5){$\overbrace{\qquad \qquad \qquad \qquad \qquad \qquad \qquad \qquad \qquad \qquad \qquad \qquad \qquad \qquad \qquad }^{L_q(z) - L_q(x) = \kappa - 2}$}
{\color{Red} \put(0.5,3.6){\vector(2,-1){2}}}
\put(2.2,3.6){\vector(-3,-2){1.5}}
\put(6.2,3.6){\vector(3,-2){1.5}}
\put(7.7,3.6){\vector(-3,-2){1.5}}
\put(8.8,3){\dots}
\put(11.3,3.6){\vector(3,-1){3}}
{\color{Red} \put(12.8,3.6){\vector(-3,-2){1.5}}}
\put(0,2){$L_p(v) \ \ {\color{Red} L_p(u) = L_p(v) + 1} \ \ L_p(v) + 2 \ \ L_p(v) + 3 \ \ \dots \ \ {\color{Red} L_p(w) = L_p(v) + \kappa - 3} \ \ L_p(v) + \kappa - 2$}

\put(1.8,1.7){$\underbrace{\qquad \qquad \qquad \qquad \qquad \qquad \qquad \qquad \qquad \qquad }_{L_p(w) - L_p(u) = \kappa - 4}$}
\end{picture}
\caption{The configuration of levels that cannot exist.
\label{fig:level1}}
\end{figure}
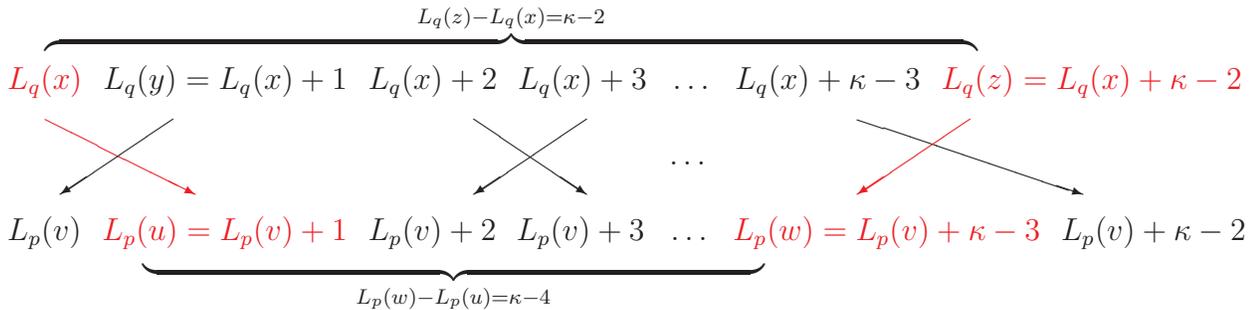

Therefore, $L_q(x) < L_q(y)$ implies $L_p(u) < L_p(v)$, which proves the proposition.
\end{proof}

\section{Proof of the main theorems}\label{sec:maintheorems}

Consider the  arc-component $\A := h(\Rr) \subset \ILcores$, and let
$E_p^{\A} = (y^i)_{i \in \Z} \subset \A$ be the set of all $p$-points of $\A$ such that $y^0 = h(\rho)$.
Let $(u^i)_{i \in \Z} \subset E_p^{\A}$ be the set of all salient $p$-points of $\A$, \ie the set of all
$\A$-salient $p$-points, with $u^0 = h(\rho)$. Recall that $\Rr$ is dense in $\ILcores$ in both directions.
Since $h$ is a homeomorphism, $\A$ and in fact $h^i(\Rr)$, $i \in \Z$, are also dense in the core $\ILcores$ in
both directions.

We want to prove that $\A = \Rr$. For a $p$-point $y$ we write $y \approx x$ if $y \in A_x$.

\begin{lemma}\label{lem:some}
There exist $M, M' \in \Z$ such that $h(t^i) \approx u^{i+M}$ and $h(t^{-j}) \approx u^{-j-M'}$, for every
$i,j \in \None$ with $i+M > 0$, $j+M' > 0$, if $h$ is order preserving, or $h(t^i) \approx u^{-i-M}$ and
$h(t^{-j}) \approx u^{j+M'+1}$ if $h$ is order reversing.
\end{lemma}

\begin{proof}
If $h : \Rr \to \A$ is order reversing, then $h \circ \sigma : \Rr \to \A$ is order preserving, and also
if the proposition works for $h \circ \sigma$, it works for $h$. Therefore we can assume without loss of
generality that $h$ is order preserving.

Let $j \in \None$, and let $B_j$ be the maximal $q$-symmetric arc with midpoint $t^j$. Since $s > \sqrt 2$,
$\rho \in B_j$. Therefore, for every $q_p$-point $x \in (\rho, t^j)$ there exists a $q_p$-point
$y \in (t^j, t^{j+1})$, such that the arc $[x, y]$ is $q$-symmetric with midpoint $t^j$ and
$L_q(x) = L_q(y)$. Let $u$ and $v$ be $p$-points such that $u \vdash h(x)$ and $v \vdash h(y)$.
By Proposition~\ref{prop:symmetric}, we have $L_p(u) = L_p(v)$. Note that for the midpoint $w$ of the arc
$[u, v]$ we also have $w \vdash h(t^j)$. This implies, by Remark~\ref{rem_basic} (a), that $L_p(w) > L_p(z)$
for every $z \in (u^0, w)$. Therefore, $w$ is a salient $p$-point, \ie $w \in (u^i)_{i \in \None}$.

Let $k, l \in \None$, $k < l$, be such that $u^k \vdash h(t^j)$ and $u^l \vdash h(t^{j+1})$. We want to
prove that $l = k+1$. Let us assume by contradiction that $l > k+1$. Since $L_p(u^{k+1}) > L_p(u^k)$, there
exists a $q_p$-point $x \in (t^j, t^{j+1})$ such that $u^{k+1} \vdash h(x)$. But $x \in (t^j, t^{j+1})$
implies $L_q(x) < L_q(t^j)$, contradicting Proposition~\ref{prop:nonsymmetric}.

In this way we have proved that $h(t^i) \approx u^{i+M}$ for some $M\in \Z$ and every $i \in \None$ with
$M+i > 0$. In an analogous way we can prove that $h(t^{-i}) \approx u^{-i-M'}$ for some $M' \in \Z$ and for
every $i \in \None$ with $M'+i > 0$.
\end{proof}

\begin{theorem}\label{thm:ray}
Every self-homeomorphism $h$ of $\ILcores$ preserves $\Rr$: $h(\Rr) = \Rr$.
\end{theorem}

\begin{proof}
Let $h : \Rr \to \A$, as before. We want to prove that $\A = \Rr$. Note that $h \circ \sigma^i : \Rr \to \A$
and $\sigma^i \circ h : \Rr \to \sigma^i(\A)$ are homeomorphisms for every $i \in \Z$, and
$\sigma^i(\A) = \Rr$ if and only if $\A = \Rr$. By using $h^{-1}$ instead of $h$ if necessary, we can assume
that $M \ge 0$ (with $M$ as in Lemma~\ref{lem:some}). Also, instead of studying $h$, we can study
$\sigma^{1-a} \circ h : \Rr \to \sigma^{1-a}(\A)$, where $a = L_p(u^{1+M})$ (recall that
$h(t^1) \approx u^{1+M}$). Therefore, without loss of generality we can assume that $h(t^1) \approx u^1$ and
$L_p(u^1) = 1$. Recall that $L_q(t^1) = 1$, $L_q(t^{-1}) = 2$ and for every $i \in \None$,
$L_q(t^{-i}) - L_q(t^i) = L_q(t^{i+1}) - L_q(t^{-i}) = 1$. If
$L_p(u^{-i}) - L_p(u^{i}) = L_p(u^{i+1}) - L_p(u^{-i}) = 1$, then $\A = \Rr$ by Lemma~\ref{lem:PR2}.

Recall that $h(t^{-1}) \approx u^{-1-M'}$, where $M'$ is as in Lemma~\ref{lem:some}. Since
$$L_q(t^1) < L_q(t^{-1}) < L_q(t^2) < L_q(t^{-2}) < \cdots,$$ by Proposition~\ref{prop:nonsymmetric} we have
$$1 = L_p(u^1) < L_p(u^{-1-M'}) < L_p(u^2) < L_p(u^{-2-M'}) < \cdots < L_p(u^n) < L_p(u^{-n-M'}) < \cdots.$$
Let $L_p(u^n) = 1 + a_1 + b_1 + \cdots + a_{n-1} + b_{n-1}$ and
$L_p(u^{-n-M'}) = 1 + a_1 + b_1 + \cdots + a_{n-1} + b_{n-1} + a_n$, for every $n \in \None$ and some
$a_1, \dots , a_n, b_1, \dots , b_{n-1} \in \None$. We want to prove that $a_i = b_i = 1$ for every
$i \in \None$.

Assume by contradiction that $k \in \None$ is the smallest integer with $a_i = b_i = 1$ for all
$i < k$ and $a_k > 1$. Then, by Proposition~\ref{prop:nonsymmetric}, there is no salient $p$-point
$u \in (u^i)_{i \in \Z}$ with $L_p(u) = L_p(u^k) + 1$. Thus, Proposition~\ref{prop:symmetric} implies that
$\A$ does not contain any $p$-point with $p$-level $L_p(u^k) + 1$, contradicting that $\A$ is dense in
$\ILcores$ in both directions. If $k \in \None$ is the smallest integer with $a_i = b_i = 1$ for all
$i < k$, $a_k = 1$ and $b_k > 1$, the proof follows in an analogous way.
\end{proof}

\begin{remark}\label{rem:M}
If $h$ is order-preserving, then by proof of Theorem \ref{thm:ray} we have $M' = M$, where $M$ and $M'$ are
as in Lemma~\ref{lem:some}. Also, by Lemma~\ref{lem:PR1}, Lemma~\ref{lem:some} and Theorem~\ref{thm:ray} we
have $L_p(u^{i+M}) = 2(i+M)-1 = (2i-1) + 2M = L_q(t^i) + 2M$ for $i > 0$ and
$L_p(u^{i-M}) = 2(-i+M) = -2i + 2M = L_q(t^i) + 2M$ for $i < 0$. Moreover, by Proposition \ref{prop:symmetric},
for every $q_p$-point $x$, and for the $p$-point $u$ with $u \vdash h(x)$, we have $L_p(u) = L_q(x) + 2M$.
\end{remark}

We finish with the

\begin{proof}[Proof of Theorem~\ref{mainthm}]
Let $1 \leq s \le \sqrt 2 < s' \leq 2$. Then $\ILcores$ is decomposable, $\ILcoresp$ is indecomposable,
and the proof follows.

Since Lemmas 2.1 and 2.2 of \cite{BBS} show how to reduce the case $1 \leq s < s' \leq \sqrt 2$ to the
case $\sqrt 2 < s < s' \leq 2$, it suffices to prove the latter case.

Let $\sqrt 2 < s < s' \leq 2$.
Suppose that there exists a homeomorphism $h : \ILcoresp \to \ILcores$. Let $r':=\frac{s'}{s'+1}$ be
the positive fixed point of $T_{s'}$ and $\rho':= (\dots, r', r', r') \in C_{s'} = \ILcoresp$. Let
$\Rr'$ denote the arc-component containing $\rho'$. Let $r$, $\rho$ and $\Rr$ be the analogous
objects of $C_{s} = \ILcores$, as before. Take $q, p \in \N0$ such that $h(\chain_q) \prec \chain_p$. Let
$(t^i)_{i \in \Z}$ be the sequence of salient $q$-points of $\Rr'$ with $t^0 = \rho'$.
Let $(u^i)_{i \in \Z}$ be the sequence of salient $p$-points of $\Rr$.

Let $f = h^{-1} \circ \sigma \circ h$, and assume by contradiction that $h(\Rr') = \A \ne \Rr$. Since $\Rr$
is the only arc-component in $\ILcores$ that is fixed by $\sigma$, we have $\sigma(\A) \ne \A$ implying
$f(\Rr') \ne \Rr'$. But this contradicts Theorem~\ref{thm:ray}. Therefore $h(\Rr') = \Rr$.

We want to prove that $FP(\Rr') = FP(\Rr)$. Without loss of generality we suppose that $h$ is order-preserving
and that $M > 0$ (with $M$ as in Remark \ref{rem:M}).

{\bf Claim 1:} Let $l \in \None$ and let $x$ be a $q$-point with $L_q(x) = l$. Then
$u := h(x) \in \ell_p^{u^{l+2M}}$ and the arc component $A_u \subset \ell_p^{u^{l+2M}}$ containing $u$, also
contains a $p$-point $y$ such that $L_p(y) = l + 2M$.

Note that Claim 1 is the same as Proposition 4.2 (1) of \cite{BBS}. The proof is analogous:

By Remark \ref{rem:M}, Claim 1 is true for all salient $q$-points and for all $q_p$-points. Note that there
exists $j \in \None$ such that every $q$-point $x \in [t^{-j}, t^j]$ is also a $q_p$-point. Therefore Claim 1
is true for all  $q$-points $x \in [t^{-j}, t^j]$, i.e., for every $q$-point $x \in [t^{-j}, t^j]$ the
arc-component $A_{h(x)}$ containing $h(x)$, also contains a $p$-point $y$ such that $L_p(y) = L_q(x) + 2M$.
Also $h([t^{-j}, t^j]) = [a_{-j},a_j]$, $u^{-j-2M} \in A_{a_{-j}}$ and $u^{j+2M} \in A_{a_j}$. Let $q$-point
$x_1 \in [t^{-j}, t^j]$ be such that the open arc $(x_1,t^{j+1})$ is $q$-symmetric with midpoint $t^j$. Such
$x_1$ exists since $L_q(t^{j+1}) - L_q(t^j) = 2$ and $L_q(t^{-j}) - L_q(t^j) = 1$. Then $h((x_1, t^{j+1}))$
is $p$-link-symmetric with midpoint $u^{j+2M}$. Since there exists a unique $p$-point $b_1$ such that the
open arc $(b_1, u^{j+1+2M})$ is $p$-symmetric with midpoint $u^{j+2M}$, for every $q$-point
$x' \in (t^j, t^{j+1})$ the arc-component $A_{h(x')}$ containing $h(x')$, also contains a $p$-point $y'$ such
that $L_p(y') = L_p(y) = L_q(x) + 2M = L_q(x') + 2M$, see Figure \ref{fig:levels}.

\begin{figure}[ht]
\unitlength=10mm
\begin{picture}(15,6)(0,1)
\put(0,6){$\dots \ t^{-j-1} \dots \dots \ {\bf t^{-j}} \dots \dots \ x_{-1} \dots \ x_1 \dots \ x \ \dots \ {\bf t^j} \dots \ x' \dots \ t^{j+1} \dots$}
\put(3,6.5){$\overbrace{\qquad \qquad \qquad \qquad \qquad \qquad \quad \ }$}
\put(1.8,5.8){$\underbrace{\qquad \qquad \quad \quad \ }_{\textrm{$q$-symmetric}}$}
\put(7.4,5.6){$\underbrace{\qquad \qquad \quad \, }_{\textrm{$q$-symmetric}}$}
\put(6.7,4.7){\vector(1,-2){0.7}}\put(6.6,3.8){$h$}
\put(0,2){$\dots \ u^{-j-1-2M} \dots \dots \ {\bf u^{-j-2M}} \dots \dots \ b_{-1} \ \dots \ b_1 \dots \ y \dots \, {\bf u^{j+2M}} \dots \ y' \dots \, u^{j+1+2M} \dots$}
\put(3.8,2.5){$\overbrace{\qquad \qquad \qquad \qquad \qquad \qquad \qquad \quad }$}
\put(2.5,1.8){$\underbrace{\qquad \qquad \qquad \quad \quad \ \ }_{\textrm{$p$-symmetric}}$}
\put(9,1.65){$\underbrace{\qquad \qquad \quad \quad \ \ }_{\textrm{$p$-symmetric}}$}
\end{picture}
\caption{The configuration of symmetric arcs.
\label{fig:levels}}
\end{figure}
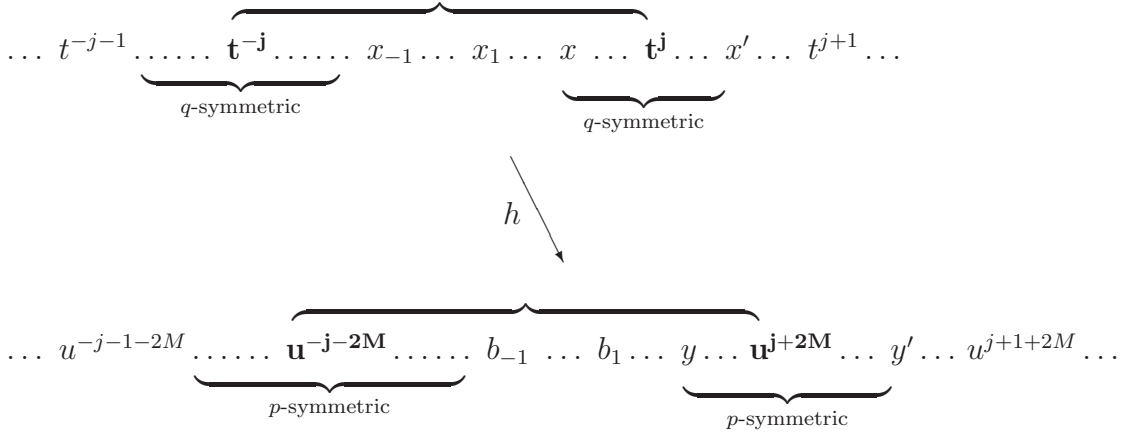

Let us consider now the arc $h([t^{-j-1}, t^{j+1}]) = [a_{-j-1},a_{j+1}]$, $u^{-j-1-2M} \in A_{a_{-j-1}}$ and
$u^{j+1+2M} \in A_{a_{j+1}}$. Let the $q$-point $x_{-1} \in [t^{-j}, t^{j+1}]$ be such that the open arc
$(t^{-j-1},x_{-1})$ is $q$-symmetric with midpoint $t^{-j}$. Such $x_{-1}$ exists since
$L_q(t^{-j-1}) - L_q(t^{-j}) = 2$ and $L_q(t^{j+1}) - L_q(t^{-j}) = 1$. Therefore $h((t^{-j-1},x_{-1}))$ is
$p$-link-symmetric with midpoint $u^{-j-2M}$. Since there exists a unique $p$-point $b_{-1}$ such that the open
arc $(u^{-j-1-2M}, b_{-1})$ is $p$-symmetric with midpoint $u^{-j-2M}$, for every $q$-point
$x'' \in (t^{-j-1}, t^{-j})$ the arc-component $A_{h(x'')}$ containing $h(x'')$, also contains a $p$-point
$y''$ such that $L_p(y'') = L_q(x'') + 2M$, as before. The proof of Claim 1 follows by induction.

{\bf Claim 2:} For $l \in \None_0$ and $i \in \None$, the number of $q$-points in $[t^{-i}, t^i]$ with
$q$-level $l$ is the same as the number of $p$-points in $[u^{-i-2M}, u^{i+2M}]$ with $p$-level $l+2M$.

Claim 2 is the same as Proposition 4.2 (2) of \cite{BBS}. The proof is very similar and we omit it.

Claims 1 and 2 show that
$$
FP_q([t^{-i}, t^i]) = FP_{p+2M}([u^{-i-2M},u^{i+2M}]) = FP_p([u^{-i},u^i]),
$$
for every positive integer $i$, and therefore $FP(\Rr') = FP(\Rr)$.

This proves the Ingram Conjecture for cores of the Fibonacci-like inverse limit spaces.
\end{proof}

\appendix
\section{The Construction of Chains}\label{sec:chains}

We turn now to the technical part, \ie the construction of special  chains that will eventually allow us to show that symmetric arcs map to  symmetric arcs
(see Proposition~\ref{prop:symmetric}).

As mentioned before, we will work with the chains which are
the $\pi_p^{-1}$ images of chains of the interval $[0,s/2]$.
More precisely, we will define a finite collection of points
$G = \{ g_0, g_1, \dots, g_N \} \subset [0, s/2]$ such that
$|g_m - g_{m+1}| \leq s^{-p} \eps/2$ for all $0 \leq m < N$
and $|0-g_0|$ and $|s/2-g_N|$ positive but very small.
>From this one can make a chain $\chain = \{ \ell_n \}_{n = 0}^{2N}$
by setting
\begin{equation}\label{eq:links}
\left\{ \begin{array}{lll}
\ell_{2m+1} = \pi_p^{-1}( (g_m,g_{m+1}) ) & \qquad  & 0 \leq m < N,\\
\ell_{2m} = \pi_p^{-1}( (g_m-\delta, g_m+\delta) \cap [0,s/2] )& & 0 \leq m \leq N,
\end{array} \right.
\end{equation}
where $\min\{ |0-g_0|, |s/2-g_N|\} < \delta \ll \min_m \{ |g_m - g_{m+1}| \}$.
Any chain of this type has links of diameter $< \eps$.

\begin{remark}\label{rem:p_points_in_links}
We could have included all the points $\cup_{j \leq p} T_s^{-j}(c)$
in $G$ to ensure that $T_s^p|_{(g_m, g_{m+1})}$ is monotone for each $m$,
but that is not necessary.
Naturally, there are chains of $\IL$ that are not of this form.

For a component $A$ of $\C0 \cap \ell$, we have the following two possibilities:
\\
\indent (i) $\C0$ goes straight through $\ell$ at $A$, \ie $A$ contains no $p$-point
and $\pi_p(\bd A) = \bd \pi_p(\ell)$; in this case $A$ enters and exits $\ell$ from different sides.
\\
\indent (ii) $\C0$ turns in $\ell$: $A$ contains (an odd number of) $p$-points
$x^0, \dots, x^{2n}$ of which the middle one $x^n$ has the highest $p$-level,
and  $\pi_p(\bd A)$ is a single point in $\bd \pi_p(\ell)$,
in this case $A$ enters and exits $\ell$ from the same side.
\end{remark}

Before giving the details of the $p$-chains we will use, we need a lemma.

\begin{lemma}\label{lem:linkDn}
If the kneading map $Q$ of $T_s$ is eventually non-decreasing and
satisfies Condition \eqref{eq:cond3},
then for all $n \in \None$ there are arbitrarily small numbers $\eta_n > 0$
with the following property:
If $n' > n$ is such that $n \in \orb_\beta(n')$, then either
$|c_{n'} - c_n| > \eta_n$ or
$|c_{n''} - c_n| < \eta_n$ for all $n \leq n'' \leq n'$ with
$n'' \in \orb_\beta(n')$.
\end{lemma}

To clarify what this lemma says, Figure~\ref{fig:toavoid}
shows the configuration of levels $\Dlev_k$ that should be avoided, because then
$\eta_n$ cannot be found.

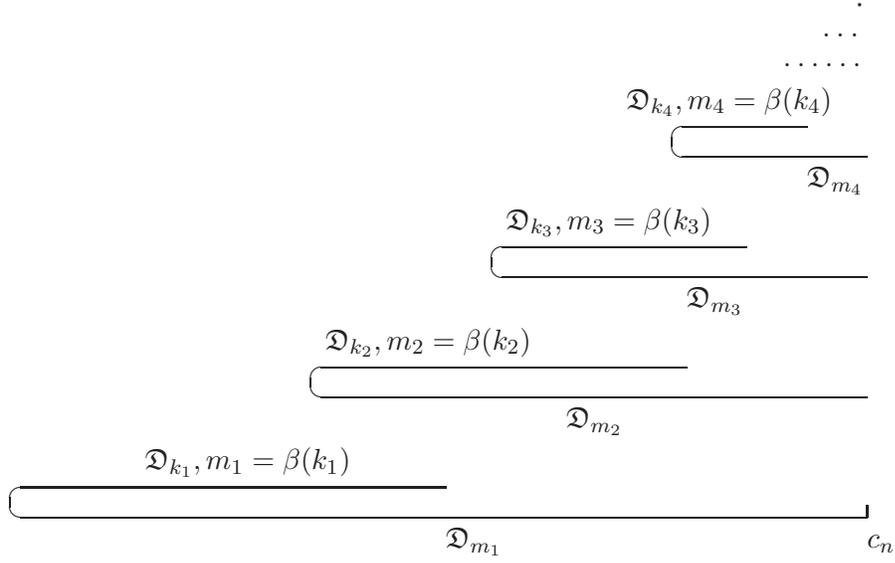
\begin{figure}[ht]
\unitlength=8mm
\begin{picture}(12,9.5)(0,0)
\put(0, 1.25){\oval(0.5,0.5)[l]}
\put(5, 3.25){\oval(0.5,0.5)[l]}
\put(8, 5.25){\oval(0.5,0.5)[l]}
\put(11, 7.25){\oval(0.5,0.5)[l]}
\put(0, 1){\line(1,0){14}}  \put(7, 0.5){\small $\Dlev_{m_1}$}
\put(14, 1){\line(0,1){0.2}} \put(14, 0.5){\small $c_n$}
\put(0, 1.5){\line(1,0){7}}  \put(2, 1.8){\small $\Dlev_{k_1}, m_1 = \beta(k_1)$}
\put(5, 3){\line(1,0){9}}  \put(9, 2.5){\small $\Dlev_{m_2}$}
\put(5, 3.5){\line(1,0){6}}  \put(5, 3.8){\small $\Dlev_{k_2}, m_2 = \beta(k_2)$}
\put(8, 5){\line(1,0){6}}  \put(11, 4.5){\small $\Dlev_{m_3}$}
\put(8, 5.5){\line(1,0){4}}  \put(8, 5.8){\small $\Dlev_{k_3}, m_3 = \beta(k_3)$}
\put(11, 7){\line(1,0){3}}  \put(13, 6.5){\small $\Dlev_{m_4}$}
\put(11, 7.5){\line(1,0){2}}  \put(10, 7.8){\small $\Dlev_{k_4}, m_4 = \beta(k_4)$}
 \put(12.6, 8.5){$\dots\dots$}
 \put(13.25, 9){$\dots$}
 \put(13.8, 9.4){$\cdot$}
\end{picture}
\caption{Linking of levels $\Dlev_{m_i}$ with $\beta(m_1) = \beta(m_2) = \beta(m_3) = \dots = n$. The semi-circles indicates that two intervals have an endpoint in common.}
\label{fig:toavoid}
\end{figure}

\begin{proof}
We will show that the pattern in Figure~\ref{fig:toavoid}
(namely with $c_{m_1} < c_{m_2} < c_{m_3} < \dots$ and $c_{m_{i-1}} < c_{k_i}$
for each $i$) does not continue indefinitely.
To do this, we redraw the first few levels from
Figure~\ref{fig:toavoid}, and discuss
four positions in $\Dlev_{m_1}$ where the precritical point $T_s^{-r}(c)  \in \Dlev_{m_1}$
of lowest order $r$ could be, indicated by points $a_1, \dots, a_4$, see Figure~\ref{fig:toavoid2}.
\begin{figure}[ht]
\unitlength=8mm
\begin{picture}(12,6.2)(0,0)
\put(0, 1.25){\oval(0.5,0.5)[l]}
\put(5, 3.25){\oval(0.5,0.5)[l]}
\put(8, 5.25){\oval(0.5,0.5)[l]}
\put(0, 1){\line(1,0){14}}  \put(4, 0.5){\small $\Dlev_{m_1}$}
\put(14,1){\circle*{0.15}}\put(14, 0.5){\small $c_n$}
\put(-0.25,1.25){\circle*{0.15}}\put(-0.5, 0.5){\small $c_{m_1}$}
\put(7,1.5){\circle*{0.15}}\put(6.8, 1.8){\small $c_{k_1}$}
\put(0, 1.5){\line(1,0){7}}  \put(2, 1.8){\small $\Dlev_{k_1}, m_1 = \beta(k_1)$}
\put(5, 3){\line(1,0){9}}  \put(9, 2.5){\small $\Dlev_{m_2}$}
\put(14,3){\circle*{0.15}}\put(14, 2.5){\small $c_n$}
\put(4.75,3.25){\circle*{0.15}}\put(4.5, 2.5){\small $c_{m_2}$}
\put(11,3.5){\circle*{0.15}}\put(10.8, 3.75){\small $c_{k_2}$}
\put(5, 3.5){\line(1,0){6}}  \put(5.5, 3.8){\small $\Dlev_{k_2}, m_2 = \beta(k_2)$}
\put(8, 5){\line(1,0){6}}  \put(11, 4.5){\small $\Dlev_{m_3}$}
\put(8, 5.5){\line(1,0){4}}  \put(8, 5.8){\small $\Dlev_{k_3}, m_3 = \beta(k_3)$}
\put(14,5){\circle*{0.15}}\put(14, 4.5){\small $c_n$}
\put(7.75,5.25){\circle*{0.15}}\put(7.5, 4.5){\small $c_{m_3}$}
\put(12,5.5){\circle*{0.15}}\put(11.8, 5.75){\small $c_{k_3}$}
\put(1, 0.5){\line(0,1){1.5}} \put(0.9, 0){\small $a_1$}
\put(6, 0.5){\line(0,1){2.7}} \put(5.9, 0){\small $a_2$}
\put(7.5, 0.5){\line(0,1){2.7}} \put(7.4, 0){\small $a_3$}
\put(13, 0.5){\line(0,1){5}} \put(12.9, 0){\small $a_4$}
\end{picture}
\caption{Linking of levels $\Dlev_{m_i}$, $i = 1,2,3$ and three possible positions of
the precritical point $a_j = T^{-r}_s(c) \in \Dlev_{m_1}$ of lowest order $r$.}
\label{fig:toavoid2}
\end{figure}
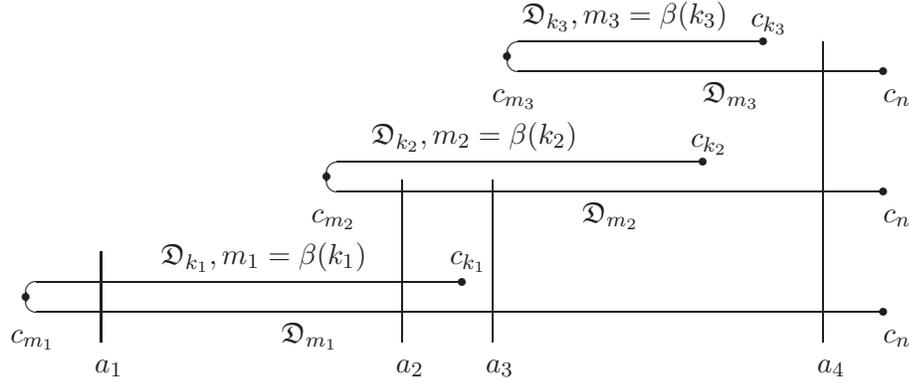

{\bf Case $a_1 \in (c_{m_1}, c_{m_2})$:}
Take the $r+1$-th iterate of the picture, which moves
$\Dlev_{m_1}$ and $\Dlev_{k_1}$ to levels with lower endpoint
$c_1$. then we can repeat the argument, until we
arrive in one of the cases below.

{\bf Case $a_2 \in (c_{m_2}, c_{k_1})$:}
Take the $r$-th iterate of the picture, which moves
$\Dlev_{m_1}$, $\Dlev_{k_1}$, $\Dlev_{m_2}$ and $\Dlev_{k_2}$ all to
cutting levels
and $c_{r+k_2} \in (c, c_{r+k_3})$.
But $m_2 > m_1$, whence $k_2 > k_1$, and this contradicts that
$|c_{S_{k_2}} - c| < |c_{S_{k_1}} - c|$.
(If $a_2 \in (c_{m_3}, c_{k_2})$, then the same argument would give
that $r+k_2 < r+k_3$ are both cutting times,
but $|c-c_{r+k_2}| < |c-c_{r+k_3}|$.)

{\bf Case $a_3 \in (c_{k_1}, c_{m_3})$:} Take the  $r$-th iterate of the picture, which moves
$\Dlev_{m_1}$, $\Dlev_{m_2}$ and $\Dlev_{k_2}$ to cutting levels, and $\Dlev_{m_3}$ to a
non-cutting level $\Dlev_u$ with $u := m_3+r$ such that
\[
 S_j := n+r =  \beta(u) = \beta(m_2+r) = \beta^2(k_2+r).
\]
The integer $u$ such that $c_u$ is closest to $c$ is for
$u = S_i + S_j$ where $j$ is minimal such that $Q(i+1) > i$,
and in this case, the itineraries of $T_s(c)$ and $T_s(c_u)$ agree for
at most $S_{Q^2(i+1)+1}-1$ iterates (if $Q(i+1) = j+1$)
or at most $S_{Q(j+1)}-1$ iterates (if $Q(i+1) > j+1$).
Call $S_h := k_2+r$, then $j = Q^2(h)$ and the itineraries
of $T_s(c_{S_h})$ and $c$ agree up to $S_{Q(h+1)}-1$ iterates.
By assumption \eqref{eq:cond3}, we have
\[
Q(j+1) \leq Q^2(i+1)+1 = Q(j+1)+1 = Q(Q^2(h)+1)+1 < Q(h+1),
\]
but this means that $\Dlev_u$ and $\Dlev_{S_h}$ cannot overlap, a contradiction.

{\bf Case $a_4 \in (c_{k_2}, c_n)$:} Then take the $r+1$-st iterate of the picture, which has
the same structure, with $c_n$ replaced by $T_s^{r+1}(a_1) = c_1$.
Repeating this argument, we will eventually arrive at Case $a_2$ or
$a_3$ above.

Therefore we can find $\eta_n$ such that $c_n - \eta_n$ separates
$c_n$ from all levels $\Dlev_{k_i}$, $\beta^2(k_i) = n$
that intersect $\Dlev_{m_1}$. Indeed, in Case $a_2$, we place
$c_n-\eta_n$ just to the right of $c_{k_1}$ and in Case $a_3$ (and hence
$c_{k_1} \in \Dlev_{k_2}$), we place
$c_n-\eta_n$ just to the right of $c_{k_2}$.
\end{proof}
\begin{proposition}\label{prop:chains}
Under the assumption of Lemma~\ref{lem:linkDn}, given $\eps > 0$,
there exists $p \in \None$ and a chain $\chain = \chain_p$ of $\IL$
with the following properties:
\begin{enumerate}
\item The links of $\chain$ have diameter $< \eps$.
\item For each $n \in \None$, there is exactly one link $\ell \in {\chain}$
such that every $x \in \IL$ that $p$-turns at $c_n$ belongs to $\ell$.
\item If $y \in \ell$ is a $p$-point not having the lowest
  $p$-level of $p$-points in $\ell$, then
both $\beta$-neighbors of $y$ belong to $\ell$.
\item If $y \nin \ell$ is a $\beta$-neighbor of $x$ above, then the
other $\beta$-neighbor of $y$ either lies outside $\ell$, or has $p$-level $n$
as well.
\end{enumerate}
\end{proposition}
\begin{proof}
We will construct the chain $\chain$ as outlined in the beginning
of this section, see \eqref{eq:links}.
So let us specify the collection $G$ by starting with at least
$\lceil 2s^p/\eps \rceil$ approximately equidistant points $g_m \in [0,s/2]$ so that no $g_m$ lies
on the critical orbit,
and then refining this collection inductively
to satisfy parts 2.-4.\ of the proposition.

Start the induction with $n = 1$, \ie the point $c_1$.
Note that $c_1 \notin G$, so there will be only one link
$\ell \in \chain$ with $c_1 \in \pi_p(\ell)$.
Let $\eta_1 \in (0, s^{-p} \eps/2)$ be as in Lemma~\ref{lem:linkDn}.
Then, since each $k$ contains $1$ in its $\beta$-orbit, each $\Dlev_k$
intersecting $(c_1-\eta_1, c_1]$ is either contained in $(c_1-\eta_1, c_1]$
or has $c_1$ as lower endpoint (\ie $\beta(k) = 1$).
In the latter case, also $\Dlev_l \cap (c_1-\eta_1, c_1] = \emptyset$
for each $l$ with $\beta(l) = k$.
Hence by inserting $c_1-\eta_1$ into $G$,
we can refine the chain $\chain$ so that properties 3.\ and 4.\ holds for
the link $\ell$ with $\pi_p(\ell) \owns c_1$.

Suppose we have refined the chain to accommodate links $\ell$
such that $\pi_p(\ell) \owns c_i$ for each $i < n$.
Then $c_n$ does not belong to the set $G$ created so far, so there will be
only one link $\ell \in {\chain}$ with $\pi_p(\ell) \owns c_n$.
Again, find $\eta_n\in (0, s^{-p}\eps /2)$ as in Lemma~\ref{lem:linkDn}
and extend $G$ with $c_n + \eta_n$ if $c_n$ is a local minimum
of $T_s^n$ or with $c_n - \eta_n$ if $c_n$ is a local minimum
of $T_s^n$.

We skip the induction step if $\Dlev_n$ already belongs to complementary
interval to
$G$ extended with all point $c_i \pm \eta_i$ created so far.
Since $|\Dlev_n| \to 0$, the induction will eventually cease altogether,
and then the required set $G$ is found.
\end{proof}

\section{Symmetric and Quasi-Symmetric Arcs} \label{sec:furtherlemmas}

From now on all chains $\chain_p$ are as in Proposition~\ref{prop:chains}. Also, we assume that
the slope $s$ is such that $T_s$ is Fibonacci-like and we abbreviate $T := T_s$.

\noindent Suppose $A = [u, v] \subset \A$ is a quasi-$p$-symmetric arc with $u, v \in \ell$, and let $A_u$ and
$A_v$ be arc-components of $\ell$ that contain $u$ and $v$ respectively. We will sometimes say, for simplicity,
that the arc $[A_u, A_v]$ between $A_u$ and $A_v$, including $A_u$ and $A_v$, is quasi-$p$-symmetric.
\begin{defi}\label{basic-quasi-p-symmetric}
A quasi-$p$-symmetric arc $A = [u, v]$ with midpoint $m$ is called {\em basic} if there is no $p$-point
$w \in (u, v)$ such that either $[u, w] \subset [u,m]$ or $[w, v] \subset [m,v]$ is a quasi-$p$-symmetric arc.
\end{defi}
\begin{ex} Let us consider the Fibonacci map and the corresponding inverse limit space. Then the arc-component
$\C0$ (as well as an arc-component $\A$) contains the arc $A = [x^0, x^{33}]$ such that the folding pattern  of
$A$ is as follows (see Figure~\ref{fig:example}):
\begin{equation}\label{eq:fold}
 27 \ 6 \ \overbrace{\underbrace{{\bf 1}_2 \ {\bf 14}_3 \ {\bf 1} \ 6 \ {\bf 1}_6}_{\textrm{basic}} \ 0 \ 3 \ 0 \ 1 \ 0 \ 2 \ 0 \ 1 \ 4 \ 1 \ {\bf 9} \ 1 \ 4 \ 1 \ 0 \ 2 \ 0 \ 1 \ 0 \ 3 \ 0 \, \underbrace{1 \ 6 \ {\bf 1}_{30}}_{\textrm{sym}}}^{\textrm{quasi-$p$-symmetric}} \ 0 \ 3 \ 0
\end{equation}
(for easier orientation we write sometimes for example $1_2$ which means that the $p$-level 1 belongs to
the $p$-point $x^2$).
We can choose a chain $\chain_p$ such that $p$-points with $p$-levels 1 and 14 belong to the same link.
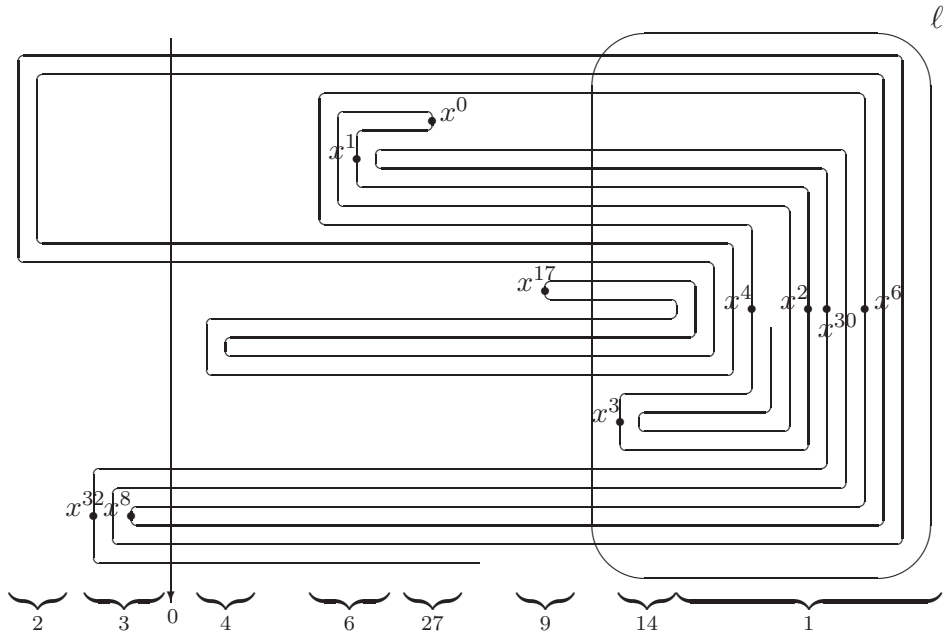
\begin{figure}[ht]
\unitlength=5mm
\begin{picture}(25,16)(-1,9)
\put(1, 24){\oval(0.5,0.5)[tl]}
  \put(1, 24.25){\line(1,0){23}} \put(0.75, 24){\line(0,-1){5}}
\put(1.5, 23.5){\oval(0.5,0.5)[tl]}
  \put(1.5, 23.75){\line(1,0){22}} \put(1.25, 23.5){\line(0,-1){4}}
\put(9, 23){\oval(0.5,0.5)[tl]}
  \put(9, 23.25){\line(1,0){14}} \put(8.75, 23){\line(0,-1){3}}
\put(9.5, 22.5){\oval(0.5,0.5)[tl]}
  \put(9.5, 22.75){\line(1,0){2}} \put(9.25, 22.5){\line(0,-1){2}}
\put(10, 22){\oval(0.5,0.5)[tl]}
  \put(10, 22.25){\line(1,0){1.5}} \put(9.75, 22){\line(0,-1){1}}
    \put(11.5, 22.5){\oval(0.5,0.5)[r]}
\put(10.5, 21.5){\oval(0.5,0.5)[l]}
  \put(10.5, 21.75){\line(1,0){12}}
  \put(10.5, 21.25){\line(1,0){11.5}}
\put(10, 21){\oval(0.5,0.5)[bl]}
  \put(10, 20.75){\line(1,0){11.5}}
\put(9.5, 20.5){\oval(0.5,0.5)[bl]}
  \put(9.5, 20.25){\line(1,0){11.5}}
\put(9, 20){\oval(0.5,0.5)[bl]}
  \put(9, 19.75){\line(1,0){11}}
\put(1.5, 19.5){\oval(0.5,0.5)[bl]}
  \put(1.5, 19.25){\line(1,0){18}}
\put(1, 19){\oval(0.5,0.5)[bl]}
  \put(1, 18.75){\line(1,0){18}}
\put(15, 18){\oval(0.5,0.5)[l]}
  \put(15, 18.25){\line(1,0){3.5}}
   \put(15, 17.75){\line(1,0){3}}
\put(6, 17){\oval(0.5,0.5)[tl]}
  \put(6, 17.25){\line(1,0){12}}  \put(5.75, 17){\line(0,-1){1}}
\put(6.5, 16.5){\oval(0.5,0.5)[l]}
  \put(6.5, 16.75){\line(1,0){12}}
    \put(6.5, 16.25){\line(1,0){12.5}}
\put(6, 16){\oval(0.5,0.5)[bl]}
  \put(6, 15.75){\line(1,0){13.5}}
\put(17, 15){\oval(0.5,0.5)[tl]}
  \put(17, 15.25){\line(1,0){3}}  \put(16.75, 15){\line(0,-1){1}}
\put(17.5, 14.5){\oval(0.5,0.5)[l]}
  \put(17.5, 14.75){\line(1,0){3}}
    \put(17.5, 14.25){\line(1,0){3.5}}
\put(17, 14){\oval(0.5,0.5)[bl]}
  \put(17, 13.75){\line(1,0){4.5}}
\put(3, 13){\oval(0.5,0.5)[tl]}
  \put(3, 13.25){\line(1,0){19}}  \put(2.75, 13){\line(0,-1){2}}
\put(3.5, 12.5){\oval(0.5,0.5)[tl]}
  \put(3.5, 12.75){\line(1,0){19}}  \put(3.25, 12.5){\line(0,-1){1}}
\put(4, 12){\oval(0.5,0.5)[l]}
  \put(4, 12.25){\line(1,0){19}}
    \put(4, 11.75){\line(1,0){19.5}}
\put(3.5, 11.5){\oval(0.5,0.5)[bl]}
  \put(3.5, 11.25){\line(1,0){20.5}}
\put(3, 11){\oval(0.5,0.5)[bl]}
  \put(3, 10.75){\line(1,0){10}}
\put(18, 17.5){\oval(0.5,0.5)[r]}
\put(18.5, 18){\oval(0.5,0.5)[tr]}
  \put(18.75, 18){\line(0,-1){1}}
\put(18.5, 17){\oval(0.5,0.5)[br]}
\put(19, 18.5){\oval(0.5,0.5)[tr]}
  \put(19.25, 18.5){\line(0,-1){2}}
\put(19, 16.5){\oval(0.5,0.5)[br]}
\put(19.5, 19){\oval(0.5,0.5)[tr]}
  \put(19.75, 19){\line(0,-1){3}}
\put(19.5, 16){\oval(0.5,0.5)[br]}
\put(20, 19.5){\oval(0.5,0.5)[tr]}
  \put(20.25, 19.5){\line(0,-1){4}}
\put(20, 15.5){\oval(0.5,0.5)[br]}
  \put(20.75, 17){\line(0,-1){2}}
\put(20.5, 15){\oval(0.5,0.5)[br]}
\put(21, 20){\oval(0.5,0.5)[tr]}
  \put(21.25, 20){\line(0,-1){5.5}}
\put(21, 14.5){\oval(0.5,0.5)[br]}
\put(21.5, 20.5){\oval(0.5,0.5)[tr]}
  \put(21.75, 20.5){\line(0,-1){6.5}}
\put(21.5, 14){\oval(0.5,0.5)[br]}
\put(22, 21){\oval(0.5,0.5)[tr]}
  \put(22.25, 21){\line(0,-1){7.5}}
\put(22, 13.5){\oval(0.5,0.5)[br]}
\put(22.5, 21.5){\oval(0.5,0.5)[tr]}
  \put(22.75, 21.5){\line(0,-1){8.5}}
\put(22.5, 13){\oval(0.5,0.5)[br]}
\put(23, 23){\oval(0.5,0.5)[tr]}
  \put(23.25, 23){\line(0,-1){10.5}}
\put(23, 12.5){\oval(0.5,0.5)[br]}
\put(23.5, 23.5){\oval(0.5,0.5)[tr]}
  \put(23.75, 23.5){\line(0,-1){11.5}}
\put(23.5, 12){\oval(0.5,0.5)[br]}
\put(24, 24){\oval(0.5,0.5)[tr]}
  \put(24.25, 24){\line(0,-1){12.5}}
\put(24, 11.5){\oval(0.5,0.5)[br]}
\put(18.25, 10){$\underbrace{\hspace{100pt}}_1$}
\put(16.7, 10){$\underbrace{\hspace{20pt}}_{14}$}
\put(14, 10){$\underbrace{\hspace{0pt}}_{9}$}
\put(11, 10){$\underbrace{\hspace{0pt}}_{27}$}
\put(8.5, 10){$\underbrace{\hspace{30pt}}_{6}$}
\put(5.5, 10){$\underbrace{\hspace{2pt}}_{4}$}
\put(4.8, 24.7){\vector(0,-1){15}}
\put(4.7, 9.28){${\hspace{0pt}}_{0}$}
\put(2.5, 10){$\underbrace{\hspace{30pt}}_{3}$}
\put(0.5, 10){$\underbrace{\hspace{2pt}}_{2}$}
\put(20.5, 17.6){\oval(9,14.5)} \put(25, 25){$\ell$}
\put(14.75, 18){\circle*{0.2}}  \put(14, 18){\small $x^{17}$}
\put(11.75, 22.5){\circle*{0.2}}  \put(11.95, 22.5){\small $x^0$}
\put(9.75, 21.5){\circle*{0.2}}  \put(9, 21.5){\small $x^1$}
\put(16.75, 14.5){\circle*{0.2}}  \put(16, 14.5){\small $x^3$}
\put(20.25, 17.5){\circle*{0.2}}  \put(19.5, 17.5){\small $x^4$}
\put(21.75, 17.5){\circle*{0.2}}  \put(21, 17.5){\small $x^2$}
\put(22.25, 17.5){\circle*{0.2}}  \put(22, 16.75){\small $x^{30}$}
\put(23.25, 17.5){\circle*{0.2}}  \put(23.5, 17.5){\small $x^6$}
\put(3.75, 12){\circle*{0.2}} \put(3, 12){\small $x^8$}
\put(2.75, 12){\circle*{0.2}} \put(2, 12){\small $x^{32}$}
\end{picture}
\caption{The arc $A$ with folding pattern as in \eqref{eq:fold},
with $p$-points of $p$-level $1$ and $14$ in a single link $\ell$.}
\label{fig:example}
\end{figure}
The arc $[x^2,x^6]$ with the folding pattern $1\ 14\ 1\ 6\ 1$ is a basic quasi-$p$-symmetric arc; the arc
$[x^2,x^{30}]$ with the folding pattern as in \eqref{eq:fold} under the wide brace is also a
quasi-$p$-symmetric but not basic, because it contains $[x^2,x^6]$. Notice also that the arc $[x^3, x^{30}]$
is a quasi-$p$-symmetric arc for which Proposition~\ref{lem:until_lm} and Proposition~\ref{lem:symarc}
do not work (see the folding patterns to the left of $[x^3, x^{30}]$ and to the right of $[x^3, x^{30}]$).
\end{ex}
\begin{lemma}\label{prop2}\label{lem:maxL}
Let $\chain_p$ be a chain and $[x,y]$ a quasi-$p$-symmetric arc with respect to this chain (not contained
in a single link) with midpoint $m$ and such that $L_p(x) \geq L_p(m)$. Let $A_x$ be the link-tip of
$[x,y]$ which contains $x$. Then $L_p(m) > L_p(z)$ for all $p$-points
$z \in [x,y] \setminus (\{ m \} \cup A_x)$.
\end{lemma}
\begin{proof}
Let $A = [a,b] \owns m$ be the smallest arc with $p$-points $a,b$ of higher $p$-level than $L_p(m)$, say
$m \in [a,b]$ and $L_p(m) \leq L_p(a) \leq L_p(b)$. By part (a) of Remark~\ref{rem_basic} we obtain
$L := L_p(m) < L_p(a) < L_p(b)$. Since $L_p(x) \geq L_p(m)$, $[x,m]$ contains one endpoint of $A$. We can
assume that $[x,m] \setminus A$ is contained in a single link, because otherwise $[x,y] \setminus \ell$-tips
is not $p$-symmetric. If $[y,m]$ does not contain the other endpoint of $A$, then the statement is proved.

Let us now assume by contradiction that $A \subset [x,y]$. Again, we can assume that $[y,m] \setminus A$ is
contained in a single link, because otherwise $[x,y] \setminus \ell$-tips is not $p$-symmetric. By part (a)
of Remark~\ref{rem_basic} once more we have $\pi_{p+L}([a,b]) = [c_{S_l}, c_{S_k}] \owns c = \pi_{p+L}(m)$
for some $k$ and $l = Q(k)$, and $|\pi_{p+L}(a) - c| > |\pi_{p+L}(b) - c|$, see the top line of
Figure~\ref{fig:xy}. It follows that $[a,b]$ contains a symmetric open arc $(b',b)$ where $b' \in (a,b)$ is
the unique point such that $T(\pi_{p+L}(b')) = T(\pi_{p+L}(b))$. Since $[x,y] \setminus \ell$-tips is
$p$-symmetric, $L_p(b) > L_p(m)$ implies $b, b' \in \ell$-tips. Moreover, the arc $[a,b']$ is contained in
the same link $\ell$ as $b$.

If $k$ and $l$ are relatively small, then $\pi^{-1}_{p}(c_{S_l})$ and $\pi^{-1}_{p}(c_{S_k})$ belong to
different links of $\chain_p$, so we can assume that they are so large that we can apply
Condition \eqref{eq:cond4}.
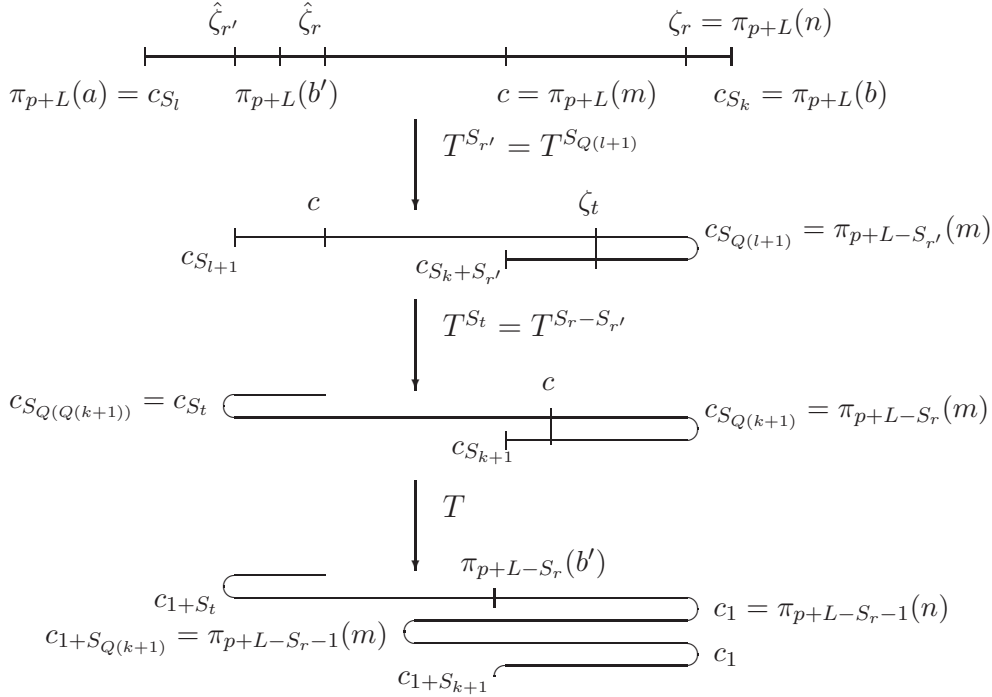
\begin{figure}[ht]
\unitlength=12mm
\begin{picture}(12,8.5)(4.2,-0.5)
\thinlines
\put(6,7){\line(1,0){6.5}}
\put(6,6.9){\line(0,1){0.2}}\put(4.5, 6.5){\small $\pi_{p+L}(a) = c_{S_l}$}
\put(12.5,6.9){\line(0,1){0.2}}\put(12.3, 6.5){\small $c_{S_k}=\pi_{p+L}(b)$}
\put(10,6.9){\line(0,1){0.2}}\put(9.9, 6.5){\small $c = \pi_{p+L}(m)$}
\put(7,6.9){\line(0,1){0.2}}\put(6.7, 7.3){\small $\hat \zeta_{r'}$}
\put(7.5,6.9){\line(0,1){0.2}}\put(7, 6.5){\small $\pi_{p+L}(b')$}
\put(8,6.9){\line(0,1){0.2}}\put(7.7, 7.3){\small $\hat \zeta_{r}$}
\put(12,6.9){\line(0,1){0.2}}\put(11.8, 7.3){\small $\zeta_{r} = \pi_{p+L}(n)$}
\put(9,6.3){\vector(0,-1){1}}\put(9.3, 5.9){$T^{S_{r'}} = T^{S_{Q(l+1)}}$}
\put(7,5){\line(1,0){5}}
\put(7,4.9){\line(0,1){0.2}}\put(6.4, 4.7){\small $c_{S_{l+1}}$}
\put(12.2, 5){\small $c_{S_{Q(l+1)}} = \pi_{p+L-S_{r'}}(m)$}
\put(8,4.9){\line(0,1){0.2}} \put(7.8, 5.3){\small $c$}
\put(11,4.65){\line(0,1){0.45}}\put(10.8, 5.3){\small $\zeta_t$}
\put(12,4.875){\oval(0.25, 0.25)[r]}
\put(10,4.75){\line(1,0){2}}
\put(10,4.65){\line(0,1){0.2}}\put(9, 4.6){\small $c_{S_k + S_{r'}}$}
\put(9,4.3){\vector(0,-1){1}}\put(9.3, 3.9){$T^{S_t} = T^{S_r-S_{r'}}$}
\put(7,3){\line(1,0){5}}
\put(7,3.125){\oval(0.25, 0.25)[l]}\put(7,3.25){\line(1,0){1}}
\put(4.5, 3.1){\small $c_{S_{Q(Q(k+1))}} = c_{S_t}$}
\put(12.2, 3){\small $c_{S_{Q(k+1)}}  = \pi_{p+L-S_{r}}(m)$}
\put(10.5,2.7){\line(0,1){0.4}}\put(10.4, 3.3){\small $c$}
\put(12,2.875){\oval(0.25, 0.25)[r]}
\put(10,2.75){\line(1,0){2}}
\put(10,2.65){\line(0,1){0.2}}\put(9.4, 2.6){\small $c_{S_{k+1}}$}
\put(9,2.3){\vector(0,-1){1}}\put(9.3, 1.9){$T$}
\put(7,1){\line(1,0){5}}
\put(7,1.125){\oval(0.25, 0.25)[l]}\put(7,1.25){\line(1,0){1}}
\put(6.1, 0.9){\small $c_{1+S_t}$}
\put(12.3, 0.8){\small $c_1  = \pi_{p+L-S_r-1}(n)$}
\put(12,0.875){\oval(0.25, 0.25)[r]}
\put(9,0.75){\line(1,0){3}}
\put(9,0.625){\oval(0.25, 0.25)[l]}\put(4.9, 0.5){\small $c_{1+S_{Q(k+1)}} = \pi_{p+L-S_{r}-1}(m)$}
\put(9,0.5){\line(1,0){3}}
\put(12,0.375){\oval(0.25, 0.25)[r]} \put(12.3, 0.3){\small $c_1$}
\put(10,0.25){\line(1,0){2}}
\put(10,0.125){\oval(0.25, 0.25)[tl]}\put(8.8, 0.05){\small $c_{1+S_{k+1}}$}
\put(9.875,0.9){\line(0,1){0.2}}\put(9.5, 1.3){\small $\pi_{p+L-S_r}(b')$}
\end{picture}
\caption{The image of $\pi_{p+L}([x,y]) \owns c = \pi_{p+L}(m)$
under appropriate iterates of $T$.
}
\label{fig:xy}
\end{figure}
Let $r = Q(k+1)$ and $r' = Q(l+1)$ be the lowest indices such that the closest precritical points
$\hat \zeta_{r'} \in [c_{S_l}, c]$ and $\zeta_r \in [c, c_{S_k}]$. By \eqref{eq:cond4},
$r'= Q(l+1) = Q(Q(k)+1) < Q(k+1) = r$. Consider the image of $[c_{S_l}, c_{S_k}]$ first under $T^{S_{r'}}$
and then under $T^{S_r}$ (second and third level in Figure~\ref{fig:xy}).

By the choice of $r$, we obtain $\pi_{p+L-S_r}([m,b]) = [ c_{S_{k+1}}, c_{S_{Q(k+1)}}]$, and
$\pi_{p+L-S_r}([a,b']) \owns c_{S_t}$ for $t = Q(Q(k+1))$. As in \eqref{eq:cond4twice},
$|c_{S_t} - c| > |c_{S_{Q(k+1)}} - c| > |c_{S_{k+1}} - c|$, and taking one more iterate, we see that
$[c_{1+S_{k+1}}, c_1] \subset [c_{1+S_{Q(k+1)}}, c_1] \subset [1+c_{S_t}, c_1]$ (last level in
Figure~\ref{fig:xy}).

Let $n \in [m,b]$ be such that $\pi_{p+L}(n) = \zeta_r$, see the first level in  Figure~\ref{fig:xy}.
Since $[a,b']$ belongs to a single link $\ell \in \chain_p$, $m \in \ell$ as well. Suppose that $[a,m]$
is not contained in $\ell$. Then there is a maximal symmetric arc $[d', d]$ with midpoint $n$ such that
the points $d, d' \notin \ell$. Then the arcs $[d', a]$ and $[d, m]$ both enter the same link $\ell$ but
they have different `first' turning levels in $\ell$, contradicting the properties of $\chain_p$ from
Proposition~\ref{prop:chains}.

This shows that $[a,m] \subset \ell$. In the beginning of the proof we argued that the components of
$[x,y] \setminus A$ belong to the same link, so that means that the entire arc $[x,y]$ is contained in a
single link, contradicting the assumptions of the proposition. This concludes
the proof.
\end{proof}
\begin{remark}
In fact, this proof shows that the $p$-point $b \in \partial A$ of the highest $p$-level belongs to $[m,x]$.
Indeed, if $a \in [m,x]$, then because $[m,b]$ has shorter arc-length than $[m,a]$, either $a$ and $b$,
and therefore $x$ and $y$ do not belong to the same link $\ell$ (whence $[x,y]$ is not quasi-$p$-symmetric),
or the arc $[a,b]$ itself is quasi-$p$-symmetric and contradicts Lemma~\ref{prop2}.
\end{remark}

\begin{corollary}\label{cor:qusi_symmetry_type}
Let $A = [x,y] \subset \A$ be a quasi-$p$-symmetric arc with midpoint $m$. Let $A_x$, $A_y$ be the
link-tips of $A$ containing $x$ and $y$ respectively. If $x$ is the midpoint of $A_x$, and $y$ is the
midpoint of $A_y$, then either $L_p(x) > L_p(m) > L_p(y)$, or $L_p(x) < L_p(m) < L_p(y)$.
\end{corollary}
\begin{remark}
Note that in general there are quasi-$p$-symmetric arcs $[x,y]$ with midpoint $m$ such that
$L_p(x) > L_p(y) > L_p(m)$. For example, if a tent map $T_s$ has a preperiodic critical point, then for every
quasi-$p$-symmetric arcs $[x,y]$ with midpoint $m$ either $L_p(x) > L_p(y) > L_p(m)$, or
$L_p(y) > L_p(x) > L_p(m)$.
\end{remark}

\begin{corollary}\label{cor:spiral_linktips}
Let $[x,y] \subset \A$ be a quasi-$p$-symmetric arc with midpoint $m$, not contained in a single link,
such that $L_p(x) > L_p(m) > L_p(y)$. If $[m,x]$ is longer than $[y,m]$ measured in arc-length, then there
exists a $p$-point $y' \in A_x$ such that $[y,y']$ is $p$-symmetric.
\end{corollary}
\begin{proof}
As in the previous proof, $b \in [x,m]$ and $y \in [m,b']$ and take $y' \in [m,b]$ such that
$\pi_{p+L}(y') = \pi_{p+L}(y)$.
\end{proof}
\begin{remark}
If $A_x \ni x$ and $A_y \ni y$ are maximal arc-components of $\A \cap \ell$ (with still
$L_p(x) > L_p(m) > L_p(y)$), and $m_y$ is the midpoint of $A_y$, then there is $y' \in A_x$
such that $[y', m_y]$ is $p$-symmetric.

In other words, when $\A$ enters and turns in a link $\ell$, then it folds in a symmetric
pattern, say with levels $L_1, L_2,\dots, L_{m-1}, L_m, L_{m-1}, \dots, L_2, L_1$. The nature of the
chain $\chain_p$ is such that $L_1$ depends only on $\ell$. The Corollary~\ref{cor:spiral_linktips}
does not say that the rest of the pattern is the same also, but only that if $A \subset \A$ is such
that $A \setminus \ell\mbox{-tips}$ is $p$-symmetric, then the folding pattern at the one link-tip is
a subpattern (stopping at a lower center level) of the folding pattern at the other link-tip.
\end{remark}
\begin{proposition}[Extending a quasi-$p$-symmetric arc at its higher level endpoint] \label{lem:symarc}
Let $A = [x,y] \subset \A$ be a basic quasi-$p$-symmetric arc, not contained in a single link,
such that the $p$-points $x, y \in \ell$ are the midpoints of the link-tips of $A$ and $L_p(x) > L_p(y)$.
Let $m$ be the midpoint of $A$. Then there exists a $p$-point $m'$ such that the arc $[m, m']$ is
(quasi-)$p$-symmetric with $x$ as its midpoint.
\end{proposition}
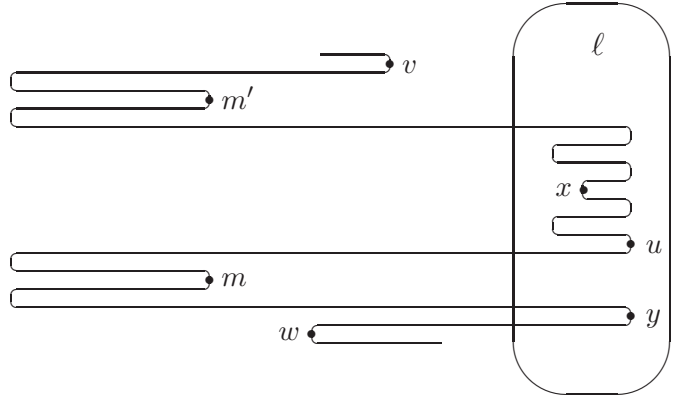
\begin{figure}[ht]
\unitlength=8mm
\begin{picture}(12,6.5)(2,1)
\put(12.5,4.6){\oval(2.6,6.5)}
\put(12.5, 7){$\ell$}
\put(8,7){\line(1,0){1}}
\put(9,6.85){\oval(0.3,0.3)[r]}
\put(9.35, 6.7){\small $v$}\put(9.15, 6.85){\circle*{0.15}}
\put(3,6.7){\line(1,0){6}}
\put(3,6.55){\oval(0.3,0.3)[l]}
\put(6,6.25){\oval(0.3,0.3)[r]}
\put(6.35, 6.1){\small $m'$}\put(6.15, 6.25){\circle*{0.15}}
\put(3,6.4){\line(1,0){3}}
\put(3,6.1){\line(1,0){3}}
\put(3,5.95){\oval(0.3,0.3)[l]}
\put(3,5.8){\line(1,0){10}}
\put(13,5.65){\oval(0.3,0.3)[r]}
\put(12,5.5){\line(1,0){1}}
\put(12,5.35){\oval(0.3,0.3)[l]}
\put(12,5.2){\line(1,0){1}}
\put(13,5.05){\oval(0.3,0.3)[r]}
\put(12.5,4.9){\line(1,0){0.5}}
\put(12.5,4.75){\oval(0.3,0.3)[l]}
\put(11.9, 4.65){\small $x$}\put(12.37, 4.75){\circle*{0.15}}
\put(12.5,4.6){\line(1,0){0.5}}
\put(13,4.45){\oval(0.3,0.3)[r]}
\put(12,4.3){\line(1,0){1}}
\put(12,4.15){\oval(0.3,0.3)[l]}
\put(12,4){\line(1,0){1}}
\put(13,3.85){\oval(0.3,0.3)[r]}
\put(13.4, 3.7){\small $u$}\put(13.15, 3.85){\circle*{0.15}}
\put(3,3.7){\line(1,0){10}}
\put(3,3.55){\oval(0.3,0.3)[l]}
\put(3,3.4){\line(1,0){3}}
\put(6,3.25){\oval(0.3,0.3)[r]}
\put(6.35, 3.15){\small $m$}\put(6.15, 3.25){\circle*{0.15}}
\put(3,3.1){\line(1,0){3}}
\put(3, 2.95){\oval(0.3,0.3)[l]}
\put(3,2.8){\line(1,0){10}}
\put(13,2.65){\oval(0.3,0.3)[r]}
\put(13.4, 2.55){\small $y$}\put(13.15, 2.65){\circle*{0.15}}
\put(8,2.5){\line(1,0){5}}
\put(8,2.2){\line(1,0){2}}
\put(8,2.35){\oval(0.3,0.3)[l]}
\put(7.3, 2.25){\small $w$}\put(7.85, 2.35){\circle*{0.15}}
\end{picture}
\caption{The configuration in Proposition~\ref{lem:symarc} where the existence of $p$-point $m'$ is proved. $v$
is the first $p$-point 'beyond' $x$ such that $L_p(v) > L_p(x)$ and $u$ is such that $[u,y]$ is $p$-symmetric
with midpoint $m$.}
\label{fig:symarc1}
\end{figure}
\begin{remark}
The conditions are all crucial in this lemma:
\begin{itemize}
\item[$(a)$]
It is important that $y$ is a $p$-point. Otherwise, if $\A$ goes straight through $\ell$ at $y$, then it is
possible that $x$ is the single $p$-point in $A_x$ (where $A_x$ is the arc-component of $\A \cap \ell$
containing $x$) and $[v,x]$ is shorter than $[x,m]$, and the lemma would fail.
\item[$(b)$]
Without the assumption that $[x,y]$ is basic the lemma can fail. If Figure~\ref{fig:example} the
quasi-$p$-symmetric arc $[x,y] = [x^3,x^{30}]$ is not basic, and indeed there is no $p$-point
$m' \in [x,v] = [x^3,x^{0}]$ with $L_p(m') = L_p(m) = L_p(x^{17}) = 9$.
\end{itemize}
\end{remark}
\begin{proof}
Since $[u,y]$ is $p$-symmetric, $L_p(u) = L_p(y) < L_p(m)$ and $x \neq u$. Let $w$ be the first $p$-point
`beyond' $y$ such that $L_p(w) > L_p(x)$. Take $L = L_p(x)$; Figure~\ref{fig:config_symarc} shows the
configuration of the relevant points on $[w,v]$ and their images under $\pi_p \circ \sigma^{-L}$ denoted
by $\tilde{}\ $-accents. Clearly $\tilde x = c$.
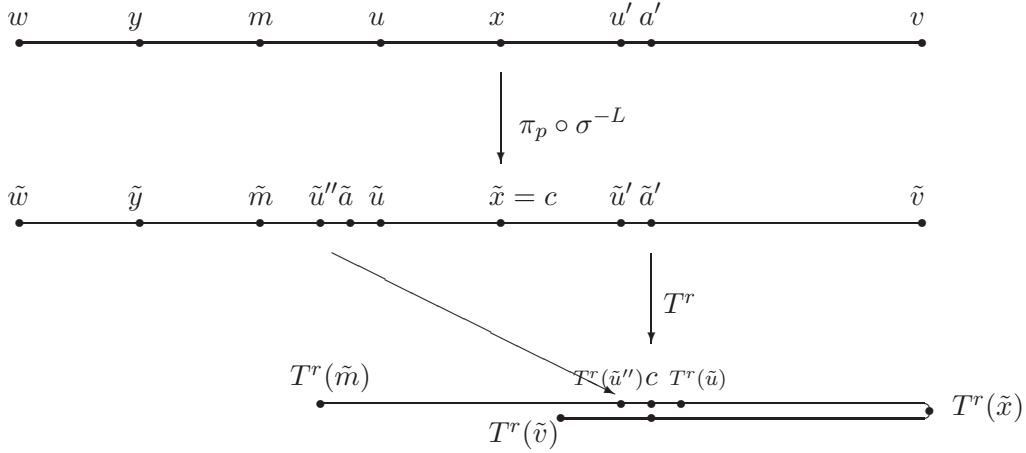
\begin{figure}[ht]
\unitlength=8mm
\begin{picture}(10,8)(3.7,-0.8)
\put(1,6){\line(1,0){15}}
\put(1, 6){\circle*{0.15}}\put(0.8, 6.3){\small $w$}
\put(3, 6){\circle*{0.15}}\put(2.8, 6.3){\small $y$}
\put(5, 6){\circle*{0.15}}\put(4.8, 6.3){\small $m$}
\put(7, 6){\circle*{0.15}}\put(6.8, 6.3){\small $u$}
\put(9, 6){\circle*{0.15}}\put(8.8, 6.3){\small $x$}
\put(11, 6){\circle*{0.15}}\put(10.8, 6.3){\small $u'$}
\put(16, 6){\circle*{0.15}}\put(15.8, 6.3){\small $v$}
\put(11.5, 6){\circle*{0.15}}\put(11.3, 6.3){\small $a'$}
\put(9, 5.5){\vector(0,-1){1.5}}\put(9.3, 4.5){\small $\pi_p \circ \sigma^{-L}$}
\put(1,3){\line(1,0){15}}
\put(1, 3){\circle*{0.15}}\put(0.8, 3.3){\small $\tilde w$}
\put(3, 3){\circle*{0.15}}\put(2.8, 3.3){\small $\tilde y$}
\put(5, 3){\circle*{0.15}}\put(4.8, 3.3){\small $\tilde m$}
\put(7, 3){\circle*{0.15}}\put(6.8, 3.3){\small $\tilde u$}
\put(9, 3){\circle*{0.15}}\put(8.8, 3.3){\small $\tilde x = c$}
\put(11, 3){\circle*{0.15}}\put(10.8, 3.3){\small $\tilde u'$}
\put(16, 3){\circle*{0.15}}\put(15.8, 3.3){\small $\tilde v$}
\put(11.5, 3){\circle*{0.15}}\put(11.3, 3.3){\small $\tilde a'$}
\put(6.5, 3){\circle*{0.15}}\put(6.3, 3.3){\small $\tilde a$}
\put(6, 3){\circle*{0.15}}\put(5.8, 3.3){\small $\tilde u''$}
 \put(11.5, 2.5){\vector(0,-1){1.5}}\put(11.7, 1.5){\small $T^r$}
 \put(6.2, 2.5){\vector(2,-1){4.7}}
\put(6,0){\line(1,0){10}}\put(16,-0.125){\oval(0.25,0.25)[r]}
\put(10,-0.25){\line(1,0){6}}
\put(11.5, 0){\circle*{0.15}}\put(11.5, -0.25){\circle*{0.15}} \put(11.4, 0.3){\small $c$}
\put(12, 0){\circle*{0.15}}\put(11.8, 0.3){\tiny $T^r(\tilde u)$}
\put(11, 0){\circle*{0.15}}\put(10.2, 0.3){\tiny $T^r(\tilde u'')$}
\put(6, 0){\circle*{0.15}}\put(5.5, 0.3){\small $T^r(\tilde m)$}
\put(10, -0.25){\circle*{0.15}}\put(8.8, -0.65){\small $T^r(\tilde v)$}
\put(16.125, -0.125){\circle*{0.15}}\put(16.5, -0.2){\small $T^r(\tilde x)$}
\end{picture}
\caption{\label{fig:config_symarc}
The configuration of points on $[w,v]$ and their images under $\pi_p \circ \sigma^{-L}$
and an additional $T^r$.}
\end{figure}

{\bf Case I:} $|\tilde w - c| < |\tilde v - c|$. Then by Remark~\ref{rem_basic} (b),
$\tilde w = c_{S_l}$ and $\tilde v = c_{S_k}$ with $k = Q(l)$. The points $\tilde y, \tilde m, \tilde u$
have symmetric copies $\tilde y', \tilde m', \tilde u'$ (\ie $T(\tilde y) = T(\tilde y')$, etc.) in reverse
order on $[c,\tilde v]$, and the pre-image under $\sigma^L \circ \pi_p^{-1}$ of the copy of $\tilde m'$
yields the required point $m'$.

{\bf Case II:} $|\tilde w - c| > |\tilde v - c|$, so in this case, $l = Q(k)$. We can in fact assume that
$|\tilde m - c| > |\tilde v - c|$ because otherwise we can find $m'$ precisely as in Case I. Now take the
$p$-point $a' \in (x,v)$ of maximal $p$-level, and let $a \in [m,x]$ be such that their
$\pi_p \circ \sigma^{-L}$-images $\tilde a$ and $\tilde a'$ are each other symmetric copies. Let $r$ be such
that $T^{r}(\tilde a) = c$; the bottom part of Figure~\ref{fig:config_symarc} shows the image of
$[\tilde m,\tilde v]$ under $T^r$. The point $T^r(\tilde x)$ and $T^r(v)$ are each others $\beta$-neighbors,
and since $L_p(v) > L_p(x)$, and by \eqref{eq:cond4}, $|T^r(\tilde x) - c| > |T^r(v)-c|$. Therefore
$[T^{r+j}(\tilde x) , T^{r+j}(\tilde a')] \supset [T^{r+j}(\tilde v) , T^{r+j}(\tilde a')]$ for all $j \geq 1$.

If $a, a' \in \ell$, then since $[x,a] \subset \ell$, this would imply that $[a', v] \subset \ell$ as well,
contrary to the fact that $x$ is the midpoint of $A_x$.

If on the other hand $a, a' \notin \ell$, then there is a point $u'' \in [m,a]$ such that $T^r(\tilde u'')$
and $T^r(\tilde u)$ are each other symmetric copies. It follows that $[u'', x]$ is a quasi-$p$-symmetric arc
properly contained in $[x,y]$, contradicting that $[x,y]$ is basic.
\end{proof}
\begin{proposition}[Extending a quasi-$p$-symmetric arc at its lower level endpoint]\label{lem:until_lm}
Let $A = [x, y] \subset \A$ be a basic quasi-$p$-symmetric arc, not contained in a single link, such that
$x$ and $y$ are the midpoints of the link-tips of $A$ and $L_p(x) > L_p(y)$. Let $m$ be the midpoint of $A$.
Then there exists a point $a$ such that $[m, a]$ is a quasi-$p$-symmetric arc with $y$ as the midpoint.
\end{proposition}
\begin{remark}
The assumption that $[x,y]$ is basic is essential. Without it, we would have a counter-example in
$[x,y] = [x^3,x^{30}]$ in Figure~\ref{fig:example}. The quasi-$p$-symmetric arc $[x^3,x^{30}]$ is
indeed not basic, because $[x^3,x^6]$ is a shorter quasi-$p$-symmetric arc in the figure. There is
a point $\wasv = x^{32}$ beyond $y$ with $L_p(\wasv) = L_p(x^{32}) = 3 > 1 = L_p(y)$, making it
impossible that $y$ is the midpoint of a quasi-$p$-symmetric arc stretching unto $m$.
\end{remark}
\begin{proof}
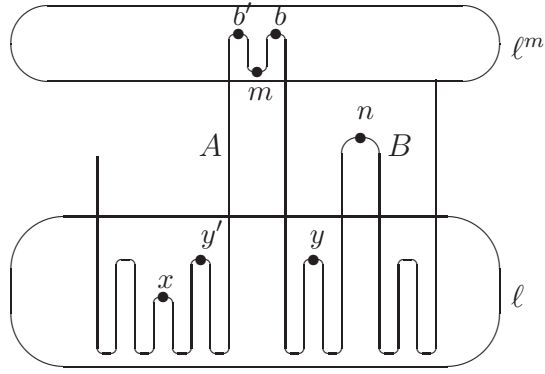
\begin{figure}[ht]
\unitlength=10mm
\begin{picture}(4,5.1)(0,0.5)
\put(0,1){\line(0,1){2.5}}
\put(0.125, 1){\oval(0.25,0.25)[b]}
\put(0.25,1){\line(0,1){1}}
\put(0.375, 2){\oval(0.25,0.25)[t]}
\put(0.5,1){\line(0,1){1}}
\put(0.625, 1){\oval(0.25,0.25)[b]}
\put(0.75,1){\line(0,1){0.5}}
\put(0.875, 1.5){\oval(0.25,0.25)[t]}
\put(0.8, 1.75){\small $x$}  \put(0.875, 1.625){\circle*{0.15}}
\put(1,1){\line(0,1){0.5}}
\put(1.125, 1){\oval(0.25,0.25)[b]}
\put(1.25,1){\line(0,1){1}}
\put(1.375, 2){\oval(0.25,0.25)[t]}
\put(1.375, 2.35){\small $y'$}  \put(1.375, 2.125){\circle*{0.15}}
\put(1.5,1){\line(0,1){1}}
\put(1.625, 1){\oval(0.25,0.25)[b]}
\put(1.75,1){\line(0,1){4}}
\put(1.875, 5){\oval(0.25,0.25)[t]}
\put(1.8, 5.25){\small $b'$}  \put(1.875, 5.125){\circle*{0.15}}
\put(2,4.75){\line(0,1){0.25}}
\put(2.125, 4.75){\oval(0.25,0.25)[b]}
\put(2, 4.25){\small $m$}  \put(2.125, 4.625){\circle*{0.15}}
\put(2.25,4.75){\line(0,1){0.25}}
\put(2.375, 5){\oval(0.25,0.25)[t]}
\put(2.35, 5.25){\small $b$}  \put(2.375, 5.125){\circle*{0.15}}
\put(2.5,1){\line(0,1){4}}
\put(2.625, 1){\oval(0.25,0.25)[b]}
\put(2.75,1){\line(0,1){1}}
\put(2.875, 2){\oval(0.25,0.25)[t]}
\put(2.825, 2.35){\small $y$}  \put(2.875, 2.125){\circle*{0.15}}
\put(3,1){\line(0,1){1}}
\put(3.125, 1){\oval(0.25,0.25)[b]}
\put(3.25,1){\line(0,1){2.5}}
\put(3.5, 3.5){\oval(0.5,0.5)[t]}
\put(3.45, 4){\small $\wasv$}  \put(3.5, 3.75){\circle*{0.15}}
\put(3.75,1){\line(0,1){2.5}}
\put(3.875, 1){\oval(0.25,0.25)[b]}
\put(4,1){\line(0,1){1}}
\put(4.125, 2){\oval(0.25,0.25)[t]}
\put(4.25,1){\line(0,1){1}}
\put(4.375, 1){\oval(0.25,0.25)[b]}
\put(4.5,1){\line(0,1){3.52}}
\put(2.1,1.7){\oval(6.5,2)}\put(5.5,1.5){$\ell$}
\put(2.1,5){\oval(6.5,1)}\put(5.5,4.8){$\ell^m$}
\put(1.35,3.5){$A$}\put(3.85,3.5){$B$}
\end{picture}
\caption{The arcs $A$ and $B$ and the relevant points for Proposition~\ref{lem:until_lm}, which is meant
to show that the point $\wasv$ does not exist in $B$.}
\label{fig:Lemma9}
\end{figure}
A quasi-$p$-symmetric arc is not contained in a single link, so $[x,m] \not\subset \ell$.
Let $H = [x,\wasv] \supset A$ be the smallest arc containing a point $\wasv$ `beyond' $y$ with
$L_p(\wasv) > L_p(y)$.

Corollary~\ref{cor:spiral_linktips} implies that the arc $[x,m]$ contains a $p$-point $y'$ with
$L_p(y') = L_p(y)$. Let $b$ and $b'$ be the $p$-points having the highest $p$-level on the arcs
$[y,m)$ and $[y',m)$ respectively. By symmetry, $L_p(b) = L_p(b')$, and possibly $b = y$, $b' = y'$.
Let $z \in [x,y']$ be the point closest to $y'$ such that $L_p(z) > L_p(b)$; possibly $z = x$.
Since $b' \in [y',m)$, we have
\[
L_p(y) = L_p(y') \le L_p(b) = L_p(b') < L_p(m).
\]
Take $L := L_p(b)$ and let $\tilde H = \pi_p \circ \sigma^{-L}(H)$. Since $y$ is the midpoint of its
link-tip, $[y, \wasv] \not\subset \ell$. Therefore
$\pi_p^{-1}(c) \cap \sigma^{-L}(H) \supset \{ \sigma^{-L}(b), \sigma^{-L}(b')\}$, and
$\tilde z = \pi_p \circ \sigma^{-L}(z)$ and $\tilde \wasv = \pi_p \circ \sigma^{-L}(\wasv)$ have
$\tilde m = \pi_p \circ \sigma^{-L}(m)$ as common $\beta$-neighbor, see Figure~\ref{fig:Lemma9b}.
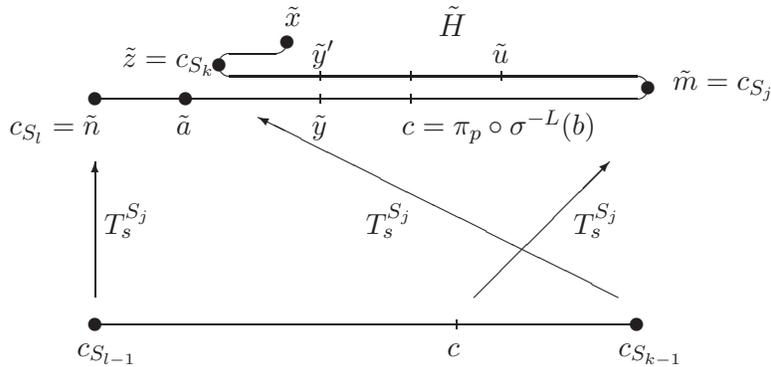
\begin{figure}[ht]
\unitlength=12mm
\begin{picture}(8,4.1)(0,-2.5)
\put(1,0.5){\line(1,0){6}}
\put(3, 1.125){\oval(0.25,0.25)[br]}
\put(3.125, 01.125){\circle*{0.15}}\put(3.1, 1.3){\small $\tilde x$}
\put(2.5,1){\line(1,0){0.5}}
\put(2.5, 0.875){\oval(0.25,0.25)[l]}
\put(2.5,0.75){\line(1,0){4}}
\put(2.375, 0.875){\circle*{0.15}}\put(1.3, 0.85){\small $\tilde z = c_{S_k}$}
\put(6.5,0.75){\line(1,0){0.5}}
\put(7, 0.625){\oval(0.25,0.25)[r]}
\put(7.125, 0.625){\circle*{0.15}}\put(7.4, 0.6){\small $\tilde m = c_{S_j}$}
\put(1, 0.5){\circle*{0.15}}\put(0.05, 0.1){\small $c_{S_l} = \tilde \wasv$}
\put(2, 0.5){\circle*{0.15}}\put(1.9, 0.1){\small $\tilde a$}
\put(4.5, 0.45){\line(0,1){0.1}}\put(4.4, 0.1){\small $c = \pi_p \circ \sigma^{-L}(b)$}
\put(3.5, 0.45){\line(0,1){0.1}}\put(3.4, 0.1){\small $\tilde y$}
\put(3.5, 0.7){\line(0,1){0.1}} \put(3.4, 0.9){\small $\tilde y'$}
\put(5.5, 0.7){\line(0,1){0.1}} \put(5.4, 0.9){\small $\tilde u$}
\put(4.5, 0.7){\line(0,1){0.1}}
\put(4.8, 1.2){$\tilde H$}
\put(1,-1.7){\vector(0,1){1.5}}\put(1.1, -1){\small $T_s^{S_j}$}
\put(5.2,-1.7){\vector(1,1){1.5}}\put(6.3, -1){\small$T_s^{S_j}$}
\put(6.8,-1.7){\vector(-2,1){3.99}}\put(4, -1){\small$T_s^{S_j}$}
\put(1,-2){\line(1,0){6}}
\put(1, -2){\circle*{0.15}}\put(0.8, -2.35){\small $c_{S_{l-1}}$}
\put(7, -2){\circle*{0.15}}\put(6.8, -2.35){\small $c_{S_{k-1}}$}
\put(5, -2.05){\line(0,1){0.1}}\put(4.9, -2.35){\small $c$}
\end{picture}
\caption{The arc $\tilde H$ drawn as multiple curve, its preimage under $T_s^{S_j}$ and the relevant
points on them.}
\label{fig:Lemma9b}
\end{figure}
Since $L_p(z) > L_p(b)$ there is $k$ such that $\tilde z = c_{S_k}$. Also take $l$ such that
$\tilde \wasv = c_{S_l}$ and $j$ such that $\tilde m = c_{S_j}$. Let $\tilde y = \pi_p \circ \sigma^{-L}(y)$
and $\tilde y' = \pi_p \circ \sigma^{-L}(y')$.

We claim that there is a point $a \in [\wasv,m]$ such that
\[
\tilde a :=  \pi_p \circ \sigma^{-L}(a)  \in [c_{S_l} , \tilde y] \quad\text{ and }\quad T_s(\tilde a) = T_s(\tilde m).
\]
Since $c_{S_j}$ is $\beta$-neighbor to both $c_{S_l}$ and $c_{S_k}$, we have three cases:
\begin{enumerate}
\item
$j = Q(k)$ and $l = Q(j)$, so $l = Q^2(k)$. In this case, Equation~\eqref{eq:cond4} and
Remark~\ref{rem:cond4} imply that $|c - c_{S_l}| > |c - c_{S_{Q(k)}}|$, so $[c_{S_l}, c]$
contains the required point $\tilde a$ with $T_s(\tilde a) = T_s(\tilde m)$. By the same token,
$|c_{S_k} - c| < |c_{S_j} - c| = \frac12 |\tilde a - \tilde m|$. Since
$|\tilde y - c| = |\tilde y'- c| < |c_{S_k} - c|$, we indeed obtain that
$\tilde a \in [ c_{S_l}, \tilde y]$.
\item
$j = Q(l)$ and $k = Q(j)$, so $k = Q^2(l)$. Then Remark~\ref{rem:Hofbauer} implies that
$|c - c_{S_k}| > |c - c_{S_l}|$. But this would mean that the arc $[\wasv,m]$ is shorter than
$[z, m]$ and in particular that $[y,\wasv] \subset \ell$, contradicting that $y$ is the midpoint
of its link-tip.
\item
$j = Q(k) = Q(l)$. In this case, we pull $\tilde H$ back for another $S_j$ iterates, or more precisely,
we look at the arc $\pi_p \circ \sigma^{-S_j-L}(H)$. The endpoints of this arc are $c_{S_{k-1}}$ and
$c_{S_{l-1}}$ which are therefore $\beta$-neighbors. If $l - 1 = Q(k-1)$, then we find
\[
Q(k) = Q(l) = Q(Q(k-1)+1)
\]
which contradicts Condition \eqref{eq:cond4} with $k$ replaced by $k-1$. If $k-1 = Q(l-1)$, then we find
\[
Q(l) = Q(k) = Q(Q(l-1)+1)
\]
which contradicts Condition \eqref{eq:cond4} with $k$ replaced by $l-1$.
\end{enumerate}
This proves the claim.

Suppose now that $\tilde y \neq c$ (\ie $y \neq b$). Then $b, b' \notin \ell$ because $y$ has the largest
$p$-level in its link-tip. Since $|c_{S_k} - c| < |c - \tilde m|$, there is a point $u \in [z,m]$ such that
$\tilde u = \pi_p \circ \sigma^{-L}(u) \in [c,\tilde m]$ and $T_s(\tilde u) = T_s(\tilde y)$. This means
that $[x,u]$ is a quasi-$p$-symmetric arc properly contained in $[x,m]$, contradicting the assumption that
$[x,y]$ is a basic quasi-symmetric arc.

Therefore $y = b$, so there are no $p$-points between $y$ and $m$ of level higher than $L_p(y)$. Instead,
the arc $[a,m]$ has midpoint $y$, and is the required quasi-$p$-symmetric arc, proving the lemma.
\end{proof}
\begin{remark}\label{rem:level-zero}
Let $A = [x,y]$ be a basic quasi-$p$-symmetric arc such that $x$ and $y$ are the midpoints of the
link-tips of $A$ and $L_p(x) > L_p(y)$. Let $\ell^m$ be the link which contains the midpoint $m$ of
$A$, and let $A_m$ be the arc-component of $\ell^m$ containing $m$. Then, by the lemma above,
$A \setminus (\ell\mbox{-tips } \cup A_m)$ does not contain any $p$-point $z$ with $L_p(z) \ge L_p(y)$.
\end{remark}

\section{Link-Symmetric Arcs}\label{sec:link}

\begin{defi}\label{def:decreasing}
We say that an arc $[x, y]$ is \emph{decreasing $($basic$)$ quasi-$p$-symmetric} if it is the
concatenation of (basic) quasi-$p$-symmetric arcs where the $p$-levels of the midpoints decrease,
\ie if there are $p$-points $x = x^0, x^1, x^2, \dots , x^{n-1}$ and $x^n = y$ can be
a $p$-point or not, such that the following hold:
\begin{itemize}
\item[(i)] $[x^{i-1}, x^{i+1}]$ is a (basic) quasi-$p$-symmetric arc with midpoint $x^{i}$, for
$i = 1, \dots , n-1$. (By definition of a (basic) quasi-$p$-symmetric arc, the points $x^{2i}$ all
belong to a single link, and the points $x^{2i-1}$ belong to a single link as well.)
\item[(ii)] $L_p(x^i) > L_p(x^{i+1})$, for $i = 1, \dots , n-1$ (and if $y$ is a $p$-point then also
$L_p(x^{n-1}) > L_p(y)$).
\end{itemize}
Similarly, we say that the arc $[x, y]$ is \emph{increasing $($basic$)$ quasi-$p$-symmetric} if it is the
concatenation of (basic) quasi-$p$-symmetric arcs where the $p$-levels of the midpoints increase.
\end{defi}
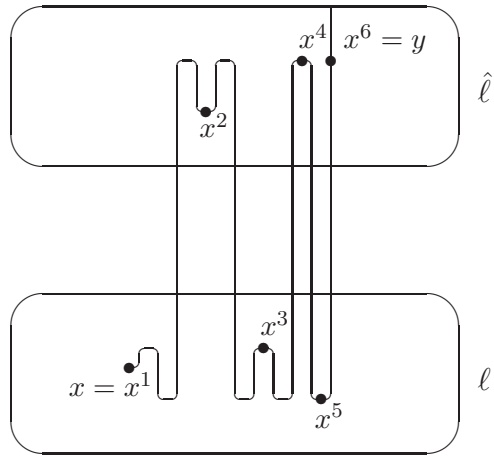
\begin{figure}[ht]
\unitlength=8.5mm
\begin{picture}(10,8)(3,10.5)
\put(6.35,12.5){\oval(0.3,0.3)[br]}\put(6.65,12.5){\oval(0.3,0.3)[t]}
\put(6.95,12){\oval(0.3,0.3)[b]}
\put(6.8,12){\line(0,1){0.5}}
\put(6.35,12.35){\circle*{0.15}}\put(5.4,11.9){\small $x = x^1$}
\put(7.1,12){\line(0,1){5}}\put(7.25, 17){\oval(0.3,0.3)[t]}
\put(7.55,16.5){\oval(0.3,0.3)[b]}\put(7.85, 17){\oval(0.3,0.3)[t]}
\put(7.7,16.5){\line(0,1){0.5}}\put(7.4,16.5){\line(0,1){0.5}}
\put(8,12){\line(0,1){5}}
\put(7.55,16.35){\circle*{0.15}}\put(7.45,15.95){\small $x^2$}
\put(8.15, 12){\oval(0.3,0.3)[b]}\put(8.45,12.5){\oval(0.3,0.3)[t]}
\put(8.75,12){\oval(0.3,0.3)[b]}
\put(8.3,12){\line(0,1){0.5}}\put(8.6,12){\line(0,1){0.5}}
\put(8.45,12.65){\circle*{0.15}}\put(8.4,12.85){\small $x^3$}
\put(8.9,12){\line(0,1){5}}\put(9.05, 17){\oval(0.3,0.3)[t]}
\put(9.2,12){\line(0,1){5}}\put(9.35, 12){\oval(0.3,0.3)[b]}
\put(9.35,11.85){\circle*{0.15}}\put(9.25,11.4){\small $x^5$}
\put(9.5,12){\line(0,1){6}}
\put(9.5,17.15){\circle*{0.15}}\put(9.7,17.35){\small $x^6 = y$}
\put(9.05,17.15){\circle*{0.15}}\put(9,17.35){\small $x^4$}
\put(5,11){\line(1,0){6}}\put(5,13.5){\line(1,0){6}}
\put(5,12.25){\oval(1,2.5)[l]}\put(11,12.25){\oval(1,2.5)[r]}
\put(5,15.5){\line(1,0){6}}\put(5,18){\line(1,0){6}}
\put(5,16.75){\oval(1,2.5)[l]}\put(11,16.75){\oval(1,2.5)[r]}
\put(11.8,16.5){$\hat \ell$}\put(11.8,12){$\ell$}
\end{picture}
\caption{Illustration of a basic decreasing quasi-$p$-symmetric arc. The point $y$ is not a
$p$-point here; instead, the arc $A$ goes straight through $\hat \ell$ at $y$.}
\label{fig:decreas-quasi-sym}
\end{figure}
\begin{ex} Consider the Fibonacci inverse limit space, and let our chain $\chain_p$
be such that $p$-points with $p$-levels $1$ and $14$ belong to the same link $\ell$, but $p$-points with
$p$-level $9$ are not contained in $\ell$. Since $p$-points with $p$-level 14 belong to the same link $\ell$
as $p$-points with $p$-level $1$, also $p$-points with $p$-levels $22$, $35$, $56$ and $77$
belong to $\ell$.
Let $p$-points with $p$-level $43$ belong to the same link as $p$-points with $p$-level $9$.
\begin{itemize}
\item[(1)] {\bf Example of a basic decreasing quasi-$p$-symmetric arc.} Let $A = [y^0, y^{12}]$ be an
arc with the following folding pattern (where the subscripts count important $p$-points):
$$
\underbrace{{\bf 1 \ 22 \ 77}_2 \ {\bf 22 \ 1} \ 9 \ 43_6 \ 9 \ {\bf 1}}_{\textrm{basic}} \overbrace{{\bf 22}_9 \ {\bf 1} \ 9_{11} \ {\bf 1}_{12}}^{\textrm{basic}}
$$
Let $x^i$ be as in the above definition. Then $x^1 = y^2$, $x^2 = y^6$, $x^3 = y^9$, $x^4 = y^{11}$,
and $x^5 = y^{12}$. So $[y^2, y^9]$ is basic quasi-$p$-symmetric with midpoint $y^6$, $[y^6, y^{11}]$
is basic quasi-$p$-symmetric with midpoint $y^9$, and $[y^9, y^{12}]$ is basic quasi-$p$-symmetric with
the midpoint $y^{11}$. Also
$L_p(y^2) = 77 > L_p(y^6) = 43 > L_p(y^9) = 22 > L_p(y^{11}) = 9 > L_p(y^{12}) = 1$.
\item[(2)] {\bf Example of a non-basic decreasing quasi-$p$-symmetric arc.} Let $[y^0, y^{72}]$ be an
arc with the following folding pattern:
$$
\overbrace{\underbrace{{\bf 1 \ 22 \ 1 \ 56}_3 \, {\bf 1 \ 22 \ 1} \ 9 \ {\bf 1}}_{\textrm{basic}} \ 4 \ 1 \ 0 \ 2 \ 0 \ 1 \ 0 \ 3 \ 0 \ {\bf 1} \ 6 \ {\bf 1 \ 14 \ 1 \ 35}_{23} {\bf 1 \ 14 \ 1} \ 6 \ {\bf 1} \ 0 \ 3 \ 0 \ 1 \ 0 \ 2 \ 0 \ 1 \ 4 \ \underbrace{{\bf 1} \ 9 \ {\bf 1}}}^{\textrm{quasi-$p$-symmetric}}
$$
$$
\underbrace{\overbrace{{\bf 22}_{41} {\bf 1} \ 9 \ {\bf 1}}^{\textrm{basic}}
  \ 4 \ 1 \ 0 \ 2 \ 0 \ 1 \ 0 \ 3 \ 0 \ {\bf 1} \ 6 \ {\bf 1 \ 14}_{57} {\bf 1}
  \ 6 \ {\bf 1} \ 0 \ 3 \ 0 \ 1 \ 0 \ 2 \ 0 \ 1 \ 4 \ \overbrace{{\bf
      1} \ 9 \ {\bf 1}_{72}}^{\textrm{sym}}}_{\textrm{quasi-$p$-symmetric}}
$$
Let $x^i$ be again as in the above definition. Then $x^1 = y^3$, $x^2 = y^{23}$, $x^3 = y^{41}$,
$x^4 = y^{57}$, and $x^5 = y^{72}$. So, arcs $[y^3, y^{41}]$, $[y^{23}, y^{57}]$ and $[y^{41}, y^{72}]$
are quasi-$p$-symmetric, and
$L_p(y^3) = 56 > L_p(y^{23}) = 35 > L_p(y^{41}) = 22 > L_p(y^{57}) = 14 > L_p(y^{72}) = 1$.
\item[(3)] {\bf Example of an arc that is the concatenation of two quasi-$p$-symmetric arcs (one of them
is basic), but is not decreasing quasi-$p$-symmetric.} Let $[y^0, y^{40}]$ be an
arc with the following folding pattern:
$$
\underbrace{{\bf 1 \ 22 \ 77}_2 \, {\bf 22 \ 1} \ 9 \ 43_6 \, 9 \ {\bf 1}}_{\textrm{basic}}\underbrace{{\bf 22}_9 \, {\bf 1} \ 9_{11} {\bf 1}_{12} 4 \ 1 \ 0 \ 2 \ 0 \ 1 \ 0 \ 3 \ 0 \ {\bf 1} \ 6 \ {\bf 1 \ 14}_{25} {\bf 1} \ 6 \ {\bf 1} \ 0 \ 3 \ 0 \ 1 \ 0 \ 2 \ 0 \ 1 \ 4 \ {\bf 1} \ 9 \ {\bf 1}_{40}}_{\textrm{quasi-$p$-symmetric}}
$$
Then $[y^2, y^9]$ is basic quasi-$p$-symmetric with midpoint $y^6$, $[y^6, y^{11}]$ is basic
quasi-$p$-symmetric with midpoint $y^9$, and $[y^9, y^{12}]$ is basic quasi-$p$-symmetric with the
midpoint $y^{11}$. However, $[y^9, y^{40}]$ is quasi-$p$-symmetric with midpoint $y^{25}$ and
$[y^6, y^{25}]$ is neither basic quasi-$p$-symmetric, nor quasi-$p$-symmetric. So $A = [y^0, y^{40}]$
is not a decreasingly quasi-$p$-symmetric arc. Note that $[y^0, y^{12}]$ is a decreasing
quasi-$p$-symmetric arc.
\end{itemize}
\end{ex}

\begin{proposition}\label{lem:not-link-sym}
Let $A$ be a non-basic quasi-$p$-symmetric arc. Then there are $k, n, m, d \in \None$, $d < k$, such that
$$
A \cap E_p = \{ x^0, \dots , x^k, \dots , x^{k+n}, \dots , x^{k+n+m} \},
$$
$[x^0, x^k]$ is a basic quasi-$p$-symmetric arc with midpoint $x^{k-d}$  and $[x^k, x^{k+n}]$ is
$p$-symmetric. Moreover,
\begin{itemize}
\item[(i)] If $[x^{k+n}, x^{k+n+m}]$ is $p$-symmetric, then
$[x^{-k+m/2}, x^{k+n+3m/2}]$ is not $p$-link-symmetric.
\item[(ii)] If $[x^{k+n}, x^{k+n+m}]$ is a basic quasi-$p$-symmetric arc,
then $A$ is contained in a decreasing quasi-$p$-symmetric arc consisting of at least two
quasi-$p$-symmetric arcs. More precisely, $[x^{-k-n/2}, x^{k+n/2}]$ and
$[x^{k+n/2}, x^{k+2m+3n/2}]$ are the quasi-$p$-symmetric arcs contained in the decreasing
quasi-$p$-symmetric arc $[x^{-k-n/2}, x^{k+2m+3n/2}]$ containing $A$.
\end{itemize}
\end{proposition}
\begin{proof}
Since $A$ is a non-basic quasi-$p$-symmetric arc, there is a basic quasi-$p$-symmetric arc which we can label
$[x^0,x^k]$. The arc $[x^{k},x^{k+n}]$ in the middle is $p$-symmetric by definition of quasi-$p$-symmetry,
and it has the same midpoint $x^{k+n/2}$ as $A$. The arc $[x^{k+n}, x^{k+n+m}]$ could be either
$p$-symmetric or basic quasi-$p$-symmetric.

$(i)$ Assume that $[x^{k+n}, x^{k+n+m}]$ is $p$-symmetric. Without loss of generality we can suppose
that $x^0$ and $x^{k+n+m}$ are the midpoints of the link-tips of $A$, and also that $x^k$ and $x^{k+n}$ are
the midpoints of their arc-components. Since the point $x^{k+n+m/2}$ is the
midpoint of the $p$-symmetric arc $[x^{k+n},x^{k+n+m}]$, and the symmetry of the arc $[x^k, x^{k+n}]$ can be
extended to the midpoints of its neighboring (quasi-)symmetric arcs, we obtain that $d = m/2$ and the point
$x^{k-m/2}$ is the midpoint of the basic quasi-$p$-symmetric arc $[x^0,x^k]$. Proposition~\ref{lem:symarc} implies
that we can extend $[x^0, x^{k-m/2}]$ beyond $x^0$ to obtain the arc $[x^{-k+m/2}, x^{k-m/2}]$ which is either
$p$-symmetric, or quasi-$p$-symmetric, and hence $p$-link-symmetric.

First, let us assume that $L_p(x^{k+n+m}) = 1$. Let us consider the arc $[x^{k+n+m/2}, x^{k+n+3m/2}]$.
Its midpoint $x^{k+n+m}$ has $p$-level $1$. If $L_p(x^{k+n+m-1}) = L_p(x^{k+n+m+1})$, then
$L_p(x^{k+n+m-1}) = 0$. Furthermore $x^{k+n+m-1} \ne x^{k+n+m/2}$ since a midpoint cannot have
$p$-level zero. It follows that $x^{k+n+m-2}$ and $x^{k+n+m+2}$ have different $p$-levels, and
are not in the same link, since by Lemma~\ref{prop2} there is no quasi-$p$-symmetric arc whose
both boundary points are $p$-points and whose midpoint has $p$-level $1$.

If $L_p(x^{k+n+m-1}) \ne L_p(x^{k+n+m+1})$ then again $x^{k+n+m-1}$ and $x^{k+n+m+1}$ are not in the same
link (by Lemma~\ref{prop2} there is no quasi-$p$-symmetric arc whose both boundary points are
$p$-points and whose midpoint has $p$-level $1$). In either case, $[x^{k+n+m/2}, x^{k+n+3m/2}]$ is not
$p$-link-symmetric and hence $[x^{-m/2}, x^{k+n+3m/2}]$ is not
$p$-link-symmetric. This proves statement (i) in the case that
$L_p(x^{k+n+m})=1$.

Now for the general case, let $L := L_p(x^{k+n+m})$. The basic idea is to shift $[x^0,x^{k+n+m}]$ back by
$L-1$ iterates, and use the above argument. Note that the arcs $[x^k, x^{k+n}]$ and $[x^{k+n}, x^{k+n+m}]$
are $p$-symmetric and hence $L_p(x^{k+n/2}) > L_p(x^{k+n}) = L_p(x^{k+n+m}) = L$. Then $\sigma^{-L+1}(A)$
is also a quasi-$p$-symmetric arc which is not basic, the arc $\sigma^{-L+1}([x^0, x^k])$ is a basic
quasi-$p$-symmetric arc and $L_p(\sigma^{-L+1}(x^{k+n+m})) = 1$. Let
$$
\sigma^{-L+1}(A) \cap E_p =
\{ u^0, \dots , u^{\hat k}, \dots , u^{{\hat k}+{\hat n}}, \dots , u^{{\hat k}+{\hat n}+{\hat m}} \},
$$
where $u^{\hat i} = \sigma^{-L+1}(x^i)$. (Note that $\hat k \le k$, $\hat n \le n$ and $\hat m \le m$,
since not every $\sigma^{-L+1}(x^i)$ needs to be a $p$-point.)
Then $G = [u^{-{\hat k}+{\hat m}/2}, u^{{\hat k}+{\hat n}+3{\hat m}/2}]$ is an arc
with `boundary arcs'
$[u^{-{\hat k}+\hat m/2}, u^{{\hat k}-\hat m/2}]$ and $[u^{k+n+\hat m /2}, u^{k+n+3\hat m/2}]$
and the midpoint of the latter has $p$-level $1$.
The above argument shows that
this arc cannot be $p$-link-symmetric, and therefore the whole arc
$G$ is not $p$-link-symmetric with midpoint $u = \sigma^{-L+1}(x^{k+n/2})$.

We want to prove that $\sigma^j(G)$ is also not $p$-link-symmetric with
the midpoint $\sigma^j(u)$ for $j = L - 1$.

Let us assume by contradiction  that $\sigma^j(G)$ is $p$-link-symmetric. Since $[x^{-k+m/2}, x^{k-m/2}]$ is
$p$-symmetric, also $\sigma^j([u^{{\hat k}+{\hat n}+{\hat m}/2}, u^{{\hat k}+{\hat n}+3{\hat m}/2}])$ is
$p$-link-symmetric. But $[u^{{\hat k}+{\hat n}+{\hat m}/2}, u^{{\hat k}+{\hat n}+3{\hat m}/2}]$ has its
midpoint at $p$-level $1$, and hence is not $p$-link-symmetric. Therefore, there exists $l < j$ such that
$\sigma^l([u^{{\hat k}+{\hat n}+{\hat m}/2}, u^{{\hat k}+{\hat n}+3{\hat m}/2}])$ is not $p$-link-symmetric
and $\sigma^{l+1}([u^{{\hat k}+{\hat n}+{\hat m}/2}, u^{{\hat k}+{\hat n}+3{\hat m}/2}])$ is
$p$-link-symmetric. By Proposition~\ref{prop:chains}, and since
$L_p(\sigma^l(u^{{\hat k}+{\hat n}+{\hat m}})) = l+1 \ne 0$, there exist
$v \in \sigma^l([u^{{\hat k}+{\hat n}+{\hat m}/2}, u^{{\hat k}+{\hat n}+{\hat m}}])$ and
$w \in \sigma^l([u^{{\hat k}+{\hat n}+{\hat m}}, u^{{\hat k}+{\hat n}+3{\hat m}/2}])$
such that $L_p(v) = L_p(w) = 0$, see Figure~\ref{fig:sym2}.

\begin{figure}[ht]
\unitlength=8.5mm
\begin{picture}(10,9.5)(-2,9)
\put(-2.75,16){\oval(0.5,0.5)[t]}
\put(-2.5,13){\line(0,1){3}}
\put(-2.25,13){\oval(0.5,0.5)[b]}
\put(-2,13){\line(0,1){4}}
\put(-1.75,17){\oval(0.5,0.5)[t]}
\put(-1.5,10.25){\line(0,1){6.75}}
\put(-1.25,10.25){\oval(0.5,0.5)[b]}
{\thicklines\put(-4.5,13.1){\line(1,0){5}}}
\put(-1.75,17.25){\circle*{0.13}}
\put(-2.25,12.75){\circle*{0.13}}
\put(-2.75,16.25){\circle*{0.13}}
\put(-1.25,10){\circle*{0.13}}
\put(-2,13.1){\circle*{0.13}}
\put(-1.5,13.1){\circle*{0.13}}
\put(-1.8,17.5){\small $y$}
\put(-1.35,9.6){\small $z$}
\put(-3.2,16.2){\small $x$}
\put(-2.3,13.16){\small $v$}
\put(-1.4,13.16){\small $w$}
\put(-4.8,13){\small $c$}
\put(-2,10){\oval(4,1.1)}
\put(-2,11){\oval(4,1.1)}
\put(-2,12){\oval(4,1.1)}
\put(-2,13){\oval(4,1.1)}
\put(-2,14){\oval(4,1.1)}
\put(-2,15){\oval(4,1.1)}
\put(-2,16){\oval(4,1.1)}
\put(1.5,14.5){\vector(1,0){3}}\put(3,14.7){\small $\sigma$}
\put(6.95,12){\oval(0.3,0.3)[b]}
\put(6.95,11.85){\circle*{0.13}}\put(6.5,11.3){\small $\sigma(x)$}
\put(7.1,12){\line(0,1){5}}\put(7.25, 17){\oval(0.3,0.3)[t]}
\put(7.55,16.5){\oval(0.3,0.3)[b]}\put(7.85, 17){\oval(0.3,0.3)[t]}
\put(7.7,16.5){\line(0,1){0.5}}\put(7.4,16.5){\line(0,1){0.5}}
\put(8,10.5){\line(0,1){6.5}}
\put(7.85,17.15){\circle*{0.13}}\put(7.4,17.4){\small $\sigma(v)$}
\put(8.3,10.5){\oval(0.6,0.6)[b]}
\put(8.3,10.2){\circle*{0.13}}\put(8,9.6){\small $\sigma(y)$}
\put(8.6,10.5){\line(0,1){6.4}}
\put(8.9,16.85){\oval(0.6,0.6)[t]}
\put(9.2,12.2){\line(0,1){4.7}}
\put(8.9,17.15){\circle*{0.13}}\put(8.6,17.4){\small $\sigma(w)$}
\put(9.35, 12.2){\oval(0.3,0.3)[b]}
\put(9.5,12.2){\line(0,1){0.5}}
\put(9.35,12.05){\circle*{0.13}}\put(9,11.5){\small $\sigma(z)$}
\put(8.2,12){\oval(6,2)}
\put(8.2,17){\oval(6,2)}
\end{picture}
\caption{The configuration of $p$-levels that does not exist. Here $x = \sigma^l(u^{{\hat k}+{\hat n}+{\hat m/2}})$,
$y = \sigma^l(u^{{\hat k}+{\hat n}+{\hat m}})$ and $z = \sigma^l(u^{{\hat k}+{\hat n}+{3\hat m/2}})$.}
\label{fig:sym2}
\end{figure}
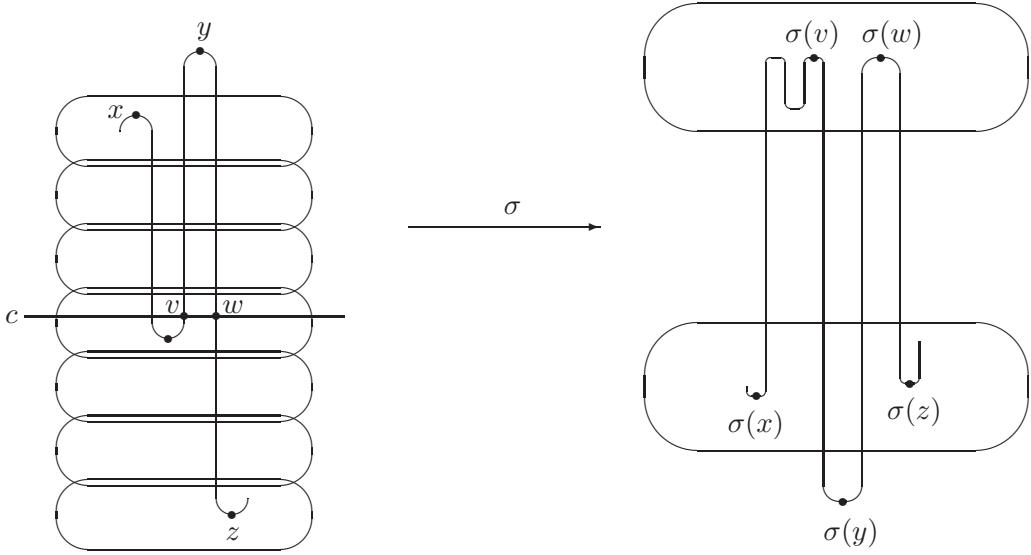

Since $\sigma^{l+1}(u^{{\hat k}+{\hat n}+{\hat m}/2})$ and
$\sigma^{l+1}(u^{{\hat k}+{\hat n}+3{\hat m}/2})$ belong to the same link and
$L_p(\sigma^{l+1}(u^{{\hat k}+{\hat n}+{\hat m}/2})) \ne L_p(\sigma^{l+1}(u^{{\hat k}+{\hat n}+3{\hat m}/2}))$, Proposition~\ref{prop:chains}
implies that
$\sigma^{l+1}(u^{{\hat k}+{\hat n}+{\hat m}/2})$ and $\sigma^{l+1}(u^{{\hat k}+{\hat n}+3{\hat m}/2})$ belong
to the same link as $\sigma(v)$ and $\sigma(w)$. But then $\sigma^{l}(u^{{\hat k}+{\hat n}+{\hat m}/2})$ and
$\sigma^{l}(u^{{\hat k}+{\hat n}+3{\hat m}/2})$ belong to the same link as $v$ and $w$, contradicting the
choice of $l$.

$(ii)$ The rough idea of this proof is as follows: Whenever $[x^{k+n}, x^{k+n+m}]$ is not $p$-symmetric,
there exists $N \in \None$ such that $\sigma^{-N}(A)$ is a basic quasi-$p$-symmetric arc and we can
apply Propositions~\ref{lem:symarc} and~\ref{lem:until_lm} to obtain the arc $B \supset \sigma^{-N}(A) $
which is decreasing basic quasi-$p$-symmetric. Then $\sigma^{N}(B) \supset A$ is the required
decreasing quasi-$p$-symmetric arc.

Let us assume now that $[x^{k+n}, x^{k+n+m}]$ is basic quasi-$p$-symmetric. Let us denote by $\ell$
the link which contains $x^0$. Then $x^k, x^{k+n}, x^{k+n+m} \in \ell$. We can assume without loss of
generality that $x^k$ and $x^{k+n}$ are the $p$-points in the link-tips of $[x^k, x^{k+n}]$ furthest away
from the midpoint $x^{k+n/2}$ and, similarly, $x^0$ and $x^{k+n+m}$ are the $p$-points in the link-tips
of $[x^0, x^{k+n+m}]$ furthest away from the midpoint $x^{k+n/2}$. Then from the properties of the chain
in Proposition~\ref{prop:chains} we conclude that
$L_p(x^0) = L_p(x^k) = L_p(x^{k+n}) = L_p(x^{k+n+m})$. Let us denote
by $x^a$ and $x^b$ the midpoints of arc-components which contains $x^0$ and $x^{k+n+m}$ respectively.
Then $x^a, x^b \in \ell$ and $x^b \ne x^{k+n+m}$. Without loss of generality we can assume that
$L_p(x^a) > L_p(x^b)$.

Since $x^{k-d}$ is the midpoint of $[x^0, x^k]$ and $A$ is quasi-$p$-symmetric, $x^{k+n+d}$ is the midpoint
of $[x^{k+n}, x^{k+n+m}]$.

By Proposition~\ref{lem:symarc}, $L_p(x^{-d}) = L_p(x^{k-d})$ and $L_p(x^{k+n+d}) = L_p(x^{k+n+m+d})$, see
Figure~\ref{fig:not-link-sym}.

Let us denote by $\ell^d$ the link which contains $x^{-d}$, and by $A_d$ the arc-component of $\ell^d$
which contains $x^{-d}$.
\\[2mm]
{\bf Claim} $x^{-d}$ is the midpoint of its arc-component $A_d$.
\\[2mm]
Consider the arc $\sigma^{-L+1}(A)$, where $L := L_p(x^b)$. Since $L_p(x^a) > L_p(x^{k+n/2}) > L_p(x^b) = L$,
the preimage $\sigma^{-L+1}(A)$ contains the points $\sigma^{-L+1}(x^b)$ with $L_p(\sigma^{-L+1}(x^b)) = 1$,
$\sigma^{-L+1}(x^a)$ and $\sigma^{-L+1}(x^{k+n/2})$ is the midpoint of $\sigma^{-L+1}(A)$.

By Corollary~\ref{cor:spiral_linktips} the arc-component containing $x^a$ also contains $p$-points $x'$
and $x''$ with the property that $[x', x'']$ is $p$-symmetric with midpoint
$x^a$ and $L_p(x') = L_p(x'') = L_p(x^b)$,
Assume also that $x'$ and $x''$ are furthest away from $x^a$ with these properties.
Therefore, $\sigma^{-L+1}(A)\cap E_p \supseteq
\{ u^0, u^{\hat a}, u^{2\hat a}, u^{2\hat a+\hat n}, u^{2\hat a+ 2\hat n} \}$,
where $u^{\hat a} = \sigma^{-L+1}(x^a)$, $u^{2 \hat a+\hat n} = \sigma^{-L+1}(x^{k+n/2})$,
$u^{2 \hat a+ 2\hat n} = \sigma^{-L+1}(x^b)$, $u^0 = \sigma^{-L+1}(x')$, $u^{2 \hat a} = \sigma^{-L+1}(x'')$
and $L_p(u^0) = L_p(u^{2\hat a}) = 1$.

Let us suppose that $\sigma^{-L+1}(A)$ is not contained in a single link. Since $\sigma^{-L+1}(x^a)$ and
$\sigma^{-L+1}(x^b)$ are contained in the same link, $\sigma^{-L+1}(A)$ is a basic quasi-$p$-symmetric arc.
Let $\ell^n$ be the link containing $u^{2\hat a + \hat n}$, and let $A_{2a + n}$ be the arc
component of $\ell^n$ containing $u^{2\hat a + \hat n}$. Since $L_p(u^{2\hat a+2\hat n}) = 1$, by Remark
\ref{rem:level-zero}, $(u^{2\hat a + \hat n}, u^{2\hat a + 2\hat n}) \setminus A_{2a + n}$ can contain at most
one $p$-point and its $p$-level is 0. Therefore $(u^{2\hat a}, u^{2\hat a + \hat n}) \setminus A_{2a + n}$
can also contain at most one $p$-point and its $p$-level is 0. By Proposition~\ref{lem:symarc},
$[u^{-\hat n}, u^{2\hat a+\hat n}]$ is either a $p$-symmetric arc, or a basic quasi-$p$-symmetric arc, see
Figure~\ref{fig:not-link-sym}. Let us denote by $A_n$ the arc-component of $\ell^n$ containing $u^{-\hat n}$.
Then $(u^{-\hat n}, u^0) \setminus A_n$ also does not contain any $p$-point with non-zero $p$-level.

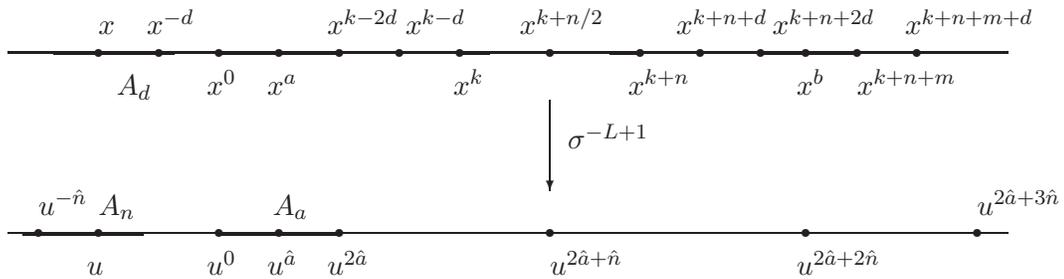
\begin{figure}[ht]
\unitlength=8mm
\begin{picture}(15,7)(0,1)
\put(-0.5,6){\line(1,0){16.6}}
\put(1, 6){\circle*{0.15}}\put(1, 6.3){\small $x$}
\put(2, 6){\circle*{0.15}}\put(1.8, 6.3){\small $x^{-d}$}
\put(1.3, 5.3){\small $A_d$}
\put(3, 6){\circle*{0.15}}\put(2.8, 5.3){\small $x^{0}$}
\put(4, 6){\circle*{0.15}}\put(3.8, 5.3){\small $x^{a}$}
\put(5, 6){\circle*{0.15}}\put(4.8, 6.3){\small $x^{k-2d}$}
\put(6, 6){\circle*{0.15}}\put(6.1, 6.3){\small $x^{k-d}$}
\put(7, 6){\circle*{0.15}}\put(6.9, 5.3){\small $x^{k}$}
\put(8.5, 6){\circle*{0.15}}\put(8, 6.3){\small $x^{k+n/2}$}
\put(10, 6){\circle*{0.15}}\put(9.8, 5.3){\small $x^{k+n}$}
\put(11, 6){\circle*{0.15}}\put(10.6, 6.3){\small $x^{k+n+d}$}
\put(12, 6){\circle*{0.15}}\put(12.2, 6.3){\small $x^{k+n+2d}$}
\put(12.75, 6){\circle*{0.15}}\put(12.6, 5.3){\small $x^{b}$}
\put(13.6, 6){\circle*{0.15}}\put(13.6, 5.3){\small $x^{k+n+m}$}
\put(14.6, 6){\circle*{0.15}}\put(14.4, 6.3){\small $x^{k+n+m+d}$}
\put(8.5, 5.2){\vector(0,-1){1.5}}\put(8.8, 4.4){\small $\sigma^{-L+1}$}
\put(-0.5,3){\line(1,0){16.6}}
\put(0, 3){\circle*{0.15}}\put(0, 3.3){\small $u^{-\hat n}$}
\put(1, 3.3){\small $A_n$}
\put(1, 3){\circle*{0.15}}\put(0.8, 2.3){\small $u$}
\put(3, 3){\circle*{0.15}}\put(2.8, 2.3){\small $u^{0}$}
\put(4, 3){\circle*{0.15}}\put(3.8, 2.3){\small $u^{\hat a}$}
\put(3.9, 3.3){\small $A_a$}
\put(5, 3){\circle*{0.15}}\put(4.8, 2.3){\small $u^{2\hat a}$}
\put(8.5, 3){\circle*{0.15}}\put(8.5, 2.3){\small $u^{2\hat a+\hat n}$}
\put(12.75, 3){\circle*{0.15}}\put(12.6, 2.3){\small $u^{2\hat a+ 2\hat n}$}
\put(15.6, 3){\circle*{0.15}}\put(15.6, 3.3){\small $u^{2\hat a+ 3\hat n}$}
\thicklines
\put(0.25,5.98){\line(1,0){2}}
\put(3,5.98){\line(1,0){2}}
\put(12,5.98){\line(1,0){1.6}}
\put(-0.25,2.98){\line(1,0){2}}
\put(3,2.98){\line(1,0){2}}
\put(7, 5.98){\line(1,0){0.5}}
\put(9.5, 5.98){\line(1,0){0.5}}
\end{picture}
\caption{\label{fig:not-link-sym}
The configuration of points on $[x^{-d}, x^{k+n+m+2d}]$ and their images under $\sigma^{-L+1}$ as in $(ii)$.}
\end{figure}

Assume by contradiction that $x^{-d}$ is not the midpoint of its arc-component $A_d$. Let us denote the midpoint
of $A_d$ by $x$, and let $u := \sigma^{-L+1}(x)$. Since $L_p(x) > L_p(x^{a})$,
also $L_p(u) > L_p(u^{\hat a})$.
Let $\ell^a$ be the link which contains $u^{\hat a}$, and let $A_a$ be the arc-component of $\ell^a$ containing
$u^{\hat a}$. Then $u \in A_n$ and $[u^{-\hat n}, u^{2\hat a+\hat n}]$ is basic quasi-$p$-symmetric. But, since
$u^{2\hat a+\hat n} \in \ell^n$ and $\sigma^{L-1}(u^{2\hat a+\hat n}) = x^{k+n/2}$,
$x^{k+n/2} \in \ell^d$. Since the arc $[x, x^{k-d}]$ is quasi-$p$-symmetric,
$[x^{k-d}, x^{k+n/2}]$ is also quasi-$p$-symmetric and
$L_p(x^{a}) > L_p(x^{k-d})$ implies $L_p(x^{k-d}) > L_p(x^{k+n/2})$, a contradiction.

Let us assume now that $\sigma^{-L+1}(A)$ is contained in a single link. Since $L_p(u) > L_p(u^{\hat a})$
and $L_p(u^0) = 1$, we have $\pi_p([u, u^0]) \subset \pi_p([u^{\hat a}, u^0])$. Then
$\sigma^{L-1}([u^{\hat a}, u^0]) \subset \ell$ implies $\sigma^{L-1}([u, u^{\hat a}]) \subset \ell$ and hence
$[x^{-d}, x^{k-d}] \subset \ell$, a contradiction.

These two contradictions prove the claim.

In the same way we can prove that $x^{k+n+m+d}$ is the midpoint of its arc-component, and by Proposition~\ref{lem:until_lm} the arc $[u^{2\hat a+\hat n}, u^{2\hat a+ 3\hat n}]$ is either $p$-symmetric,
or quasi-$p$-symmetric.

So we have proved that the arcs $[u^{-\hat n}, u^{2\hat a+\hat n}]$ and
$[u^{2\hat a+\hat n}, u^{2\hat a+ 3\hat n}]$ are both either $p$-symmetric, or quasi-$p$-symmetric. Since
$[x^{a}, x^{b}] = \sigma^{L-1}([u^{\hat a}, u^{2\hat a+ 2\hat n}])$ is
quasi-$p$-symmetric, the arcs
$\sigma^{L-1}([u^{-\hat n}, u^{2\hat a+\hat n}])$ and
$\sigma^{L-1}([u^{2\hat a+\hat n}, u^{2\hat a+ 3\hat n}])$ are both either $p$-symmetric, or
quasi-$p$-symmetric. This implies that $[x^{-2d-n/2}, x^{k+n/2}]$ and $[x^{k+n/2}, x^{k+n+m+2d+n/2}]$ are
contained in the decreasing quasi-$p$-symmetric arc
$[x^{-2d-n/2}, x^{k+n+m+2d+n/2}]$ containing $A$.
\end{proof}
\begin{ex}
(Example for $(ii)$ of Proposition~\ref{lem:not-link-sym}.)
Let us consider the Fibonacci map and the corresponding inverse limit space.
The arc-component $\C0$ contains an arc $A = [x^0, x^{77}]$ with the following folding pattern:
$$
1 \ 9_1 \underbrace{{\bf 1}_2 \, {\bf 22 \ 1 \ 56 \ 1 \ 22 \ 1} \ 9 \ {\bf 1}}_{\textrm{basic}} \ 4 \ 1 \ 0 \ 2 \ 0 \ 1 \ 0 \ 3 \ 0 \ 1 \ 6
$$
$$
\overbrace{\underbrace{{\bf 1}_{22} {\bf 14 \ 1 \ 35 \ 1 \ 14 \ 1} \ 6 \ {\bf 1}}_{\textrm{basic}} \ 0 \ 3 \ 0 \ 1 \ 0 \ 2 \ 0 \ 1 \ 4 \ {\bf 1} \ 9 \ {\bf 1 \ 22 \ 1} \ 9 \ {\bf 1} \ 4 \ 1 \ 0 \ 2 \ 0 \ 1 \ 0 \ 3 \ 0 \ \underbrace{{\bf 1} \ 6 \ {\bf 1 \ 14 \ 1}_{60}}_{\textrm{basic}}}^{\textrm{quasi-$p$-symmetric}}
$$

$$
6 \ 1 \ 0 \ 3 \ 0 \ 1 \ 0 \ 2 \ 0 \ 1 \ 4 \ \underbrace{1 \ 9 \ {\bf 1}_{74}}_{\textrm{sym}} 4_{75} 1 \ 0
$$
We can choose a chain $\chain_p$ such that $p$-points with $p$-levels 1, 14, 22, 35 and 56 belong
to the same link. Then the arc $[x^{22}, x^{60}]$
is quasi-$p$-symmetric, and it is not basic. The arc $\sigma^{-13}([x^{22}, x^{60}])$ is basic
quasi-$p$-symmetric with the folding pattern $1 \ 22 \ 1 \ 9 \ 1$.
So we can apply Propositions~\ref{lem:symarc} and~\ref{lem:until_lm}
as in the above proof. The arc $[x^2, x^{74}]$ is decreasing
quasi-$p$-symmetric. Note that the arc $[x^1, x^{75}]$ is not $p$-link-symmetric.
\end{ex}
\begin{defi}\label{def:maximaldecreasing}
An arc $A = [x, y]$ is called \emph{maximal decreasing (basic) quasi-$p$-symmetric} if it is
decreasing (basic) quasi-$p$-symmetric and there is no decreasing (basic) quasi-$p$-symmetric arc
$B \supset A$ that consists of more (basic) quasi-$p$-symmetric arcs than $A$.

Similarly we define a \emph{maximal increasing (basic) quasi-$p$-symmetric} arc.
\end{defi}

\begin{remark}\label{rem:quasi-sym}
(a) Propositions~\ref{lem:symarc} and~\ref{lem:until_lm} imply that $A = [x, y]$ is a
maximal decreasing basic quasi-$p$-symmetric arc if and only if $A$ is a decreasing basic
quasi-$p$-symmetric and for $x = x^0, x^1, \dots , x^{n-1}, x^n = y$ which satisfy (i) of
Definition~\ref{def:decreasing}, there exists a point $x^{-1}$ such that $[x^{-1}, x^1]$ is $p$-symmetric
with midpoint $x^0$ and $x^n$ is not a $p$-point. The arc $[x^{-1}, x^n]$ we call the \emph{extended}
maximal decreasing basic quasi-$p$-symmetric arc. The points $x^{-1}$,
$x = x^0, x^1, \dots , x^{n-1}, x^n = y$ we call the \emph{nodes} of $[x^{-1}, x^n]$.

The analogous statement holds if $A$ is a maximal increasing basic quasi-$p$-symmetric arc:
If $A = [x^0, x^{n+1}]$ is an extended maximal increasing basic
quasi-$p$ symmetric arc, then $x^0$ is not a $p$-point, $L_p(x^n) > L_p(z)$ for every $p$-point $z \in A$,
$z \ne x^n$, and $L_p(x^{n-1}) = L_p(x^{n+1})$.

(b) Let $A = [x^0, x^{n+1}]$ be an extended maximal increasing basic
quasi-$p$ symmetric arc. If there exists an additional $p$-point
$x^{n+2}$ such that the arc $[x^n, x^{n+2}]$ is quasi-$p$ symmetric with midpoint $x^{n+1}$,
Propositions~\ref{lem:symarc} and~\ref{lem:until_lm} imply that $A$ is contained in an $p$-symmetric arc
$B = [x^0, x^{2n}]$ where the arc $[x^{n-1}, x^{2n}]$ is an extended maximal decreasing basic
quasi-$p$-symmetric arc.

The analogous statement holds if $A$ is a maximal decreasing basic quasi-$p$-symmetric arc.
\end{remark}

\begin{lemma}\label{lem:max-quasi-sym}
Every (basic) quasi-$p$-symmetric arc $A$ can be extended to a maximal decreasing/increasing (basic)
quasi-$p$-symmetric arc $B \supset A$.
\end{lemma}

\begin{proof}
We take the largest decreasing (basic) quasi-$p$-symmetric arc $B$ containing $A$. The only thing to prove
is that there really is a largest $B$. If this were not the case, then there would be an infinite sequence
$(x_i)_{i \ge 0}$ with $x_0 \in \partial A$, $L_p(x_i) < L_p(x_{i+1})$ and $[x_i, x_{i+2}]$ is a (basic)
quasi-$p$-symmetric arc for all $i \ge 0$. By the definition of (basic) quasi-$p$-symmetric arc, there are two
links $\ell$ and $\hat \ell$ containing $x_i$ for all even $i$ and odd $i$ respectively. (Note that
$\ell = \hat \ell$ is possible.)
By Lemma~\ref{lem:maxL} for the basic case, the $p$-points in
$\bigcup_{i \ge 0} [x_0, x_i] \setminus (\ell \cup \hat \ell)$ can only have finitely many different
$p$-levels. By the construction in the proof of Proposition~\ref{lem:not-link-sym} $(ii)$,
the same conclusion is true for the non-basic case as well.
But $\bigcup_{i \ge 0} [x_0, x_i]$ is a ray, and contains $p$-points of all
(sufficiently high) $p$-levels. Since the closure of $\pi_p(\{ x : L_p(x) \ge N\})$ contains $\omega(c)$
for all $N$, this set is not contained in the $\pi_p$-images of the two links $\ell$ and $\hat \ell$
only. So we have a contradiction.
\end{proof}

\begin{proposition}\label{thm:linksym}
Let $A$ be a $p$-link-symmetric arc with midpoint $m$ and $\partial A = \{ x, y \} \subset E_p$. Then
$A$ is $p$-symmetric, or is contained in an extended maximal decreasing/increasing (basic)
quasi-$p$-symmetric arc, or is contained in a $p$-symmetric arc which is the concatenation of two arcs, one
of which is a maximal increasing (basic) quasi-$p$-symmetric arc, and the other one is a maximal decreasing
(basic) quasi-$p$-symmetric arc.
\end{proposition}

\begin{proof}
Let $A \cap E_p = \{{x}^{-k'}, \dots , {x}^{-1}, {x}^{0}, {x}^{1}, \dots {x}^{k}\}$ and ${x}^{0} = m$.
Without loss of generality we assume that ${x}^{-k'}$ and ${x}^{k}$ are the midpoints of the link-tips
of $A$. If $L_{p}(x^{-i}) = L_{p}(x^{i})$, for $i = 1, \dots , \min \{ k', k \}$, then the arc $A$ is
either $p$-symmetric, or (basic) quasi-$p$-symmetric. Hence in this case the theorem is true.

Let us assume that there exists $j < \min \{ k', k \}$ such that $L_{p}(x^{-i}) = L_{p}(x^{i})$, for
$i = 1, \dots , j-1$, and $L_{p}(x^{-j}) \ne L_{p}(x^{j})$. The arc $[x^{-j}, x^{j}]$ is (basic)
quasi-$p$-symmetric and by Lemma~\ref{lem:max-quasi-sym} and Remark~\ref{rem:quasi-sym}, there exists an
extended maximal decreasing/increasing (basic) quasi-$p$-symmetric arc which contains $[x^{-j}, x^{j}]$.
Hence in this case the theorem is also true.
\end{proof}

\medskip
\noindent
Faculty of Mathematics, University of Vienna,\\
Nordbergstra{\ss}e 15/Oskar Morgensternplatz 1, A-1090 Vienna, Austria\\
\texttt{henk.bruin@univie.ac.at}\\
\texttt{http://www.mat.univie.ac.at/}$\sim$\texttt{bruin}

\medskip
\noindent
Department of Mathematics, University of Zagreb.\\
Bijeni\v cka 30, 10 000 Zagreb, Croatia\\
\texttt{sonja@math.hr}\\
\texttt{http://www.math.hr/}$\sim$\texttt{sonja}


\begin{thebibliography}{99}

\bibitem{ALM} A.\ Avila, M.\ Lyubich, W.\ de Melo, {\em Regular or stochastic dynamics in real analytic families of unimodal maps}, Invent.\ Math.\ {\bf 154} (2003) 451--550.

\bibitem{BBD} M.\ Barge, K.\ Brucks, B.\ Diamond, {\em Self-similarity in inverse limit spaces of the tent family}, Proc.\ Amer.\ Math.\ Soc.\ {\bf 124} (1996) 3563--3570.

\bibitem{BBS} M.\ Barge, H.\ Bruin, S.\ \v Stimac, {\em The Ingram Conjecture}, Geometry and Topology
{\bf 16} (2012) 2481--2516.

\bibitem{BC} M.\ Benedicks, L.\ Carleson, {\em The dynamics of the H\'enon map}, Ann.\ of Math.\ {\bf 133} (1991) 73--169.

\bibitem{Betal} L.\ Block, S.\ Jakimovik, J.\ Keesling, L.\ Kailhofer, {\em On the classification of inverse limits of tent maps}, Fund.\ Math.\ {\bf 187} (2005) 171--192.

\bibitem{BB} K.\ Brucks, H.\ Bruin, {\em Subcontinua of inverse limit spaces of unimodal maps}, Fund.\ Math.\ {\bf 160} (1999) 219--246.

\bibitem{BBbook} K.\ Brucks, H.\ Bruin, {\em Topics in one-dimensional dynamics}, London Math.\ Soc.\ Student texts {\bf 62} Cambridge University Press 2004.


\bibitem{Bknea} H.\ Bruin, {\em Combinatorics of the kneading map}, Int.\ Jour.\ of Bifur.\ and Chaos {\bf 5}  (1995) 1339--1349.

\bibitem{BTams} H.\ Bruin, {\em Topological conditions for the existence of absorbing Cantor sets}, Trans.\ Amer.\ Math.\ Soc.\ {\bf 350} (1998) 2229--2263.

\bibitem{BNonlin} H.\ Bruin, {\em Quasi-symmetry of conjugacies between interval maps}, Nonlinearity {\bf 9} (1996) 1191--1207.

\bibitem{Bsubcontinua} H.\ Bruin, {\em Subcontinua of Fibonacci-like unimodal inverse limit spaces}, Topology Proceedings {\bf 31} (2007) 37--50.

\bibitem{Btree} H.\ Bruin, {\em (Non)invertibility of Fibonacci-like unimodal maps restricted to their critical omega-limit sets}, Topology Proceedings {\bf 37} (2011) 459--480.

\bibitem{BKNS} H.\ Bruin, G.\ Keller, T.\ Nowicki, S.\ van Strien, {\em Wild Cantor Attractors exist}, Ann.\ of Math.\ (2) {\bf 143} (1996) 97--130.

\bibitem{BKS}  H.\ Bruin, G.\ Keller, M.\ St.-Pierre, {\em Adding machines and wild attractors}, Ergod.\ Th.\ \& Dynam.\ Sys.\ {\bf 17} (1997) 1267--1287.

\bibitem{Hof1} F.\ Hofbauer, {\em The topological entropy of a transformation $x \mapsto ax(1-x)$}, Monath.\ Math.\ {\bf 90} (1980) 117--141.

\bibitem{hofkel0} F.\ Hofbauer, G.\ Keller, {\em Some remarks on recent results about S-unimodal maps,} Ann.\ Inst.\ Henri Poincar\'e, Physique th\'eorique, {\bf 53} (1990) 413--425.


\bibitem{Kail2} L.\ Kailhofer, {\em A classification of inverse limit spaces of tent maps with periodic critical points}, Fund.\ Math.\ {\bf 177} (2003) 95--120.

\bibitem{KN} G.\ Keller, T.\ Nowicki, {\em Fibonacci maps re(a$\ell$)-visited}, Ergod.\ Th.\ \&  Dynam.\ Sys.\ {\bf 15} (1995) 99--120.

\bibitem{Lyu} M.\ Lyubich, {\em Combinatorics, geometry and attractors of quasi-quadratic maps}, Ann.\ of Math.\ (2) {\bf 140} (1994) 347--404.

\bibitem{LM} M.\ Lyubich, J.\ Milnor, {\em The Fibonacci unimodal map}, J.\ Amer.\ Math.\ Soc.\ {\bf 6} (1993) 425--457.

\bibitem{Mil} J.\ Milnor, {\em On the concept of attractor}, Commun.\ Math.\ Phys.\ {\bf 99} (1985) 177--195.


\bibitem{RS} B.\ Raines, S.\ \v{S}timac, {\em A classification of inverse limit spacers of tent maps with non-recurrent critical point}, Algebraic and Geometric Topology {\bf 9} (2009) 1049--1088.

\bibitem{Stim2} S.\ \v{S}timac, {\em Structure of inverse limit spaces of tent maps with finite critical orbit}, Fund.\ Math.\ 191 (2006) 125--150.

\bibitem{Stim} S.\ \v{S}timac, {\em A classification of inverse limit spaces of tent maps with finite critical orbit}, Topology Appl.\ {\bf 154} (2007) 2265--2281.

\end{thebibliography}
\end{document}